%% file: AsPropHypEq-RW-6.tex
\documentclass{birkmult}
\usepackage{amsmath,amssymb}
\usepackage{mathrsfs}
\usepackage{hyperref}
\usepackage{graphicx,color}
 \newtheorem{thm}{Theorem}[section]
 \newtheorem{cor}[thm]{Corollary}
 \newtheorem{lem}[thm]{Lemma}
 \newtheorem{prop}[thm]{Proposition}
 \theoremstyle{definition}
 
 \theoremstyle{remark}

 \numberwithin{equation}{section}

\DeclareMathOperator{\supp}{supp}
\def\D{\mathrm D}
\let\Re\relax\let\Im\relax
\DeclareMathOperator{\Re}{Re}
\DeclareMathOperator{\Im}{Im}
\DeclareMathOperator{\diag}{diag}
\DeclareMathOperator{\ran}{range}
\DeclareMathOperator{\dist}{dist}
\DeclareMathOperator{\rank}{rank}

\def\R{\mathbb R}
\def\Rn{{{\mathbb R}^n}}
\def\va{\varphi}

\def\jp#1{{\left\langle{#1}\right\rangle}}

\begin{document}
%
%
%
%
%
%
%
%
%
\title[Asymptotic Properties of Hyperbolic PDEs]
 {Asymptotic Behaviour of Solutions to Hyperbolic Partial Differential Equations}

\author[M.~Ruzhansky]{Michael Ruzhansky}
\address{%
Department of Mathematics\\
 Imperial College London\\
 180 Queen's Gate\\
 London SW7 2AZ\\
 United Kingdom}
 \email{m.ruzhansky@imperial.ac.uk}

\author[J.~Wirth]{Jens Wirth}
\address{%
Institut f\"ur Analysis, Dynamik und Modellierung\\ FB Mathematik\\ Universit\"at Stuttgart \\Pfaffenwaldring 57\\ 70569 Stuttgart\\ Germany}
\email{jens.wirth@iadm.uni-stuttgart.de}

\thanks{
Part of the work presented here was supported by EPSRC grant EP/E062873/1. 
The first author is supported by the EPSRC Leadership Fellowship EP/G007233/1.
The second author is grateful to travel support from DAAD, grant 50022370, for visiting London in December 2010 and February 2011.
The selection of materials is based on the mini-courses taught by the first author at the CRM, Barcelona, and by the second 
author at Aalto University, Helsinki, both in 2011.
}



\date{July 1, 2011}

\begin{abstract}
These notes provide an introduction and a survey on recent results about the 
long-time behaviour of solutions
to hyperbolic partial differential equations with time-dependent coefficients. 
Emphasis is given also to questions about the sharpness of estimates. 
\end{abstract}

\maketitle

\tableofcontents


%

\section{Introduction}
In this introductory section we want to recall some well-known estimates for the wave equation
and their relation to Fourier integrals and restriction theorems. This shall serve as motivation to 
study similar estimates for perturbations of the wave equation by lower order terms and outline strategies to be pursued.
\subsection{Energy and dispersive estimates} 
We will recall some well-known properties of solutions to the Cauchy problem for the wave equation
\begin{equation}\label{eq:1:CP-free}
   u_{tt} - \Delta u = 0,\qquad u(0,\cdot)=u_0,\quad u_t(0,\cdot)=u_1
\end{equation}
in $\mathbb R\times\mathbb R^n$. We assume for simplicity that data belong to $C_0^\infty(\mathbb R^n)$. A simple integration by parts argument allows to show that solutions possess 
a cone of dependence property and are therefore also compactly supported for all fixed times $t$.
We will not make this argument rigorous, but point out that by the same reason 
the total energy defined as
\begin{equation}
  \mathbb E(u;t) = \frac 12 \int (|\nabla u|^2 + |u_t|^2) \mathrm d x
\end{equation}
is preserved,
\begin{equation}
   \frac{\mathrm d}{\mathrm dt} \mathbb E(u;t) =\Re  \int (\overline u_t u_{tt} + \nabla \overline u\cdot \nabla u_t)\mathrm d x = \Re \int \overline u_t (u_{tt}-\Delta u)\mathrm d x =  0.
\end{equation} 

This preserved energy is spread out (uniformly) over a region increasing in time.
Dispersive estimates due to von Wahl \cite{vWahl}, Brenner  \cite{Brenner:1975}, Pecher \cite{Pecher:1976} and more generally
Strichartz estimates \cite{Strichartz:1970} describe this effect. See also Keel--Tao \cite{Keel:1998}
for the precise interrelation between Strichartz and dispersive estimates.

We denote by $L^p_r(\mathbb R^n)$ the Bessel potential space of order $r$ over $L^p(\mathbb R^n)$ defined in terms of the Fourier transform as
\begin{equation}
  L^p_r(\mathbb R^n)=\langle\D_x\rangle^{-r}L^p(\mathbb R^n) = \mathscr F^{-1} [\langle\xi\rangle^{-r} \mathscr FL^p(\mathbb R^n)] .
 \end{equation}
 Then the following statement contains the essence of the above cited papers. Here and later,
we denote $\jp{\xi}=(1+|\xi|^2)^{1/2}.$
 
\begin{thm}\label{thm:1:dispEst}
Solutions to \eqref{eq:1:CP-free} satisfy the dispersive estimates
\begin{align}
  \| u(t,\cdot)\|_{L^\infty} &\le C_r t^{-\frac{n-1}2} \left(\|u_0\|_{L^1_r} + \|u_1\|_{L^1_{r-1}}\right),\\
  \| \nabla u(t,\cdot)\|_{L^\infty} + \|u_t(t,\cdot)\|_{L^\infty} &\le C_r t^{-\frac{n-1}2} \left(\|u_0\|_{L^1_{r+1}} + \|u_1\|_{L^1_{r}}\right),
\end{align}
for $r>\frac{n+1}2$.
\end{thm}

Dispersive estimates are much harder to prove than energy estimates as they encode structural information about the representation of solutions. We will collect some of the crucial ingredients of their proof in this introductory section.

\subsection{Equations with constant coefficients}\label{sec:1:constCoeff} Of particular importance for us are equations 
with constant coefficients as they can be solved explicitly in terms of the Fourier transform and serve
as important model examples. We keep it as simply and sketchy as possible to not extend the exposition too much. Later we will discuss variable coefficient versions of the wave equation
\begin{equation}\label{eq:1:wave}
   u_{tt} - \Delta u = 0,
\end{equation}
the damped wave equation (or dissipative wave equation)
\begin{equation}\label{eq:1:damped-wave}
   u_{tt} - \Delta u + u_t = 0,
\end{equation}
and the Klein--Gordon equation
\begin{equation}\label{eq:1:KGeq}
   u_{tt} - \Delta u +u = 0.
\end{equation}
Natural questions to ask for these problems are:
\begin{enumerate}
\item What are reasonable energies to estimate? While for the damped wave equation
$\mathbb E(u;t)$ is a good candidate, we need to include  $u$ itself into the energy term for the Klein--Gordon equation. Then the Klein--Gordon energy is preserved, while for the damped wave
equation we obtain $\partial_t\mathbb E(u;t)\le 0$ and the (wave) energy decreases.
\item What are sharp (decay) rates for the energy? This is of interest for the damped wave equation
and the answer is by no means trivial. It heavily depends on the assumptions made for data, see e.g. the work of Ikehata \cite{Ikehata:2003b}, \cite{Ikehata:2003e}, \cite{Ikehata:2004} for some examples of this. The obvious estimate
\begin{equation}
  \mathbb E(u;t) \le \mathbb E(u;0)
\end{equation}
is sharp for the damped wave equation with arbitrary data from the energy space. On the other
hand there is the estimate of Matsumura \cite{Matsumura:1977}
\begin{equation}
   \mathbb E(u;t) \lesssim (1+t)^{-1-\frac n2} \left( \|u_0\|^2_{H^1\cap L^1} + \|u_1\|^2_{L^2\cap L^1}\right).
\end{equation}
\item How do dispersive estimates for the above models look like? Here we will just point out that the decay rate 
for the Klein--Gordon model is $t^{-n/2}$ in contrast to the rate $t^{-(n-1)/2}$ for the free
wave equation. There exist uniform estimates for solutions to the damped wave equation, which happen to
give the same decay rate of  $t^{-n/2}$.
\item Are there other asymptotic representations for solutions? As example we refer to Nishihara \cite{Nishihara:2003} for an asymptotic description of solutions to the damped wave equation in terms of solutions of an associated heat equation and some exponentially decaying free waves. 
\end{enumerate}
We will come back to some of these questions for more general hyperbolic models in the subsequent sections.
An extensive study of higher order hyperbolic equations with constant coefficients and their
dispersive estimates was done in Ruzhansky--Smith \cite{Ruzhansky:2010}, we restrict ourselves here in this informal introductory part to the wave equation \eqref{eq:1:wave}. Solutions in $C^2(\mathbb R; \mathscr S'(\mathbb R^n))$ can be studied in terms of the partial Fourier transform with respect to the $x$-variable. 
This gives the parameter-dependent ordinary differential equation
\begin{equation}
   \widehat u_{tt} + |\xi|^2 \widehat u=0
\end{equation}
with solutions explicitly given as
\begin{equation}\label{eq:1:CP-free-sol}
  \widehat u(t,\xi) = \cos(t|\xi|) \widehat u_0(\xi) + \frac{\sin(t|\xi|)}{|\xi|} \widehat u_1(\xi).
\end{equation}
Estimates in $L^2$-scale can be obtained immediately from Plancherel identity. We will
only give the following lemma stating higher order energy estimates for the free
wave equation. The proof is straightforward from \eqref{eq:1:CP-free-sol}.
\begin{lem}
The solutions to the Cauchy problem \eqref{eq:1:CP-free} satisfy the a priori estimate
\begin{equation}\label{eq:1:wave-sol-Est}
   \| u(t,\cdot)\|_{L^2} \le \|u_0\|_{L^2} + t \|u_1\|_{L^2}
\end{equation}
together with 
\begin{equation}
   \| \D_t^k\D_x^\alpha u(t,\cdot)\|_{L^2} \le C_{k,\alpha} \left(\|u_0\|_{H^{k+|\alpha|}} +  \|u_1\|_{H^{k+|\alpha|-1}}\right)
\end{equation}
for all $k\in\mathbb N_0$ and $\alpha\in\mathbb N^n_0$, $|\alpha|\ge1$.
\end{lem}

The proof of the dispersive estimates of Theorem~\ref{thm:1:dispEst} are based on the representation of solutions as sum of Fourier integrals
\begin{equation}
\int     \mathrm e^{\mathrm i(x\cdot\xi \pm t|\xi|)} a_{j,\pm}(t,\xi) u_j(\xi) \mathrm d\xi 
\end{equation}
with amplitudes $a_{j,\pm}(t,\xi)$ supported in $t|\xi|\gtrsim1$ and an (easy to obtain) better estimate for the remainder term localised to  $t|\xi|\lesssim1$. A dyadic decomposition of 
frequency space is used to reduce this to integrals over spherical shells in $\xi$ and they 
are treated by means of usually stationary phase estimates. 

\subsection{Stationary phase estimates}\label{sec:1:StatPhas}
 We will conclude this introductory
section by reviewing some estimates for Fourier transforms of surface carried measures
and relate them to estimates for Fourier integrals. Such estimates have a rather long history and 
go back to the original work of van der Corput \cite{vdCorput:1935}, \cite{vdCorput:1936}, Hlawka
\cite{Hlawka:1949}, \cite{Hlawka:1949b}, Randol \cite{Randol:1969} and in particular for applications in hyperbolic partial differential equations Strichartz \cite{Strichartz:1970}, \cite{Strichartz:1977} and Littman \cite{Littman:1963}. 
Later, Sugimoto \cite{Sugimoto:1994}, \cite{Sugimoto:1996}, established further decay rates depending
on the geometry of the level sets of the phase, in the analytic case, introducing notions of convex
and non-convex indices in this context.
We will rely on further extension of these by Ruzhansky \cite{Ruzhansky:2009}, \cite{Ruzhansky:2012},
allowing phases and amplitudes of limited regularity and depending on parameters, with uniform
estimates with the respect to parameter, bridging the gap between the van der Corput lemma and the
stationary phase method.

Let $\Sigma$ be a smooth closed hypersurface embedded in $\mathbb R^n$. We assume first
that $\Sigma$ encloses a convex domain of $\mathbb R^n$, which is in particular the case if the Gaussian curvature of
$\Sigma$ is non-vanishing. Let further $f\in C^\infty(\Sigma)$ be a smooth function defined on the surface $\Sigma$. In a first step we are interested in decay properties of the inverse Fourier transform
\begin{equation}\label{eq:1:FTSCM}
\check f(x) =   \int_\Sigma \mathrm e^{\mathrm i x\cdot \xi} f(\xi) \mathrm d\xi,
\end{equation}
the integral taken with respect to the surface measure induced from the ambient
Lebesgue measure. Before stating the result, we introduce the contact index $\gamma(\Sigma)$
of the surface $\Sigma$ to be the maximal order of contact between $\Sigma$ and its tangent
lines.
For a 2-plane 
$H$ containing the normal of the surface $\Sigma$ at $p\in\Sigma$ we define
 $\gamma(\Sigma; p,H)$ to be
 the order of contact between the tangent $\mathrm T_p\Sigma\cap H$
and the curve $\Sigma\cap H$ in the point $p$.
Consequently, we set 
\begin{equation}\label{eq:1:SugInd}
   \gamma(\Sigma) = \max_{p\in\Sigma} \max_{H : \mathrm N_p\Sigma\subset H } \gamma(\Sigma; p, H).
\end{equation}
Then the following result is valid.


\begin{lem}\label{lem:1:sug1}
Assume $\Sigma$ is convex. Then the estimate
\begin{equation}
  |\check f(x)| \lesssim \langle x\rangle^{-\frac{n-1}{\gamma(\Sigma)}} \| f\|_{C^k(\Sigma)}
\end{equation}
is valid for all $f\in C^k(\Sigma)$, $k\ge \frac{n-1}{\gamma(\Sigma)}+1$.
\end{lem}
%

This result is closely related to the following multiplier theorem
which we formulate in the form due to Ruzhansky \cite{Ruzhansky:2009, Ruzhansky:2012}
(see also Ruzhansky--Smith \cite{Ruzhansky:2010} and  Sugimoto \cite{Sugimoto:1994}). 
Let $\chi\in C^\infty(\mathbb R^n)$ be an excision function, i.e., we assume it is equal to $1$ for large $|\xi|$ and vanishes near the origin. We then consider the operator
\begin{equation}\label{eq:1:Tdef}
  T_{\phi,r} : u(x) \mapsto \int_{\mathbb R^n} \mathrm e^{\mathrm i(x\cdot \xi + \phi(\xi))} |\xi|^{-r} \chi(\xi) \widehat u(\xi)
   \mathrm d\xi
\end{equation}
for  a given real-valued
phase function $\phi(\xi)\in C^\infty(\mathbb R^n\setminus\{0\})$ being positively homogeneous
$\phi(\rho\xi) = \rho\phi(\xi)$, $\rho>0$. The $L^p$--$L^{p'}$-boundedness of such an operator is
related to the geometry of the Fresnel surface / level set
\begin{equation}\label{eq:1:Fresnel}
    \Sigma = \{ \xi \in \mathbb R^n : \phi(\xi)=1 \}.
\end{equation}
It is sufficient to prove the limit case for $p=1$, boundedness between dual Lebesgue spaces
follows by interpolation with the (obvious) boundedness in $L^2(\mathbb R^n)$. First we recall the definition of  Besov spaces. Let $\phi\in C_0^\infty(\mathbb R_+)$ define a Littlewood--Paley decomposition, i.e., be such that $\sum_{j\in\mathbb Z} \phi(2^{-j}\tau) =1$ for any $\tau\ne0$. Then the homogeneous Besov norm of regularity $s\ge0$ and with Lebesgue-indices $p, q\in[1,\infty]$ is defined by  
\begin{equation}\label{eq:1:homBesov}
    \| v \|_{\dot B^s_{p,q}}^q = \sum_{j\in\mathbb Z} \big( 2^{js} \| \mathscr F^{-1} \phi(2^{-j}|\xi|) \widehat v(\xi) \|_{L^p}\big)^q,
\end{equation}
while the Besov space $B^s_{p,q}(\mathbb R^n)$ contains all functions with
\begin{equation}\label{eq:1:Besov}
    \| v \|_{B^s_{p,q}}^q =  \| \mathscr F^{-1} \psi(\xi) \widehat v(\xi) \|_{L^p}^q +
    \sum_{j=1}^\infty \big( 2^{js} \| \mathscr F^{-1} \phi(2^{-j}|\xi|) \widehat v(\xi) \|_{L^p}\big)^q<\infty,
\end{equation}
where $\psi(s) = 1-\sum_{j=1}^\infty \phi(2^{-j}s)$. 
It follows $B^s_{p,2}(\mathbb R^n) = \dot B_{p,2}^s(\mathbb R^n) \cap L^p(\mathbb R^n)$
and there are continuous embeddings 
$L^p(\mathbb R^n)\hookrightarrow B^0_{p,2}(\mathbb R^n)$ for $1<p\le 2$ and similarly
$B^0_{p',2}(\mathbb R^n)\hookrightarrow L^{p'}(\mathbb R^n)$ for 
$2\le p'<\infty$. The opposite embeddings require to pay a small amount of regularity (see Runst--Sickel  \cite{RuSi:1996} for more details).

\begin{lem}\label{lem:1:sug1A}
Assume that $\Sigma$ defined by \eqref{eq:1:Fresnel} is convex.
Then $T_{\phi,r}$ is a bounded operator mapping $ \dot B^0_{1,2}(\mathbb R^n)\to L^\infty(\mathbb R^n)$
for all $r\ge n-\frac{n-1}{\gamma(\Sigma)}$.
\end{lem}

The statement of this lemma immediately implies dispersive type estimates for Fourier
integral representations of the form
\begin{equation}\label{eq:1:FIO-rep}
    u(t,x) = (2\pi)^{-n} \int \mathrm e^{\mathrm i(x\cdot\xi + t\phi(\xi))} \widehat u_0(\xi) \mathrm d\xi.
\end{equation}
Indeed, by rescaling $\eta=t\xi$, $y=x/t$, we see that
\begin{equation}\label{eq:1:dispEst}
    \| u(t,\cdot) \|_{L^\infty} \lesssim t^{-\frac{n-1}{\gamma(\Sigma)}  } \| u_0 \|_{\dot B^r_{1,2}}
\end{equation} 
for $r= n-\frac{n-1}{\gamma(\Sigma)}$. The above estimate \eqref{eq:1:dispEst} is of interest
to us because \eqref{eq:1:FIO-rep} solves the Cauchy problem
\begin{equation}
   u_{tt} + \phi^2(\D_x) u = 0,\qquad u(0,\cdot) = u_0,\quad u_t(0,\cdot) = 0 
\end{equation}
generalising \eqref{eq:1:CP-free}. Here and later on we use the notation
$\phi(\D_x) : u \mapsto \mathscr F^{-1} [\phi(\xi) \widehat u(\xi)]$ for Fourier multipliers.

In general there is no reason for the Fresnel surface of a phase to be convex. Of particular 
importance for the study of hyperbolic systems arising in crystal acoustics or elasticity theory or for higher order scalar equations are generalisations of the Lemmata~\ref{lem:1:sug1} and \ref{lem:1:sug1A}
to non-convex surfaces.  To consider this situation, we define the
non-convex contact index
\begin{equation}
   \gamma_0(\Sigma) = \max_{p\in\Sigma} \min_{H : \mathrm N_p\Sigma\subset H } \gamma(\Sigma; p, H)
\end{equation}
as maximal order of minimal contact. The replacement to Lemma~\ref{lem:1:sug1} is
\begin{lem}\label{lem:1:sug2}
For general (non-convex) surfaces $\Sigma$ the estimate
\begin{equation}
  |\check f(x)| \lesssim \langle x\rangle^{-\frac{1}{\gamma_0(\Sigma)}} \| f\|_{C^1(\Sigma)}
\end{equation}
is valid for all $f\in C^1(\Sigma)$.
\end{lem}

Again we obtain a multiplier theorem for the operator $T_{\phi,r}$ defined in
\eqref{eq:1:Tdef}.
\begin{lem}\label{lem:1:sug2A}
The operator  $T_{\phi,r}$ is bounded $ \dot B^0_{1,2}(\mathbb R^n)\to L^\infty(\mathbb R^n)$
for $r\ge n-\frac{1}{\gamma_0(\Sigma)}$.
\end{lem}

Later on in Section~\ref{sec:3:dispersive} we will generalise the above estimates to include
phase functions depending on both $t$ and $x$-variables suitable for the treatment of more
general hyperbolic equations.

\section{Equations with constant coefficients}\label{sec:cc}

Before we pursue the analysis of equations and systems with time-dependent coefficients
it is instructive to understand what happens in the case of equations with
constant coefficients. One of the very helpful observations available in this case
is that after a Fourier transform in the spatial variable $x$ we obtain an ordinary differential
equation with constant coefficients which can be solved almost explicitly
once we know its characteristics. This works well for frequencies where the
characteristics are simple. If they become multiple, the representation breaks
down and other methods are required. In the presentation of this part we
follow \cite{Ruzhansky:2010} to which we refer for the detailed arguments and
complete proofs of the material in this section.

\subsection{Formulation of the problem}

Let us consider the Cauchy
problem for a general scalar strictly hyperbolic operator with constant coefficients,
\begin{equation}\label{EQ:CP}
\left\{\begin{aligned}& \D_t^m
u+\sum_{j=1}^{m}P_{j}(\D_x)\D_t^{m-j}u+
\sum_{l=0}^{m-1}\sum_{|\alpha|+r=l}
c_{\alpha,r}\D_x^\alpha \D_t^r u=0,\; t>0,\\
&\partial_t^l u(0,x)=f_l(x)\in C_0^{\infty}(\R^n),\quad l=0,\dots,m-1,\;
x\in\R^n\,.
\end{aligned}\right.
\end{equation}
The symbol $P_j(\xi)$ of the operator $P_j(\D_x)$ is assumed to be a  
homogeneous polynomial of order $j$, and the $c_{\alpha,r}$
are (complex) constants. 
We denote by $L(\D_t,\D_x)$ and $L_m(\D_t,\D_x)$ the operator 
in \eqref{EQ:CP} and its principal
part, respectively,
\begin{equation}\label{EQ:CPh}
\begin{aligned}  L_m(\D_t,\D_x) & =\D_t^m
+\sum_{j=1}^{m}P_{j}(\D_x)\D_t^{m-j}, \\
L(\D_t,\D_x) & = L_m(\D_t,\D_x)+
\sum_{l=0}^{m-1}\sum_{|\alpha|+r=l}
c_{\alpha,r}\D_x^\alpha \D_t^r.
\end{aligned}
\end{equation}
We denote by $\va_k(\xi)$ and $\tau_k(\xi)$, $k=1,\ldots,m$,
the characteristic roots of $L_m$ and $L$, respectively. Hence, we write
\begin{equation}
L_m(\tau,\xi)=\prod_{k=1}^m (\tau-\va_k(\xi)),\qquad
L(\tau,\xi)=\prod_{k=1}^m (\tau-\tau_k(\xi)).
\end{equation}
The strict hyperbolicity of $L(\D_t,\D_x)$ means that the functions
$\va_k(\xi)$ are real and distinct, for all $\xi\not=0$. It follows that
$\va_k\in C^\infty(\R^n\backslash 0)$ and that the $\va_k$ are
positively homogeneous of order one,
$\va_k(\lambda\xi)=\lambda \va_k(\xi)$ for all
$\lambda>0$ and $\xi\not=0$.

In the general analysis it is usually desirable to 
have conditions on the lower order terms for different
rates of decay of solutions to \eqref{EQ:CP}.
However, it is more convenient to prove the results 
making assumptions on the characteristic roots. Some 
results deducing properties of characteristics from
properties of coefficients are available, and we 
refer to \cite{Ruzhansky:2010} for some details in 
this direction.

It is natural to impose a stability condition excluding 
an exponential growth of solutions in time.
Namely, we assume  that for all $\xi\in\R^n$ we have that
\begin{equation}\label{EQ:imtau>=0}
\Im\tau_k(\xi)\ge0\quad\text{for }k=1,\dots,m\,.
\end{equation}
In fact, certain microlocal
decay estimates are possible even without this condition
if the supports of the Fourier transforms of the Cauchy data
are contained in the set where condition
\eqref{EQ:imtau>=0} holds. However, this restriction is
only technical so we may assume \eqref{EQ:imtau>=0}
without great loss of generality since otherwise
no time decay of solutions can be expected, which becomes
clear from the following representation of solutions.

\subsection{Combined estimates}

By taking the Fourier transform of \eqref{EQ:CP} with respect
to $x$, solving the resulting ordinary differential equation, and taking
the inverse Fourier transform, we obtain that the
solution to the Cauchy problem \eqref{EQ:CP}
can be written in the form
\begin{equation}
u(t,x)=\sum_{j=0}^{m-1} E_j(t) f_j(x),
\end{equation}
where the propagators $E_j(t)$ are defined by
\begin{equation}\label{EQ:soln}
E_j(t)f(x)=\int_{\Rn} \mathrm e^{\mathrm i x\cdot\xi}\Big(\sum_{k=1}^m
e^{i\tau_k(\xi)t}A_j^k(t,\xi)\Big)
\widehat{f}(\xi)\,d\xi\,,
\end{equation}
with suitable amplitudes $A_j^k(t,\xi)$. 
More precisely, for each $k$ and
$j$, the functions $A_j^k(t,\xi)$ are actually independent of $t$ 
on the set of simple characteristics 
\begin{equation}
 S_k:=\{\xi\in\R^n:\tau_k(\xi)\ne
\tau_l(\xi)\quad\forall\,l\ne k\}. 
\end{equation} 
For these $A_j^k(\xi)$,
we have 
\begin{lem}\label{LEM:ordercoeff}
For $\xi\in S_k$ we have 
\begin{equation}\label{EQ:Ajkformula}
A_j^k(\xi)=
\frac
{(-1)^j\sideset{}{^{(k)}}
\sum\limits_{1\le
s_1<\dots<s_{m-j-1}\le m}
\prod\limits_{q=1}^{m-j-1}\tau_{s_q}(\xi)}
{\prod\limits_{l=1,l\ne k}^m(\tau_l(\xi)-\tau_k(\xi))}\;,
\end{equation}
where $\sum^{(k)}$ means sum over the range indicated excluding $k$.
Furthermore, for each $j=0,\dots,m-1$,
and $k=1,\dots,m$, we have that the amplitude 
$A_j^k(\xi)$ is smooth in~$S_k$ and
\begin{equation} A_j^k(\xi)=O(|\xi|^{-j}) \textrm{ as } |\xi|\to\infty.\end{equation}
\end{lem}

It turns out that it is sensible to divide the
considerations of how characteristic roots behave into two parts:
their behaviour for large values of $|\xi|$  and for bounded
values of $|\xi|$. The reason is that the large time behaviour
of solutions of \eqref{EQ:CP} is governed by the properties of
characteristics $\tau_k$ of the full equation. In turn, these
are similar to those of $\va_k$ for large $|\xi|$.
It turns out that  the key properties of characteristics to
consider are
\begin{itemize}
\item the behaviour for large frequencies in which case the roots
$\tau_k$ must be distinct;
\item the multiplicities of roots (this only occurs for
bounded frequencies);
\item whether roots lie on the real axis or are separated
from it;
\item how roots meet the real axis (if they do);
\item the properties of the Hessian of the roots, $\nabla_\xi^2\tau_k(\xi)$;
\item convexity-type conditions guaranteeing better decay rates.
\end{itemize}

For frequencies away from multiplicities we can 
actually establish 
estimates for the corresponding 
oscillatory integrals that contribute to
the solution in \eqref{EQ:soln}. Around multiplicities we need to 
take extra care of the structure of solutions. This
can be done by dividing the frequencies into zones each of
which will give a certain decay rate. Combined together
they yield the total decay rate for solution to
\eqref{EQ:CP}. 
We now state the summarised result as it appeared in 
\cite{Ruzhansky:2010}.

\begin{thm}\label{THM:overallmainthm}
Suppose $u=u(t,x)$ is the solution of the $m^{\text{th}}$ order
linear, constant coefficient, strictly hyperbolic
Cauchy problem~\eqref{EQ:CP}. Denote the
characteristic roots of the operator by
$\tau_1(\xi),\dots,\tau_m(\xi)$, and assume that
$\Im\tau_k(\xi)\geq 0$ for all $k=1,\ldots,m$, and all
$\xi\in\Rn$.
We introduce two functions, $K^{(\text{l})}(t)$ and
$K^{(\text{b})}(t)$, which take values as follows\textup{:}
\begin{enumerate}
\item[I.]\label{RESULT:mainthmlargexi}
Consider the behaviour of~each characteristic root,
$\tau_k(\xi)$, in the region $|\xi|\ge M$,
where~$M$ is large enough. The following
table gives values for the function $K_k^{(\text{l})}(t)$
corresponding to possible properties of~$\tau_k(\xi)$\textup{;} if
$\tau_k(\xi)$ satisfies more than one, then take
$K_k^{(\text{l})}(t)$ to be function that decays the slowest as
$t\to\infty$.

\begin{center}\begin{upshape}
\begin{tabular}[4]{|c|c|c|}\hline
Location of $\tau_k(\xi)$ & Additional Property &
$K_k^{\text(l)}(t)$\\
\hline\hline%
away from real axis && $e^{-\delta t}$, some $\delta>0$\\\hline
&$\det\nabla_\xi^2\tau_k(\xi)\ne0$ &
$t^{-\frac{n}{2}(\frac{1}{p}-\frac{1}{q})}$\\
on real axis &${\rank}\nabla_\xi^2\tau_k(\xi)=n-1$ &
$t^{-\frac{n-1}{2}(\frac{1}{p}-\frac{1}{q})}$\\
& convexity condition $\gamma$ &
$t^{-\frac{n-1}{\gamma}(\frac{1}{p}-\frac{1}{q})}$\\ & no
convexity condition, $\gamma_0$ &
$t^{-\frac{1}{\gamma_0}}$\\\hline 
\end{tabular}\end{upshape}
\end{center}
Then take $K^{(\text{l})}(t)=\max_{k=1\,\dots,n}K_k^{\text(l)}(t)$.

\item[II.]\label{RESULT:mainthmbddxi}
Consider the behaviour of the characteristic roots in the bounded
region~$|\xi|\le M$\textup{;} again, take
$K^{(\text{b})}(t)$ to be the maximum \textup{(}slowest
decaying\textup{)} function for which there are roots satisfying the
conditions in the following table\textup{:}

\begin{center}\begin{upshape}
\begin{tabular}[4]{|c|c|c|}\hline
Location of Root(s)& Properties & $K^{(\text{b})}(t)$\\
\hline\hline%
away from axis & no multiplicities &$e^{-\delta t}$, some $\delta>0$\\
&$L$ roots coinciding & $t^L e^{-\delta t}$\\
\hline on axis,&$\det\nabla_\xi^2\tau_k(\xi)\ne0$ &
$t^{-\frac{n}{2}(\frac{1}{p}-\frac{1}{q})}$\\
no multiplicities
$^\ast$ 
& convexity condition $\gamma$ &
$t^{-\frac{n-1}{\gamma}(\frac{1}{p}-\frac{1}{q})}$\\
& no convexity condition, $\gamma_0$ &
$t^{-\frac{1}{\gamma_0}(\frac{1}{p}-\frac{1}{q})}$\\
\hline on axis,&$L$ roots coincide&  \\
multiplicities$^\ast, ^{**}$&on set of codimension $\ell$&$t^{L-1-\ell}$\\
\hline meeting axis & $L$ roots coincide&\\
with finite order $s$& on set of codimension $\ell$&
$t^{L-1-\frac{\ell}{s}(\frac{1}{p}-\frac{1}{q})}$\\
\hline
\end{tabular}\end{upshape}
\end{center}
\begin{small}$^\ast$ 
These two cases of roots lying on the real axis
require some additional regularity assumptions; see 
corresponding microlocal statements for details. \\
\noindent
$^{**}$ This is the $L^1-L^\infty$
rate in a shrinking region; see
\cite{Ruzhansky:2010} for the details.
\end{small}
\end{enumerate}
\medskip
Then, with
$K(t)=\max\big(K^{\text{(b)}}(t),K^{\text{(l)}}(t)\big)$,
the following estimate holds\textup{:}
\begin{equation}
\| \partial_x^\alpha\partial_t^r u(t,\cdot)\|_{L^q}\le C_{\alpha,r} K(t)
\sum_{l=0}^{m-1}\|f_l\|_{L_{N_p-l}^p}\,,
\end{equation}
where $1\le p\le2$, $pq=p+q$, and
$N_p=N_p(\alpha,r)$ is a constant depending upon~$p,\alpha$ and~$r$.
\end{thm}
Let us now briefly explain how to understand these tables.
Since the decay rates do depend on the behaviour of characteristic
roots in different regions, in Theorem 
\ref{THM:overallmainthm} 
we single out properties which determine
the final decay rate. Since the same characteristic root, 
say $\tau_k$,
may exhibit different properties in different regions, we look
at the corresponding rates 
$K^{\text{(b)}}(t),K^{\text{(l)}}(t)$ under each possible condition
and then take the slowest one for the final answer. 
The value of the Sobolev index $N_p=N_p(\alpha,r)$ 
depends on the regions
as well, and it can be found from the relevant microlocal statements 
that exist for each region.

In conditions of Part I of the theorem, it can be
shown by perturbation arguments that only three
cases are possible for large $\xi$, namely, the characteristic root
may be uniformly separated from the real axis, it may lie on the
axis, or it may converge to the real axis at infinity. If,
for example, the
root lies on the axis and, in addition, it satisfies the convexity
condition with index $\gamma$, we get the corresponding decay rate
$K^{\text{(l)}}(t)=t^{-\frac{n-1}{\gamma}(\frac{1}{p}-\frac{1}{q})}$.
Indices $\gamma$ and $\gamma_0$ in the tables are defined
as the maximum of the corresponding indices $\gamma(\Sigma_\lambda)$
and $\gamma_0(\Sigma_\lambda)$, respectively, where 
$\Sigma_\lambda=\{\xi:\tau_k(\xi)=\lambda\}$, over all $k$ and over
all $\lambda$, for which $\xi$ lies in the corresponding region.
The indices $\gamma(\Sigma_\lambda)$
and $\gamma_0(\Sigma_\lambda)$ are those defined in
Section \ref{sec:1:StatPhas}.

The statement in Part II is more involved since we may have multiple
roots intersecting on rather irregular sets. 
The number $L$ of coinciding roots corresponds to the number of
roots which actually contribute to the loss of regularity.
For example, operator $(\partial_t^2-\Delta_x)(\partial_t^2-2\Delta_x)$
would have $L=2$ for both pairs of roots 
$\pm|\xi|$ and $\pm\sqrt{2}|\xi|$,
intersecting at the
origin.  Meeting the axis with finite order $s$ 
means that we have the estimate 
\begin{equation}\label{EQ:disttau}
\dist(\xi,Z_k)^s\leq c|\Im\tau_k(\xi)|
\end{equation}
for all the intersecting roots, where $Z_k=\{\xi: \Im\tau_k(\xi)=0\}.$
In Part II of Theorem \ref{THM:overallmainthm}, 
the condition that 
$L$ roots meet the axis with finite order $s$ on a set of codimension
$\ell$ means that all these estimates hold
and that there is a ($C^1$) set $\mathcal M$ 
of codimension $\ell$ such that
$Z_k\subset \mathcal M$ for all corresponding $k$. 

In Part II of the theorem, condition $^{**}$ is formulated in
the region of the size decreasing with time:
if we have $L$ multiple roots 
which coincide on the real axis on a set $\mathcal M$ 
of codimension $\ell$,
we have an estimate  
\begin{equation}\label{EQ:around}
|u(t,x)|\leq C(1+t)^{L-1-\ell}
\sum_{l=0}^{m-1}\|f_l\|_{L^1},
\end{equation}
if we cut off the Fourier transforms of the Cauchy data
to the $\epsilon$-neighbourhood $\mathcal M^\epsilon$ of $\mathcal M$
with $\epsilon=1/t$. 
Here we may  relax the definition of the intersection
above and say that if $L$ roots coincide
in a set $\mathcal M$, then 
they coincide on a set of codimension $\ell$
if the measure of the $\epsilon$-neighbourhood $\mathcal M^\epsilon$ 
of $\mathcal M$
satisfies $|\mathcal M^\epsilon|\leq C\epsilon^{\ell}$ 
for small $\epsilon>0$;
here $\mathcal M^\epsilon=\{\xi\in\R^n: 
\dist (\xi,\mathcal M)\leq\epsilon\}.$

We can then combine this with the remaining cases outside
of this neighbourhood, where it is possible to establish
decay by different arguments. In particular, this is the case
of homogeneous equations with roots intersecting at the
origin. However, one sometimes needs to introduce special
norms to handle $L^2$-estimates around the multiplicities.
We leave this outside the scope of this review, and refer
to \cite{Ruzhansky:2010} for further details,
as well as for microlocal improvements of some of the estimates
under certain more refined assumptions.

\subsection{Properties of hyperbolic polynomials}
\label{SEC:hp}

Here we collect some properties of hyperbolic polynomials.
We start with the property of general polynomials that 
roots have bounds in terms of the coefficients.

\begin{lem}\label{LEM:auxbdd}
Consider the polynomial over $\mathbb C$ with complex coefficients
\begin{equation}
z^m+c_1z^{m-1}+\dots+c_{m-1}z+c_m=\prod_{k=1}^m(z-z_k).
\end{equation}
If there exists $M>0$ such that $|c_j|\le M^j$ for each
$j=1,\dots,m$, then $|z_k|\le2M$ for all $k=1,\dots,m$.
On the other hand, if there exists $N>0$ such that 
$|c_j|\le N$
for each $j=1,\dots,m$, then 
$|z_k|\le\max\{2,2N\}$. 
\end{lem}

Now we give some properties of the characteristic roots.
\begin{prop}\label{PROP:ctyofroots}
Let $L=L(\D_t,\D_x)$ be a linear $m^{\text{th}}$ order constant
coefficient
differential operator in $\D_t$
with coefficients that are pseudo-differential operators
in~$x$, with symbol 
\begin{equation}\label{EQ:symbolL}
L(\tau,\xi)=\tau^m+\sum_{j=1}^m P_j(\xi)\tau^{m-j}+
\sum_{j=1}^m a_j(\xi)\tau^{m-j},
\end{equation}
where $P_j(\lambda\xi)=\lambda^j P_j(\xi)$ for all 
$\lambda\gg 1$, $|\xi|\gg 1$, and $a_j\in S^{j-\epsilon}$,
for some $\epsilon>0$.

Then each of the
characteristic roots of $L$, denoted
$\tau_1(\xi),\dots,\tau_m(\xi)$, is continuous in
$\R^n$.
Furthermore, for each
$k=1,\dots,m$, the characteristic root $\tau_k(\xi)$ is
smooth away from multiplicities, and
analytic if the operator $L(\D_t,\D_x)$ is differential.
If operator $L(\D_t,\D_x)$ is strictly hyperbolic, then
there exists a constant $M$ such that, 
if $|\xi|\geq M$
then the characteristic roots 
$\tau_1(\xi),\dots,\tau_m(\xi)$ of $L$
are pairwise distinct.
\end{prop}
\begin{proof} 
The first part of Proposition \ref{PROP:ctyofroots}
is simple.
For the second part we follow \cite{Ruzhansky:2010}
and use the notation and results from Chapter 12 of
\cite{gelf+kapr+zele94} concerning the discriminant $\Delta_p$ of the
polynomial $p(x)=p_mx^m+\dots+p_1x+p_0$,
\begin{equation}
\Delta_p\equiv\Delta(p_0,\dots,p_m):=
(-1)^{\frac{m(m-1)}{2}}p_m^{2m-2}\prod_{i<j}(x_i-x_j)^2\,,
\end{equation}
where the $x_j$ ($j=1,\dots,m$) are the roots of $p(x)$. We note that
$\Delta_p$ is a continuous function of the coefficients $p_0,\dots,p_m$
of $p(x)$ and it is a homogeneous function of degree $2m-2$ in them.
In addition, it satisfies the quasi-homogeneity property:
\begin{equation}
\Delta(p_0,\lambda
p_1,\lambda^2p_2,\dots,\lambda^mp_m)=\lambda^{m(m-1)}\Delta(p_0,\dots,p_m).
\end{equation}
Clearly, $\Delta_p=0$ if and only if $p(x)$ has a double root.
We now write $L(\tau,\xi)$ in the form
\begin{equation}
L(\tau,\xi)= L_m(\tau,\xi)+a_{1}(\xi)\tau^{m-1}+a_{2}(\xi)\tau^{m-2}
+\dots +a_{m-1}(\xi)\tau+a_m(\xi),
\end{equation}
where
\begin{equation}
L_m(\tau,\xi)=\tau^m+\sum_{j=1}^mP_j(\xi)\tau^{m-j}
\end{equation}
is the principal part of $L(\tau,\xi)$. Note that the $P_j(\xi)$ are
homogeneous of degree~$j$ and the $a_j(\xi)$ are
symbols of degree $< j$. By the homogeneity and
quasi-homogeneity properties of $\Delta_L$, we have, for $\lambda\ne0$,
\begin{align*}
\Delta_L(\lambda\xi)&=\Delta(P_m(\lambda\xi)+a_m(\lambda\xi),
\dots,P_{1}(\lambda\xi)+a_{1}(\lambda\xi),1)\\
=&\;\Delta(\lambda^m[P_m(\xi)+\tfrac{a_m(\lambda\xi)}{\lambda^m}],
\dots,\lambda[P_1(\xi)+\tfrac{a_{1}(\lambda\xi)}{\lambda}],1)\\
=&\;\lambda^{m(2m-2)}\Delta(P_m(\xi)+\tfrac{a_m(\lambda\xi)}{\lambda^m},
\dots,\lambda^{-(m-1)}[P_1(\xi)+\tfrac{a_{1}
(\lambda\xi)}{\lambda}],\lambda^{-m})\\
&\qquad\qquad\qquad\qquad\text{(using that $\Delta$ is homogenous of
degree
$2m-2$)}\\
=&\;\lambda^{m(m-1)}\Delta(P_m(\xi)+\tfrac{a_m(\lambda\xi)}{\lambda^m},
\dots,P_1(\xi)+\tfrac{a_{1}(\lambda\xi)}{\lambda},1)\\
&\qquad\qquad\qquad\qquad\qquad\qquad\qquad\qquad\qquad\text{ (by
quasi-homogeneity)}.
\end{align*}
Now, since $L$ is strictly hyperbolic, the characteristic roots
$\varphi_1(\xi),\dots,\varphi_m(\xi)$ of $L_m$ are pairwise distinct
for $\xi\ne0$, so
\begin{equation} \Delta_{L_m}(\xi)=\Delta(P_m(\xi),\dots,P_1(\xi),1)\ne0\text{
for }\xi\ne0.\end{equation}
Since the discriminant is continuous in each
argument, there exists $\delta>0$ such that if
$|\tfrac{a_j(\lambda\xi)}{\lambda^{j}}|<\delta$ for all $j=1,\dots,m$
then
\begin{equation} |\Delta(P_m(\xi)+\tfrac{a_m(\lambda\xi)}{\lambda^m}\,,
\dots,P_1(\xi)+\tfrac{a_{1}(\lambda\xi)}{\lambda}\,,1)|
\not=0,\end{equation} and hence the
roots of the associated polynomial are pairwise distinct. 
Since $a_j$ are symbols of order $<j$, 
it follows that for $|\xi|=1$ and large $\lambda$, we have
$\Delta_L(\lambda\xi)\not= 0$, which means that 
the characteristic roots of $L$ are
pairwise distinct for $|\lambda\xi|\geq M$ for some constant
$M>0$.
\end{proof}

Therefore, we see that for large frequencies, the characteristic roots of the
full symbol are smooth. In fact, for large frequencies again, we can
regard the lower order terms of the operators as a perturbation of
its principal part. Using this, we can show that the
roots actually have a good symbolic behaviour for large frequencies
with a certain dependence of the order of the symbol on the
order of the lower order terms. 

We will use the standard notation for
the symbol class $S^\mu$ of all amplitudes $a=a(x,\xi)\in C^\infty(\Rn\times\Rn)$
satisfying
\begin{equation}
|\partial_x^\beta \partial_\xi^\alpha a(x,\xi)|\leq
C_{\alpha\beta} \langle\xi\rangle^{\mu-|\alpha|},
\end{equation}
for all multi-indices $\alpha,\beta$ and all $x,\xi\in\Rn$. Here,
as usual, $ \langle\xi\rangle=\sqrt{1+|\xi|^2}.$

\begin{prop}\label{PROP:perturbationresults} 
Let $L=L(\D_t,\D_x)$ 
be a hyperbolic operator of the form
\begin{equation} L(\D_t,\D_x)=\D_t^m+\sum_{j=1}^m P_j(\D_x) \D_t^{m-j}+
\sum_{j=1}^m\sum_{|\alpha|+m-j=K} c_{\alpha,j}(\D_x) \D_t^{m-j},\end{equation}
where $P_j(\lambda\xi)=\lambda^j P_j(\xi)$ for 
$\lambda\gg 1$, $|\xi|\gg 1$, and $c_{\alpha,j}\in S^{|\alpha|}.$
Here 
$0\le K\le m-1$ is the maximum order of the lower order
terms of $L$.
Let $\tau_1(\xi),\dots,\tau_m(\xi)$ denote its
characteristic roots. Then
there exists a constant $C>0$ such that
\begin{equation}\label{EQ:tauisoderxi}
|\tau_k(\xi)|\le C\langle{\xi}\rangle\quad\text{for all }\xi\in\R^n, \, k=1,\ldots,m.
\end{equation}
Suppose in addition that $L$ is \emph{strictly}
hyperbolic, and denote the roots of the principal part
$L_m(\tau,\xi)$ by $\varphi_1(\xi),\dots,\varphi_m(\xi)$.
Then the following holds:
\begin{enumerate}
\item\label{LEM:tau-vabounds} 
For each
$\tau_k(\xi)$, $k=1,\dots,m$, there exists a
corresponding root of the principal symbol $\va_k(\xi)$
\textup{(}possibly after reordering\textup{)} such that
\begin{equation}\label{EQ:tau-vaboundkthorder}
|\tau_k(\xi)-\va_k(\xi)|\le C\langle{\xi}\rangle^{K+1-m} \quad\text{for
all }\xi\in\R^n\,.
\end{equation}

\item\label{PROP:boundsonderivsoftau} There exists $M>0$
such that, for each characteristic root of~$L$ and for each
multi-index $\alpha$, we can find constants $C=C_{k,\alpha}>0$
such that
\begin{equation}\label{EQ:rootisassymbol}
|\partial^\alpha_\xi\tau_k(\xi)|\le
C|\xi|^{1-|\alpha|}\,\quad\text{for all }|\xi|\ge M\,,
\end{equation}

\item \label{PROP:derivoftau-derivofphi} There exists $M>0$ such
that, for each $\tau_k(\xi)$ a corresponding root of the
principal symbol~$\va_k(\xi)$ can be found \textup{(}possibly after
reordering\textup{)} which satisfies, for each multi-index
$\alpha$ and $k=1,\dots,m$,
\begin{equation}\label{EQ:derivsoftau-phiforfewerlot}
|\partial^\alpha_\xi\tau_k(\xi)-\partial_\xi^\alpha\va_k(\xi)| \le
C|\xi|^{K+1-m-|\alpha|}\quad\text{for all }|\xi|\ge M
\end{equation}
\end{enumerate}
\end{prop}
The estimate \eqref{EQ:tauisoderxi}
follows immediately from Lemma \ref{LEM:auxbdd}.
The symbolic behaviour can be shown by perturbative arguments.
In particular, since we always have $K\leq m-1$, we get
the following special case of Proposition 
\ref{PROP:perturbationresults} which we formulate for its
own sake:

\begin{cor}\label{PROP:perturbationresults-cor}
Let $L(\D_t,\D_x)$ be a strictly hyperbolic operator as in
Proposition \ref{PROP:perturbationresults}. 
Then we have
\begin{equation}\label{EQ:tau-vabound}
|\tau_k(\xi)-\va_k(\xi)|\le C\quad\text{for all }\xi\in\R^n\,,
\end{equation}
and there exists constants $M>0$ and  $C>0$ such that we have
\begin{equation}\label{EQ:gradtauisbounded}
|\nabla\tau_k(\xi)|\le C\quad\text{for all }|\xi|\ge M\,,
\end{equation}
and
\begin{equation}\label{EQ:derivsoftauandphisymbols}
|\partial^\alpha_\xi\tau_k(\xi)-\partial_\xi^\alpha\va_k(\xi)| \le
C|\xi|^{-|\alpha|}\quad\text{for all }|\xi|\ge M\,,
\end{equation}
for each multi-index $\alpha$ and $k=1,\dots,m$.
\end{cor}

\subsection{Estimates for oscillatory integrals}

Theorem \ref{THM:overallmainthm} is composed of a number of statements
each giving a corresponding time decay rate in a corresponding frequency
region. For the precise formulations and proofs of these statements we
refer to \cite{Ruzhansky:2010}. However, anticipating the estimates for the
oscillatory integrals appearing in the analysis of systems with time-dependent
coefficients in subsequent sections, we give now two estimates corresponding to
large frequencies in which case the properties of hyperbolic polynomials in
Section \ref{SEC:hp} can be used.

We start with the case where all the level sets
\begin{equation}
\Sigma_\lambda(\tau)=\{\xi\in\Rn: \tau(\xi)=\lambda\}
\end{equation}
of the phase function $\tau:\Rn\to\R$, if non-empty,
are convex in the sense that the set enclosed by $\Sigma_\lambda$ is convex.
For simplicity of the formulation, we agree that the
empty set is also convex. 

We also note that if $\tau$ is
a characteristic of the operator in \eqref{EQ:CP}, we have to
estimate the oscillatory integrals appearing in
\eqref{EQ:soln}. These are dealt with in Theorem \ref{THM:convexsp}
below. Proposition \ref{PROP:perturbationresults} assures us that the
function $\tau$ satisfies the assumptions of Theorem \ref{THM:convexsp}
for large frequencies $\xi$.

\begin{thm}\label{THM:convexsp}
Let $\tau:\R^n\to\R$ be such that
$\Sigma_\lambda(\tau)$ are convex for all $\lambda>0$.
Let $\chi\in C^\infty(\Rn)$, and
assume that on its support $\supp\chi$ we have the following
properties:
\begin{enumerate}
\item\label{HYP:realtauisasymbol} for all multi-indices $\alpha$ there
exist constants $C_\alpha>0$ such that
\begin{equation}
|\partial_\xi^\alpha\tau(\xi)|\le
C_\alpha\langle{\xi}\rangle^{1-|\alpha|};
\end{equation}
\item\label{HYP:realtauboundedbelow} there exist constants $M,C>0$
such that for all $|\xi|\ge M$ we have $|\tau(\xi)|\ge
C|\xi|$\textup{;}
\item\label{HYP:realderivativeoftaunonzero} there exists a constant
$C>0$ such that $|\partial_\omega\tau(\lambda\omega)|\ge C$ for all $\omega\in
\mathbb S^{n-1}$, $\lambda>0$\textup{;} in particular,
$|\nabla\tau(\xi)|\ge C$ for all
$\xi\in\R^n\setminus\{0\}$\textup{;}
\item \label{HYP:reallimitofSi_la} there exists a constant $R_1>0$
such that, for all $\lambda>0$,
\begin{equation}
\lambda^{-1} \Sigma_\lambda(\tau)
\subset B_{R_1}(0)\,.\end{equation}
\end{enumerate}
Set $\gamma:=\sup_{\lambda>0}\gamma(\Sigma_\lambda(\tau))$ and assume
this is finite. Let 
$a_j=a_j(\xi)\in S^{-j}$ be a symbol of order
$-j$  on $\R^n$. 
Then for all $t\geq 0$ we have the estimate
\begin{equation}\label{EQ:resultforconvonaxis-res}
\left\| \int_{\R^n} \mathrm e^{\mathrm i(x\cdot\xi+\tau(\xi)t)}
a_j(\xi)\chi(\xi)\widehat{f}(\xi)\,d\xi \right\|_{L^q}\le
C(1+t)^{-\frac{n-1}{\gamma}(\frac{1}{p}-\frac{1}{q})}
\|f\|_{L^p_{N_{p,j,t}}}\,,
\end{equation}
where $pq=p+q$, $1<p\le 2$, 
and the Sobolev order satisfies
$N_{p,j,t}\ge n(\frac1p-\frac1q)-j$ for
$0\leq t<1$, and
$N_{p,j,t}\ge(n-\frac{n-1}{\gamma})
(\frac1p-\frac1q)-j$ for
$t\geq 1$.
\end{thm}

Estimate \eqref{EQ:resultforconvonaxis-res} 
follows by interpolation from the $L^2-L^2$ estimate 
(easy part) combined
with an $L^1-L^\infty$ estimate (harder part). 
The estimate for large times follows from
the $L^\infty$-estimate for the kernel of
the integral operator in \eqref{EQ:resultforconvonaxis-res}.
In fact, for the kernel we have the following 

\begin{prop}\label{cor:convex}
Under the conditions of Theorem \ref{THM:convexsp}
with $\chi\equiv 1$ and with the assumption that $a\in C_0^\infty(\Rn)$
the estimate 
\begin{equation}\label{EQ:sugimotothmest0}
\left| \int_{\R^n}\mathrm e^{\mathrm i(x\cdot\xi+ \tau(\xi)t)}a(\xi)
\,d\xi \right| \le C(1+t)^{-\frac{n-1}{\gamma}}\,
\end{equation}
follows for all $x\in\Rn$ and all $t\ge0$.
\end{prop}

We now give an analogue of Theorem \ref{THM:convexsp}
without the convexity assumption on the level sets of $\tau$.
In this case we can only assure a weaker, one-dimensional
decay rate.

\begin{thm}\label{THM:noncovexargument-sp}
Let $\tau:\R^n\to\R$ be a smooth function.
Let $\chi\in C^\infty(\Rn)$, and on its support
$\supp\chi$, assume the following:
\begin{enumerate}
\item\label{HYP:nonconcplxtauisasymbol} for all multi-indices $\alpha$
there exist constants $C_\alpha>0$ such that
\begin{equation}
|\partial_\xi^\alpha\tau(\xi)|\le C_\alpha \langle{\xi}\rangle^{1-|\alpha|};
\end{equation}
\item\label{HYP:nonconcplxtauboundedbelow} there exist constants
$M,C>0$ such that for all $|\xi|\ge M$ we have
$|\tau(\xi)|\ge C|\xi |$\textup{;}
\item\label{HYP:nonconcplxderivativeoftaunonzero} there exists a
constant $C>0$ such that 
$|\partial_\omega\tau(\lambda\omega)|\ge C$ for
all $\omega\in \mathbb S^{n-1}$ and $\lambda>0$\textup{;}
\item\label{HYP:nonconcplxlimitofSi_la} there exists a constant
$R_1>0$ such that  for all $\lambda>0$ we have
\begin{equation}
\lambda^{-1} \Sigma_\lambda(\tau) \subset
B_{R_1}(0)\,.
\end{equation}
\end{enumerate}
Set $\gamma_0:=\sup_{\lambda>0}\gamma_0(\Sigma_\lambda(\tau))$ and assume it is
finite. 
Let $a_j=a_j(\xi)\in S^{-j}$ be a symbol of order
$-j$ on $\R^n$. 
Then for all $t\geq 0$ we have the estimate
\begin{equation}\label{EQ:resultfornonconvonaxis-res}
\left\| \int_{\R^n} \mathrm e^{\mathrm i(x\cdot\xi+\tau(\xi)t)}
a_j(\xi)\chi(\xi)\widehat{f}(\xi)\,d\xi\right\|_{L^q}\le
C(1+t)^{-\frac{1}{\gamma_0}(\frac{1}{p}-\frac{1}{q})}
\|f\|_{L^p_{N_{p,j,t}}}\,,
\end{equation}
where $pq=p+q$, $1<p\le 2$, 
and the Sobolev order satisfies
$N_{p,j,t}\ge n(\frac1p-\frac1q)-j$ for
$0\leq t<1$, and
$N_{p,j,t}\ge(n-\frac{1}{\gamma_0})
(\frac1p-\frac1q)-j$ for
$t\geq 1$.
\end{thm}
 As in the convex
case, the main thing is to show the following kernel estimate.

\begin{prop}\label{cor:nonconvex}
Under the conditions of Theorem \ref{THM:noncovexargument-sp}
with $\chi\equiv 1$ and the assumption
that $a\in C_0^\infty(\Rn)$ the estimate
\begin{equation}
\left| \int_{\R^n} \mathrm e^{\mathrm i(x\cdot\xi+ \tau(\xi)t)}a(\xi)
\,d\xi\right| \le C(1+t)^{-\frac{1}{\gamma_0}}\,
\end{equation}
holds true for all $x\in\Rn$ and $t\ge0$.
\end{prop}

As a corollary and an example of these theorems, 
we get the following possibilities
of decay for parts of solutions with characteristic roots on the axis.
We can use a cut-off function $\chi$ to microlocalise
around points with different qualitative behaviour
(hence we also do not have to worry about Sobolev orders).
Combining the above theorems with what we can obtain by the
stationary phase method, we get

\begin{cor}\label{COR:realaxis}
Let $\Omega\subset\Rn$ be an open set and let
$\tau:\Omega\to\R$ be a smooth real-valued function. 
Let $\chi\in C_0^\infty(\Omega).$
Let us make the following choices of $K(t)$, depending on
which of the following conditions are satisfied on $\supp\chi$.
\begin{itemize}
\item[{\rm (1)}] If $\det\nabla^2\tau(\xi)\not=0$ for all 
$\xi\in\Omega$, we set 
$K(t)=(1+t)^{-\frac{n}{2}(\frac1p-\frac1q)}.$
\item[{\rm (2)}] If $\rank\nabla^2 \tau(\xi)=n-1$ for all 
$\xi\in\Omega$, we set 
$K(t)=(1+t)^{-\frac{n-1}{2}(\frac1p-\frac1q)}.$
\item[{\rm (3)}] If $\tau$ satisfies the convexity condition
with index $\gamma$, we set 
$K(t)=(1+t)^{-\frac{n-1}{\gamma}(\frac1p-\frac1q)}.$
\item[{\rm (4)}] If $\tau$ does not satisfy the convexity
condition but has non-convex index $\gamma_0$, we set 
$K(t)=(1+t)^{-\frac{1}{\gamma_0}(\frac1p-\frac1q)}.$
\end{itemize}
Assume in each case that other assumptions of the corresponding
Theorems {\rm \ref{THM:convexsp}--\ref{THM:noncovexargument-sp}}
in cases {\rm (3)--(4)}
are satisfied, as well as assumptions in \cite{Ruzhansky:2010}
for cases {\rm (1)--(2)}. Let $1\leq p\leq 2$, $pq=p+q$.
Then for all $t\geq 0$ we have
\begin{equation}\label{EQ:realaxis}
\left|\left|\int_\Rn \mathrm e^{\mathrm i(x\cdot\xi+\tau(\xi)t)} a(\xi) \chi(\xi) 
\widehat{f}(\xi) d\xi
\right|\right|_{L^q}
\leq C K(t)||f||_{L^p}.
\end{equation}
\end{cor}
We note that no derivatives appear in the $L^p$--norm of $f$
because the support of $\chi$ is bounded. In general, there
are different ways to ensure the convexity condition for 
$\tau$, some of the criteria given in \cite{Ruzhansky:2010}.

\section{Some interesting model cases}
We now look at equations with time-dependent coefficients.
In this section we will review two scale invariant model cases, which can be treated by means
of special functions. They both highlight a structural change in the behaviour of solutions when
lower order terms become effective. This change is a true variable coefficient phenomenon, it can not arise for equations with constant coefficients as treated in Section~\ref{sec:cc}.
\subsection{Scale invariant weak dissipation} We follow the treatment in \cite{Wirth:2003} and consider the wave model
\begin{equation}\label{eq:CP-weakdiss}
  u_{tt} - \Delta u + \frac{2\mu}{t} u_t =0,\qquad u(1,\cdot)=u_0,\quad u_t(1,\cdot)=u_1 
\end{equation}
for data $u_0,u_1\in\mathscr S'(\mathbb R^n)$.  By partial Fourier transform the above problem 
reduces to the ordinary differential equation
\begin{equation}\label{eq:CP-weakdiss-trans}
 \widehat u_{tt} + |\xi|^2 \widehat u + \frac{2\mu}{t} \widehat u_t = 0
\end{equation}
parameterised by $|\xi|^2$. We will proceed in two steps, first we construct a fundamental system
of solutions to this equation and after this use it to provide a Fourier multiplier representation to solutions of \eqref{eq:CP-weakdiss}. 

\subsubsection{Reduction to special functions}
Equation \eqref{eq:CP-weakdiss-trans} can be reduced to Bessel's differential equation. 
Let $\rho=\frac12-\mu$. If we look for particular solutions of the form 
$\widehat u(t,\xi) =(t|\xi|)^\rho v( t|\xi| )$, a short calculation reduces \eqref{eq:CP-weakdiss-trans}
to 
\begin{equation}
  \tau^2  v'' +  \tau v' +(\tau^2-\rho^2) v=0, 
\end{equation}
which is Bessel's differential equation of order $\rho$. For large values of $\tau$ we will use the
fundamental system of solutions given by the Hankel functions $\mathcal H^{\pm}_{\rho}$.
In order to get a precise description of its solutions for small $\tau$, we have to distinguish between integral $\rho$ and non-integral $\rho$. In the first case,
solutions can be represented by the fundamental system $\mathcal J_\rho$, $\mathcal Y_\rho$,
while in the latter we use $\mathcal J_{\pm\rho}$. 

\begin{lem}
Any solution to \eqref{eq:CP-weakdiss-trans} can be represented in the form
\begin{align}
\widehat u(t,\xi) &= C_+(\xi) (t|\xi|)^\rho \mathcal H_\rho^+(t|\xi|) +  C_-(\xi) (t|\xi|)^\rho \mathcal H_\rho^-(t|\xi|),\\
&= A(\xi) (t|\xi|)^\rho \mathcal J_{-\rho}(t|\xi|) + B(\xi) (t|\xi|)^\rho \mathcal J_{\rho} (t|\xi|) ,\qquad &&\rho\not\in\mathbb Z, \label{eq:2:SolRep1}\\
&=\tilde A(\xi) (t|\xi|)^\rho \mathcal J_{-\rho}(t|\xi|) + \tilde B(\xi) (t|\xi|)^\rho \mathcal Y_{-\rho} (t|\xi|) ,\qquad && \rho\in\mathbb Z,
\end{align}
with coefficients depending linearly on the Cauchy data
\begin{align} 
   C_\pm(\xi) = C_{\pm,0}(\xi)\widehat u_0(\xi) + C_{\pm,1}(\xi)\widehat u_1(\xi)
\end{align}
and similarly for $A(\xi)$ and $B(\xi)$.
\end{lem}
Explicit expressions for $C_{\pm,j}(\xi)$, $A_j(\xi)$ and $B_j(\xi)$ can be obtained from these formulas and known expressions for derivatives and Wronskians of Bessel functions. We will not
go into these details here and refer to the reader to the original paper \cite{Wirth:2003}. Of greater importance for us are consequences about structural properties of solutions. They are based
on elementary properties of the Bessel functions. We collect some of them in the following proposition. They are taken from the treatise of Watson, \cite[Sections 3.52, 10.6 and 7.2]{Watson:1922}.
\begin{prop}
\begin{enumerate}
\item The functions $\mathcal H_\rho^\pm(\tau)$ possess an asymptotic expansion
\begin{equation}
  \mathcal H_\rho^\pm(\tau) \sim \mathrm e^{\pm\mathrm i \tau}\sum_{j=0}^\infty a_j^\pm \tau^{-\frac12-j}
\end{equation}
as $\tau\to\infty$ which can be differentiated term by term;
\item the function $\tau^{-\rho}\mathcal J_\rho(\tau)$ is entire and non-vanishing in $\tau=0$;
\item we have
\begin{equation}
  \mathcal Y_n(\tau) = \frac2\pi \mathcal J_n(\tau) \log \tau + \mathcal A_n(\tau)
\end{equation}
with $\tau^n\mathcal A_n(\tau)$ entire and non-vanishing in $\tau=0$.
\end{enumerate}
\end{prop}
If we want to derive properties of solutions, we have to distinguish between large
$t|\xi|$ and small $t|\xi|$. This distinction will later on lead to the introduction of zones and 
will also play a crucial r\^ole in the definition of symbol classes and in the more general consideration of hyperbolic systems.
\begin{figure}\label{fig1}
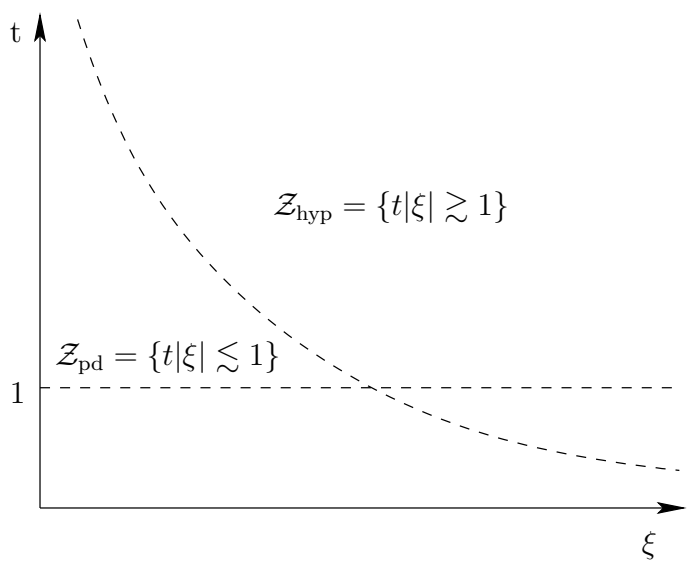
\caption{Decomposition of the phase space into zones}
\end{figure}

\subsubsection{High frequency asymptotics}\label{sec:2:HF}
The asymptotic expansion for Hankel functions immediately yield that
\begin{equation}\label{eq:2:HF-exp}
   \widehat u(t,\xi) \sim C_\pm(\xi) \mathrm e^{\pm\mathrm i t|\xi|} \left( a_0^\pm t^{-\mu} |\xi|^{-\mu} + \mathrm{l.o.t.}\right)
\end{equation}
as $t|\xi|\to\infty$. A more precise look at the terms $C_\pm(\xi)$ gives a representation 
\begin{equation}
   C_{\pm,j}(\xi) \approx |\xi|^{\mu-j}+\mathrm{l.o.t.}, \qquad  j=0,1,
\end{equation}
such that if $|\xi|\gtrsim1$ and $t|\xi|\gtrsim1$ both can be combined to the following (very rough) description. The appearing Fourier multipliers have a uniform decay rate $t^{-\mu}$ which 
corresponds to the high frequency energy estimate of the following lemma.
\begin{lem}\label{lem:2:HFE}
Assume $0\not\in\supp\widehat u_0$ and $0\not\in\supp\widehat u_1$. Then the solutions to
the weakly damped wave equation \eqref{eq:CP-weakdiss} satisfy
\begin{equation}
 \| u(t,x) \|_{L^2} \lesssim t^{-\mu} \left( \|u_0\|_{L^2} + \|u_1\|_{H^{-1}}\right)
\end{equation}
and the energy estimate
\begin{equation}
  \mathbb E(u;t) \lesssim t^{-2\mu} \mathbb E(u;0).
\end{equation}
\end{lem}
The constants in these estimates do in general depend on the distance of $0$ to the Fourier support of the data. 

We can be slightly more precise than the above lemma.  
If we multiply $\widehat u(t,\xi)$ by $t^\mu$ the main term reduces to a representation of a free wave while lower order terms decay at least as $t^{-1}$. This gives a 
description of the large-time asymptotic behaviour of weakly damped waves in terms
of free waves, see also \cite{Wirth:2005}.
 
\begin{lem}\label{lem:2:HFE2}
Assume $0\not\in\supp\widehat u_0$ and $0\not\in\supp\widehat u_1$. Then there exists a free 
wave $w$, i.e., a solution to the free wave equation 
\begin{equation}
   w_{tt}-\Delta w = 0,\qquad w(1,\cdot) = w_0,\quad w_t(1,\cdot)=w_1
\end{equation} 
to appropriate data, such that solutions to the weakly damped wave equation 
\eqref{eq:CP-weakdiss} satisfy 
\begin{equation}\label{eq:2:HFE2-scatt} 
   \| t^\mu u(t,\cdot) - w(t,\cdot )\|_{L^2} \lesssim t^{-1} \left(  \|u_0\|_{L^2} + \|u_1\|_{H^{-1}}\right)
\end{equation}
The operator assigning the data $w_0$, $w_1$ to $u_0$ and $u_1$ is bounded on 
$L^2\times H^{-1}$.
\end{lem}
\begin{proof} The proof is based on the elementary expression of $\widehat w(t,\xi)$ in terms of the inital data,
\begin{equation}
 2 \widehat w(t,\xi) = \mathrm e^{\mathrm i t|\xi|} \big(   \widehat w_0(\xi) - |\xi|^{-1} \widehat w_1(\xi)\big)
 +\mathrm e^{-\mathrm it|\xi|}\big(   \widehat w_0(\xi) + |\xi|^{-1} \widehat w_1(\xi) \big) .
\end{equation}
Therefore, we relate the initial data $w_0$ and $w_1$ to $u_0$ and $u_1$ by the system of linear equations
\begin{equation}\label{eq:2:2.16}
   \widehat w_0(\xi) \mp |\xi|^{-1} \widehat w_1(\xi)  =   2 |\xi|^{-\mu} \big(C_{\pm,0}(\xi) \widehat u_0(\xi) + C_{\pm,1}(\xi)\widehat u_1(\xi)\big). 
\end{equation}
When forming now the difference \eqref{eq:2:HFE2-scatt} this cancels the main terms of \eqref{eq:2:HF-exp} and we are left with terms decaying like $t^{-1}$ or faster. Note that \eqref{eq:2:2.16} determines the data $w_0$ and $w_1$ and yields the desired boundedness property. 
\end{proof}

\subsubsection{Low frequency asymptotics} The situation for low frequencies is completely different. 
We restrict ourselves to the case of non-integral $\rho$ (and except for $\mu=\frac12$ this will not alter any estimates we provide here). Solutions are represented by \eqref{eq:2:SolRep1}.
Crucial point is that we get {\em no decay} in time for the multiplier 
$(t|\xi|)^\rho \mathcal J_{-\rho}(t|\xi|)$, while the behaviour of
$(t|\xi|)^\rho \mathcal J_{\rho}(t|\xi|)$ depends on whether $\rho>0$ or $\rho<0$, i.e., whether $\mu<\frac12$ or $\mu>\frac12$. In the first case, estimates of solutions are in general increasing in time (similar to \eqref{eq:1:wave-sol-Est}). In the second case or when estimating higher derivatives of solutions $(t|\xi|)^\rho \mathcal J_{-\rho}(t|\xi|)$ will become the dominant part. 
Then any form of 
decay in time has to come from the behaviour of $A(\xi)$ near $\xi=0$. If $A(\xi)$ vanishes to
order $k$, we get an estimate by $t^{-k}$ uniform in $t|\xi|\lesssim1$ and if $k$ is not too large also
uniform in $|\xi|\lesssim1$. This can be used to deduce the following kind of higher order energy
estimates for solutions to \eqref{eq:CP-weakdiss}. 
\begin{lem} 
Let $\mu>\frac12$ and $k\in\mathbb N$ satisfy  $k\le\mu$. Then the following higher order
energy estimate 
\begin{equation}
  \| \D_x^\alpha u(t,\cdot) \|_{L^2} \lesssim t^{-k} \left( \|u_0\|_{H^k} + \|u_1\|_{H^{k-1}}\right)\qquad |\alpha|=k,
\end{equation}
is satisfied by any solution of \eqref{eq:CP-weakdiss}.
\end{lem}
\begin{proof}[Sketch of proof] We will omit some of the details here, for the full argument we refer to \cite[Sect. 3.1]{Wirth:2003}. As mentioned above, the decay rate for high frequencies is $t^{-\mu}$, while the decay for small frequencies has to be related to a certain zero-behaviour of the coefficient $A(\xi)$ in $\xi=0$. This is done by taking $H^k$-norms on the right-hand side and estimating just homogeneous $\dot H^k$-norms on the left-hand side. This gives an additional factor of $|\xi|^k$ for small $|\xi|$ and in turn a uniform estimate by $t^{-k}$ within the zone $t|\xi|\lesssim 1$. Note, that the second term in \eqref{eq:2:SolRep1}  decays faster, as there is an additional $t^{1-2\mu}$ factor.

It remains to discuss the intermediate part where $t|\xi|$ is large and $|\xi|$ remains bounded.  Here we will estimate the multiplier by $|\xi|^k  (t|\xi|)^{-\mu}\lesssim t^{-k}$, which is true
whenever $k\le \mu$.  
\end{proof}

We remark that the proof from \cite{Wirth:2003} also yields a similar statement involving both $t$ and $x$-derivatives of solutions. We decided to omit this, because then more detailed arguments involving recursing formulae for Bessel functions would have been needed here.

\subsubsection{Notions of sharpness} There are different ways to measure the sharpness
of a priori estimates. We will mention two of them and explain their importance related to the 
further considerations. If we are given an energy estimate of the form
\begin{equation}
  \mathbb E(u;t) \lesssim  f(t)
\end{equation}
it should always be equipped with a class of data. A first question should be: Can we find
data from this class, such that the energy really does behave in this way? This is clearly the case of 
the high frequency energy estimate of Lemma~\ref{lem:2:HFE}. In fact, for all data with 
frequency support away from zero this happens. If $\mu<1$ then this happens even for all data
from $H^1\times L^2$, see \cite{Wirth:2005}.

On the other hand, if we cannot find such data we may ask whether we can improve the estimate 
to $\mathbb E(u;t)\lesssim g(t)$ for some $g(t)=o(f(t))$. If this can be done, the estimate clearly 
was not sharp. However, if we do not find data with the prescribed rate but also can not improve
the rate, a different kind of sharpness appears. This is seen in the low frequency asymptotics. 
For any $\mu>1$ the estimate
\begin{equation}
  \|\nabla u(t,\cdot)\|_{L^2} + \|u_t(t,\cdot)\|_{L^2} \lesssim t^{-1} \left( \|u_0\|_{H^1} + \|u_1\|_{L^2} \right)
\end{equation}
is of this form.

\subsection{Scale invariant mass terms} This situation was studied recently by Del Santo--Kinoshita--Reissig \cite{DelSKR:2007} and in the PhD thesis of
C.~B\"ohme, \cite{Boehme:2011}, following along similar lines to the treatment in the previous section. We will only sketch the major differences and the conclusions to be drawn from them.

The model under consideration is the Cauchy problem
\begin{equation}
   u_{tt} - \Delta u + \frac{\kappa^2}{4 t^2} u = 0,\qquad u(1,\cdot)=u_0,\quad u_t(1,\cdot)=u_1
\end{equation}
for a Klein--Gordon equation with time-dependent mass. Again we assume data to belong
to $\mathscr S'(\mathbb R^n)$ and we reduce the problem by a partial Fourier transform to
the ordinary differential equation
\begin{equation}\label{eq:2:KG-mod-FT}
   \widehat u_{tt} + |\xi|^2 \widehat u + \frac{\kappa^2}{4 t^2}\widehat u=0.
\end{equation}
This differential equation can be related to Kummer's confluent hypergeometric equation.
Let $2\rho = 1+\sqrt{1-\kappa^2}$. If we look for particular solutions of the special form
\begin{equation}
   \widehat u(t,\xi) = \mathrm e^{\mathrm i t|\xi|}  (t|\xi|)^\rho v(t|\xi|) 
\end{equation}
we obtain with the substitution $\tau=2\mathrm it|\xi|$ 
\begin{equation}
   \tau v'' + (2\rho-\tau) v' + \tau v = 0.
\end{equation}
Solutions to this equation are given by confluent hypergeometric functions. A system of linearly
independent solutions is given in terms of Kummer's functions by
\begin{equation}
   \Theta_0(\rho,2\rho; 2 \mathrm i t|\xi|) , \qquad \Theta_0(1-\rho,2-2\rho; -2 \mathrm i t|\xi|)
\end{equation}  
for 
\begin{equation}
   \Theta_0(\alpha,\beta; \tau) = \begin{cases} \Phi(\alpha,\beta;\tau) ,\qquad &\beta\not\in\mathbb Z,\\
   \Psi(\alpha,\beta;\tau) ,\qquad & \beta \in\mathbb Z.\end{cases}
\end{equation}
In combination with known statements about these functions we can again
describe the Fourier multipliers appearing in the representation of solutions. They behave different
in different zones of the phase space, the decomposition is the one depicted in Figure~\ref{fig1}. Furthermore, we have to distinguish between small values of $\kappa$ and large values.

\medskip
\paragraph{The case $\kappa\le1$} In this case the parameter $\rho\in[1/2,1)$ is real and a solutions to 
\eqref{eq:2:KG-mod-FT} show the following behaviour. If $t|\xi|\lesssim 1$, the main terms 
of the fundamental solution behave like $t^\rho$ and $t^{1-\rho}$ (and with improvement of one order for any $t$-derivative or multiplication by $|\xi|$), except for $\kappa=1$ and $\rho=1/2$
where an additional $\log$-term appears. The large-frequency behaviour on 
$t|\xi|\gtrsim 1$ is described again by the oscillatory terms $\exp(\pm \mathrm it\xi)$ similar to
the high-frequency expansion \eqref{eq:2:HF-exp} for the scale invariant weak dissipation case.

\medskip
\paragraph{The case $\kappa>1$} Now the parameter $\rho$ is complex with real part equal to $1/2$. When $t|\xi|\lesssim 1$, solutions  to \eqref{eq:2:KG-mod-FT} grow like $t^{1/2}$ but also exhibit an oscillatory behaviour like $\exp(\pm \mathrm i \sqrt{\kappa^2-1} \log t)$. For $t|\xi|\gtrsim1$ the solutions are bounded and oscillating. The high-frequency expansion will have a changed phase.

\section{Time-dependent hyperbolic systems}\label{sec:3}

In this section we will provide a diagonalisation based approach to obtain 
the high-frequency asymptotic properties of the representation of solutions for
more general uniformly strictly hyperbolic systems. The exposition is based 
on ideas from the authors' paper \cite{RW:2011}. 

One immediate complication compared to the constant coefficients case in
Section \ref{sec:cc} is that there is no simple formula similar to
\eqref{EQ:Ajkformula} for amplitudes of the solutions. Thus, the first 
task is to construct suitable substitutes for \eqref{EQ:Ajkformula}, which can
be done by different methods depending on the properties of the coefficients
of the equation. In the presence of the lower order terms of the equation,
an hierarchy has to be introduced in order to fall into the construction scheme
for the amplitudes. Thus, in Section \ref{sec:cc} any lower order terms could
have been allowed, but here, we need to impose decay conditions, for them
to fall into the required symbolic hierarchy.

\subsection{Motivating examples}
The motivation to consider systems in this framework is two-fold. On the one hand,
the treatment of wave models and more general higher order hyperbolic
equations naturally leads to a reformulation as pseudo-differential hyperbolic systems.
This was key ingredient for Reissig--Smith \cite{RS:2005} to treat wave equations with bounded
time-dependent propagation speed. In \cite{Wirth:2006}, \cite{Wirth:2007}, Wirth
considered time-dependent dissipation terms and discussed their influence on energy
and dispersive type estimates for their solutions. The treatment of the non-effective case fits to 
the considerations presented here. 
Equations with homogeneous symbols and time-dependent coefficients appear naturally
also in the analysis of the Kirchhoff equations. Higher order equations of Kirchhoff type
were discussed by Matsuyama and Ruzhansky in \cite{MR:2009}.

On the other hand, systems are of interest on their own.
Recently, D'Abbico--Lucente--Taglialatela \cite{dALT:2009} studied general time-dependent
strictly hyperbolic differential systems of the form
\begin{equation}\label{eq:3:diffHypSystem}
    \D_t U = \sum_{j=1}^n A_j(t)\D_{x_j} U+B(t)U ,\qquad U(0,\cdot)=U_0,
\end{equation}
where as usual $\D = -\mathrm i\partial$ and the matrices $A_j(t), B(t)\in\mathbb C^{d\times d}$
satisfy natural symbolic and structural assumptions.
Besides guaranteeing a form of uniform strict hyperbolicity up to $t=\infty$ it is of importance to
control the amount of oscillations in coefficients. In this setting this is usually done by assuming
\begin{equation}
   A_j(t) \in \mathcal T_\nu\{0\},\qquad B(t) \in \mathcal T_\nu\{1\},
\end{equation}
where
\begin{equation}
  \mathcal T_\nu\{\ell\} = \left\{ f \in C^\infty(\mathbb R_+) :   |\D_t^k f(t)| \le C_k \left(\frac{(\log(e+t))^\nu}{1+t} \right)^{\ell+k} \right\}
\end{equation}
and $\nu\in[0,1]$. We will concentrate on the case $\nu=0$ here; this simplifies the consideration
without omitting most of the main ideas.\footnote{We omit the index $\nu$ if $\nu=0$.} 
Key assumption will always be that the characteristic roots,
i.e. the solutions $\lambda_j(t,\xi)$, $j=1,\ldots, d$ of the polynomial equation
\begin{equation}\label{eq:3:diffHypSystemRoots}
    \det \left(\lambda(t,\xi) - \sum_{j=1}^n A_j(t)\xi_j\right) = 0
\end{equation}
are real and uniformly separated from each other for $t\ge0$ and $|\xi|=1$.

The approach also applies to higher order scalar hyperbolic equations with purely time-dependent
coefficients of the form
\begin{equation}\label{eq:3:higher-order-eq}
   \D_t^m u + \sum_{j=0}^{m-1} \sum_{|\alpha| \le m-j} a_{j,\alpha}(t) \D_t^j \D_x^\alpha u = 0,
   \qquad \D_t^j u(0,\cdot)=u_j,\; j=0,1,\ldots, m-1,
\end{equation}
where
\begin{equation}
  a_{j,\alpha}(t) \in\mathcal T_\nu\{m-j-|\alpha|\},
\end{equation}
provided the roots of the homogeneous polynomial 
\begin{equation}
    \lambda^m + \sum_{j=0}^{m-1} \sum_{|\alpha|=m-j} a_{j,\alpha}(t) \lambda^j \xi^\alpha = 0
\end{equation}
are real and again uniformly separated for $t\ge0$ and $|\xi|=1$.

Note, that our approach uses a high regularity of coefficients (resulting in sharp conditions on the decay behaviour). There is an alternative approach due to Matsuyama--Ruzhansky \cite{MR:2010} based on an asymptotic integration argument applicable to homogeneous higher order equations of the form \eqref{eq:3:higher-order-eq} with $a_{j,\alpha}'(t)    \in L^1(\mathbb R_+)$
for $j+|\alpha|=m$ and $a_{j,\alpha}(t)=0$ otherwise. This yields uniform bounds on the energy of solutions as well as it allows for the derivation of dispersive estimates. Main difference is that
one has to use a higher regularity of the data to compensate the low regularity of coefficients in order to obtain sharp dispersive decay rates, see Section \ref{AI}.

\subsection{Symbol classes}
The following calculations are based on a decomposition of the phase space into 
different zones. They correspond to the distinction between small and large frequencies
used in the previous section, see Figure \ref{fig1}. We denote
\begin{equation}
   \mathcal Z_{\rm hyp} = \{ (t,\xi) : (1+t)|\xi|\ge 1 \}
\end{equation}
as the hyperbolic zone and set $\mathcal Z_{\rm pd} = (\mathbb R_+\times\mathbb R^n)\setminus \mathcal Z_{\rm hyp}$. This decomposition was inspired by the treatment of Yagdjian, \cite{Yagdjian:1997}. We denote the common boundary between both regions by $t_\xi$ and use the notation
$\chi_{\rm hyp}(t,\xi) = \chi((1+t)|\xi|)$ for a smooth excision function supported inside $\mathcal Z_{\rm hyp}$ and being equal to one for $(1+t)|\xi|\ge2$.

We denote by $\mathcal S\{m_1,m_2\}$ the set of all time-dependent Fourier multipliers $a(t,\xi) \in C^\infty(\mathbb R_+\times\mathbb R^n)$ satisfying the symbol estimates
\begin{equation}
  | \D_t^k \D_\xi^\alpha a(t,\xi) | \le  C_{k,\alpha} \left(\max\left\{|\xi|, \frac1{1+t}\right\}\right)^{m_1-|\alpha|} \left(\frac1{1+t}\right)^{m_2+k}
\end{equation}
for all $k$ and multi-indices $\alpha$. The set $\mathcal S\{m_1,m_2\}$ possesses a natural Fr\'echet space
structure. Furthermore, the classes behave well under forming products, taking derivatives and 
certain integrations. The excision function $\chi_{\rm hyp}$ belongs to $\mathcal S\{0,0\}$.

\begin{prop}\label{prop:3:SymbCalc}
\begin{enumerate}
\item $\mathcal S\{m_1-k,m_2+k'\} \subseteq \mathcal S\{m_1,m_2\}$ whenever $k'\ge k$;
\item $\mathcal S\{m_1,m_2\} \cdot \mathcal S\{m_1',m_2'\} \subseteq \mathcal S\{m_1+m_1',m_2+m_2'\}$;
\item $\D_t^k\D_\xi^\alpha \mathcal S\{m_1,m_2\} \subseteq \mathcal S\{m_1-|\alpha|,m_2+k\}$.
\item If $a(t,\xi) \in \mathcal S\{-1,2\}$ satisfies $\supp a(t,\xi) \subset \mathcal Z_{\rm hyp}$, then 
\begin{equation}
  b(t,\xi) =   \int_t^\infty a(\tau,\xi)\mathrm d\tau, \qquad \xi\ne0
\end{equation}
satisfies $\chi_{\rm hyp}(t,\xi) b(t,\xi) \in \mathcal S\{-1,1\}$.
\end{enumerate}
\end{prop}
\begin{proof}
We only explain how to prove the last statement, the others are obvious. From the 
symbol estimates we get for $t\ge t_\xi$
\begin{align}
   |b(t,\xi)| \le C \int_{t}^\infty \frac{1}{(1+\tau)^2 |\xi|} \mathrm d\tau = \frac{C}{(1+t)|\xi|}
\end{align}
and the uniform bound by $C$ for $t\le t_\xi$. As derivatives with respect to $t$ clearly
satisfy the right estimates, we are left with $\xi$-derivatives. For them we obtain in 
a similar fashion
\begin{align*}
   |\D_\xi^\alpha b(t,\xi)| \le C \int_{t}^\infty \frac{1}{(1+\tau)^2 |\xi|^{1+|\alpha|}} \mathrm d\tau 
   = \frac{C}{(1+t)|\xi|^{1+|\alpha|}}
\end{align*}
whenever $t\ge t_\xi$. Combining this with support properties of $\chi_{\rm hyp}(t,\xi)$ and
its symbol estimates the desired statement follows.
\end{proof}

Of particular importance for us will be the embedding $\mathcal S\{m_1-k,m_2+k\}\hookrightarrow
\mathcal S\{m_1,m_2\}$.  It is customary to denote the residual class of this hierarchy by
\begin{equation}
   \mathcal H\{\ell\} = \bigcap_{m_1+m_2=\ell} \mathcal S\{m_1,m_2\} .
\end{equation} 
We will also need a certain homogeneous version of these classes allowing for
singluraties in $\xi=0$. We write $\mathcal S_*\{m_1,m_2\}$ to denote 
the class of all functions $a(t,\xi) \in C^\infty (\mathbb R_+\times (\mathbb R^n\setminus\{0\}))$
satisfying 
\begin{equation}
  | \D_t^k \D_\xi^\alpha a(t,\xi) | \le  C_{k,\alpha} |\xi|^{m_1-|\alpha|} \left(\frac1{1+t}\right)^{m_2+k}
\end{equation}
uniform in $(1+t)|\xi|\ge c_0$ for some constant $c_0>0$.

\subsection{Uniformly strictly hyperbolic systems}
We collect our main assumptions now. We use the notation $\mathcal S\{m_1,m_2\}$ also for
matrix-valued multipliers and denote operators corresponding to such multipliers
by $a(t,\D_x)$, i.e., we define
\begin{equation}
   a(t,\D_x) f = \mathscr F^{-1} [ a(t,\xi) \widehat f (\xi)].
\end{equation}
Then we consider the Cauchy problem 
\begin{equation}\label{eq:3:CP}
   \D_t U = A(t,\D_x) U,\qquad U(0,\cdot) = U_0 \in\mathscr S'(\mathbb R^n; \mathbb C^d)
\end{equation}
for a $d\times d$ matrix $A(t,\xi)\in \mathcal S\{1,0\}$ satisfying suitable conditions. 
\begin{description}
\item[(A1)]  There exists a positively $\xi$-homogeneous matrix function $A_1(t,\xi)\in \mathcal S_*\{1,0\}$ satisfying 
 $A_1(t,\rho\xi) = \rho A_1(t,\xi)$, $\rho>0$, 
 such that
 \begin{equation}
     A(t,\xi)-A_1(t,\xi)\in \mathcal S_*\{0,1\}.
 \end{equation}
\item[(A2)] The eigenvalues $\lambda_j(t,\xi)$ of $A_1(t,\xi)$ are real and uniformly distinct
  in the sense that
  \begin{equation}
     \inf_{t, \xi\not=0} \frac{|\lambda_i(t,\xi) -\lambda_j(t,\xi)|}{|\xi|} > 0
  \end{equation}
  for all $i\ne j$.
\end{description}
We call \eqref{eq:3:CP} uniformly strictly hyperbolic if both of these assumptions are satisfied. The 
symbolic calculus of Proposition~\ref{prop:3:SymbCalc} allows to draw several conclusions. Statements and proofs are taken from \cite{RW:2011}.

\begin{lem}\label{lem:3:SymbEstDiag1}
\begin{enumerate}
\item
The eigenvalues $\lambda_j(t,\xi)$ of $A_1(t,\xi)$ satisfy the symbol estimates
\begin{equation}
  \lambda_j(t,\xi) \in \mathcal S_*\{1,0\}
\end{equation}
together with
\begin{equation}
 \big( \lambda_i(t,\xi) - \lambda_j(t,\xi)\big)^{-1} \in \mathcal S_*\{-1,0\}.
\end{equation}
\item 
The  spectral projection $P_j(t,\xi)$ associated to the eigenvalue $\lambda_j(t,\xi)$ satisfies $P_j(t,\xi)\in \mathcal S_*\{0,0\}$. 
\item
There exists an invertible matrix $M(t,\xi)\in \mathcal S_*\{0,0\}$ with inverse satisfying
$M^{-1}(t,\xi)\in \mathcal S_*\{0,0\}$ such that
\begin{equation}
  A_1(t,\xi) M(t,\xi)= M(t,\xi) \mathcal D(t,\xi)
\end{equation}
holds true for $\mathcal D(t,\xi) = \diag(\lambda_1(t,\xi),\ldots,\lambda_d(t,\xi))$ the diagonal matrix
with entries $\lambda_j(t,\xi)$.
\end{enumerate}
\end{lem}
\begin{proof}
1.  The properties of the characteristic roots follow from the spectral estimate
$ |\lambda_j(t,\xi) | \le \| A_1(t,\xi)\|$ together with the symbol properties of the coefficients 
of the characteristic polynomial and the uniform strict hyperbolicity. Indeed, differentiating
the characteristic polynomial 
\begin{equation}
   0 = \sum_{k=0}^d I_{d-k}(t,\xi) \big(\lambda_j(t,\xi)\big)^{k} ,\qquad I_k(t,\xi) \in \mathcal S_*\{k,0\}
\end{equation}
yields the linear equation 
\begin{equation}
   \D_t \lambda_j(t,\xi) \sum_{k=1}^d k I_{d-k}(t,\xi) \big(\lambda_j(t,\xi)\big)^{k-1}
   =  -   \sum_{k=0}^d \big(\D_t I_{d-k}(t,\xi)\big) \big(\lambda_j(t,\xi)\big)^{k}
\end{equation} 
for the $t$-derivatives. The assumption of uniform strict hyperbolicity (A2) is equivalent to
a uniform bound for the inverse of the sum on the left-hand side, indeed it follows from Vieta's formula that
\begin{equation}
  \sum_{k=1}^d k I_{d-k}(t,\xi) \big(\lambda_j(t,\xi)\big)^{k-1} 
  = \prod_{i\ne j} \big(\lambda_j(t,\xi) - \lambda_i(t,\xi)\big)
\end{equation}
holds true for any fixed $j$. This yields the desired estimate for $\D_t \lambda_j(t,\xi)$. Iterating this procedure gives
expressions for higher time-derivatives in terms of lower ones and implies corresponding estimates. 
On the other hand, smoothness alone together with the homogeneity implies the symbolic estimates with respect to $\xi$.

Finally, the symbolic estimates for the inverse of the difference follow from differentating the identity $\big(\lambda_i(t,\xi) - \lambda_j(t,\xi)\big)^{-1}\big(\lambda_i(t,\xi) - \lambda_j(t,\xi)\big)=1$.

2. The eigenprojections can be expressed in terms of the characteristic roots as
\begin{equation}
P_j(t,\xi) = \prod_{i\ne j} \frac{A_1(t,\xi) - \lambda_i(t,\xi)}{\lambda_j(t,\xi)- \lambda_i(t,\xi)},
\end{equation}
such that the symbol estimates for $P_j(t,\xi)$ follow directly from the symbolic calculus of Proposition~\ref{prop:3:SymbCalc}.

3. A symmetriser of the matrix $A_1(t,\xi)$, i.e., a matrix $H(t,\xi)$ such that $H(t,\xi)A_1(t,\xi)$ is self-adjoint, is given by
\begin{equation}
   H(t,\xi)= \sum_{j=1}^d P_j^*(t,\xi) P_j(t,\xi)  
\end{equation}
and therefore satisfies $H(t,\xi)\in \mathcal S_*\{0,0\}$. We can express the inverse of the diagonaliser in terms of this symmetriser. Let $\nu_j(t,\xi)$ be a (smoothly chosen) unit vector
from the one-dimensional $j$th eigenspace $\ran P_j(t,\xi)$ and $M^{-1}(t,\xi)\zeta$ for any fixed $\zeta\in\mathbb C^d$ be the vector 
with the inner products $(\nu_j , P_j(t,\xi) \zeta) = (\nu_j , H(t,\xi) \zeta)$ as entries. 

Since $\nu_j(t,\xi)$ is unique up to sign locally in $t$ and $\xi$, it is expressible as 
$P_j(t,\xi) \zeta  / \|P_j(t,\xi)\zeta\|$ for some fixed vector $\zeta\in\mathbb C^d$ chosen away from the complement of the eigenspace. Then differentiating this expression directly implies 
$\nu_j(t,\xi)\in \mathcal S_*\{0,0\}$. This implies $M^{-1}(t,\xi)\in \mathcal S_*\{0,0\}$ by the above definition of $M^{-1}(t,\xi)$. Furthermore, $M(t,\xi)$ has the vectors $\nu_j(t,\xi)$ as  columns and again $M (t,\xi)\in\mathcal S_*\{0,0\}$ follows.
\end{proof}

\subsection{Diagonalisation}\label{sec:3:diag}
As the next step we want to diagonalise the system \eqref{eq:3:CP} within the $\mathcal S\{\cdot,\cdot\}$ hierarchy modulo remainders from $\mathcal H\{1\}$ and terms 
supported within $\mathcal Z_{\rm pd}$. 

\subsubsection{Initial step} We first diagonalise the main part $A_1(t,\xi)$ using the matrix
family $M(t,\xi)\in \mathcal S_*\{0,0\}$. We denote
\begin{equation}
   V^{(0)}(t,\xi) = M^{-1}(t,\xi) \widehat U(t,\xi),
\end{equation}
within $\mathcal Z_{\rm hyp}$, so that a short calculation yields
\begin{equation}\label{eq:3:CP-diag1}
   \D_t V^{(0)} = \left (  \mathcal D(t,\xi) + R_0(t,\xi)  \right) V^{(0)} 
\end{equation}
with 
\begin{align*}
  R_0(t,\xi) &=  M^{-1}(t,\xi) (A(t,\xi) - A_1(t,\xi) M(t,\xi) \\&\quad + (\D_t M^{-1}(t,\xi)) M(t,\xi) 
   \\& \in  \mathcal S_*\{0,1\}.
\end{align*}
We will use \eqref{eq:3:CP-diag1} as starting point for a further (sequence) of transformations
applied to the system.

\subsubsection{The diagonalisation hierarchy} We will construct matrices
$$N^{(k)}(t,\xi)\in\mathcal S_*\{-k,k\} \textrm{ and } F^{(k)}(t,\xi)\in \mathcal S_*\{-k,k+1\}$$
such that with
\begin{equation}
   N_K(t,\xi) = \mathrm I + \sum_{k=1}^K N^{(k)}(t,\xi) ,\qquad
   F_{K-1}(t,\xi) = \sum_{k=0}^{K-1} F^{(k)}(t,\xi)
\end{equation}  
the operator identity
\begin{multline}\label{eq:3:Op-Ident}
 B_K(t,\xi) =  \big( \D_t -\mathcal D (t,\xi) - R_0(t,\xi) \big) N_K(t,\xi) \\ - N_K(t,\xi)\big( \D_t -\mathcal D(t,\xi) - F_{K-1}(t,\xi)\big) \in \mathcal S_*\{ -K,K+1\}
\end{multline}
is valid. The construction is done recursively. The identity \eqref{eq:3:Op-Ident} will yield
commutator equations for the matrices $N^{(k)}(t,\xi)$ and the matrices $F^{(k)}(t,\xi)$ are
determined by their solvability condition. The strict hyperbolicity of the system is crucial for this construction.

First step, $K=1$. We collect all terms from \eqref{eq:3:Op-Ident} 
which do not a priori belong to $\mathcal S_*\{-1,2\}$ and require their sum to vanish. 
This yields the condition
\begin{equation}
   [\mathcal D(t,\xi), N^{(1)}(t,\xi) ] = - R_0(t,\xi) + F^{(0)}(t,\xi). 
\end{equation}
Because  $\mathcal D(t,\xi)$ is diagonal, the diagonal entries of the commutator vanish
and we must have
\begin{equation}
   F^{(0)}(t,\xi) = \diag  R_0(t,\xi)
\end{equation}
for solvability. Furthermore, the strict hyperbolicity assumption (A2) implies
that the solution is unique up to diagonal matrices and the entries of the matrix 
$N^{(1)}(t,\xi)$ must be given by
\begin{equation}
    \big(N^{(1)}(t,\xi)\big)_{i,j} =  \frac{\big(R_0(t,\xi)\big)_{i,j}}{\lambda_i(t,\xi)-\lambda_j(t,\xi)},
    \qquad i\ne j.
\end{equation}
The diagonal entries will be set as $  \big(N^{(1)}(t,\xi)\big)_{i,i} = 0$. It is evident
that $F^{(0)}(t,\xi) \in \mathcal S_*\{0,1\}$, while $N^{(1)}(t,\xi)\in\mathcal S_*\{-1,1\}$
follows from Lemma~\ref{lem:3:SymbEstDiag1}.

Recursion $k\mapsto k+1$. If we assume that all terms up to order $k$ are already constructed
and satisfy the symbolic inequalities then the conditions for the next terms follow from the 
requirement
\begin{equation}
    B_{k+1}(t,\xi) - B_k(t,\xi) \in \mathcal S_*\{-k-1,k+2\},
\end{equation}
with $B_k$ as in \eqref{eq:3:Op-Ident}.
Indeed, collecting again just the terms which do not a priori belong to this symbol class and setting their sum to zero yields the commutator equation
\begin{equation}
   [\mathcal D(t,\xi) , N^{(k+1)}(t,\xi)] = B_{k}(t,\xi) + F^{(k)}(t,\xi).
\end{equation}
The solution is again given by
\begin{equation}
    \big(N^{(k+1)}(t,\xi)\big)_{i,j} = - \frac{\big(B_k(t,\xi)\big)_{i,j}}{\lambda_i(t,\xi)-\lambda_j(t,\xi)},
    \qquad i\ne j,
\end{equation}
provided we set
\begin{equation}
   F^{(k)}(t,\xi) =- \diag B_k(t,\xi).
\end{equation}
The diagonal entries will again be fixed as $\big(N^{(k+1)}(t,\xi)\big)_{i,i} =0$. Clearly, the diagonal
terms satisfy 
$F^{(k)}(t,\xi)\in \mathcal S_*\{-k,k+1\}$ as consequence of the assumption
$B_k(t,\xi)\in\mathcal S_*\{-k,k+1\}$, and
 $N^{(k+1)}(t,\xi)\in  \mathcal S_*\{-k-1,k+1\}$
again follows from Lemma~\ref{lem:3:SymbEstDiag1}. Furthermore, this choice
of matrices implies $B_{k+1}(t,\xi)\in\mathcal S_*\{-k-1,k+2\}$ and the recursion
step is completed.

\subsubsection{Zone constants and invertibility} If we consider only the part of the phase
space defined by $(1+t)|\xi|\ge c$ for large $c$,  the matrix-norm of 
$\mathrm I - N_k(t,\xi) \in \mathcal S_*\{-1,1\}$ is of size $\mathcal O(c^{-1})$ and choosing
$c$ large enough implies invertibility of the matrix family $N_k(t,\xi)$ with
$N_k^{-1}(t,\xi)\in \mathcal S_*\{0,0\}$. Hence, the problem \eqref{eq:3:CP-diag1} is
equivalent to considering
\begin{equation}\label{eq:3:CP-diagk}
 \D_t V^{(k)} = \big(\mathcal D(t,\xi) + F_{k-1}(t,\xi) +R_k(t,\xi)\big) V^{(k)}
\end{equation}
for $V^{(k)}(t,\xi) = N_k^{-1}(t,\xi) V^{(0)}(t,\xi)$ and with remainder term
\begin{equation}
   R_k(t,\xi) = - N_k^{-1}(t,\xi) B_k(t,\xi) \in \mathcal S_*\{-k,k+1\}
\end{equation}
within the smaller hyperbolic zone $\mathcal Z_{\rm hyp}(c_k)=\{(t,\xi) : (1+t)|\xi|\ge c_k\}$.

We collect the main result in the following lemma.

\begin{lem}\label{lem:3:diagLem}
 Let $k\in\mathbb N$, $k\ge 1$. Then there exist matrices
$N_k(t,\xi) \in \mathcal S_*\{0,0\}$, diagonal matrices 
$F_{k-1}(t,\xi) \in\mathcal S_*\{0,1\}$ and a remainder
$R_k\in\mathcal S_*\{-k,k+1\}$, such that
\begin{multline}
   \big(\D_t -\mathcal D(t,\xi) - R_0(t,\xi)\big) N_k(t,\xi) \\ = N_k(t,\xi) \big(\D_t -\mathcal D(t,\xi) -
   F_{k-1}(t,\xi) - R_k(t,\xi) \big)
\end{multline}
holds true within $Z_{\rm hyp}(c_k)$ for sufficiently large zone constant depending on $k$.
The matrices $N_k(t,\xi)$ are uniformly invertible within $\mathcal Z_{\rm hyp}(c_k)$ with $N_k^{-1}(t,\xi)\in\mathcal S_*\{0,0\}$.
\end{lem}

\subsection{Solving the diagonalised system} \label{sec:3:SolvDiag}
It remains to asymptotically solve \eqref{eq:3:CP-diagk} as $(1+t)|\xi|\to\infty$. This will be done
in two main steps. We always assume that we are indeed working inside the hyperbolic
zone $\mathcal Z_{\rm hyp}(c_k)$ for sufficiently large constant $c_k$.

\subsubsection{Treating the diagonal terms} The fundamental matrix of the diagonal part
of \eqref{eq:3:CP-diagk}, i.e., the matrix-valued solution $\mathcal E_k(t,s,\xi)$ of 
\begin{equation}
   \D_t     \mathcal E_k(t,s,\xi)  = \big(\mathcal D(t,\xi) + F_{k-1}(t,\xi) \big)\mathcal E_k(t,s,\xi),
        \qquad \mathcal E_k(s,s,\xi) = \mathrm I \in \mathbb C^{d\times d},
\end{equation}
is just given by integration and taking exponentials
\begin{equation}\label{eq:3:Ek-def}
  \mathcal E_k(t,s,\xi)  = \exp \left( \mathrm i \int_s^t  \big(\mathcal D(\tau,\xi) + F_{k-1}(\tau,\xi) \big)
  \mathrm d\tau \right).  
\end{equation}
Here we essentially used that diagonal matrices commute with each other. 
By assumption, the entries of
$\mathcal D(t,\xi)$ are real and yield a unitary matrix after exponentiating. Similarly, 
$F^{(j)}(t,\xi)\in \mathcal S_*\{-j,j+1\}$ is integrable over the hyperbolic zone
for all $j\ge 1$. This yields that the actual large time asymptotic behaviour of $\mathcal E_k(t,s,\xi)$
is encoded in the term $F^{(0)}(t,\xi)$ and we can show that
\begin{equation}
    | \mathcal E_k(t,s,\xi) | \approx \exp\left(- \int_s^t \Im F^{(0)}(\tau,\xi) \mathrm d\tau\right)   
\end{equation}
holds true for all individual diagonal entries of $\mathcal E_k(t,s,\xi)$ as two-sided estimate
with constants depending on $k$ but not on $s$, $t$ and $\xi$. From 
$F^{(0)}(t,\xi)\in \mathcal S_*\{0,1\}$ we can only conclude polynomial bounds on the right hand
side, thus there exists exponents $K_1$ and $K_2$ such that for $t>s$ the estimate
\begin{equation}\label{eq:3:Ek-bound}
   \left(\frac{1+s}{1+t}\right)^{K_1} \lesssim \| \mathcal E_k(t,s,\xi) \| \lesssim    \left(\frac{1+t}{1+s}\right)^{K_2} 
\end{equation}
is valid within $\mathcal Z_{\rm hyp}(c_k)$. The exponents are independent of $k$ and a similar estimate
is true for $t<s$ where the exponents are switched.

\subsubsection{Generalised energy conservation} We will speak of generalised energy conservation
for the system \eqref{eq:3:CP}, if its fundamental solution (given as solution to)
\begin{equation}\label{eq:3:fundSol}
   \D_t \mathcal E(t,s,\xi) = A(t,\xi) \mathcal E(t,s,\xi) ,\qquad \mathcal E(s,s,\xi) = \mathrm I\in\mathbb C^{d\times d}
\end{equation}
satisfies 
\begin{equation}
    \| \mathcal E(t,s,\xi) \| \lesssim 1
\end{equation}
uniformly in $(t,\xi), (s,\xi)\in \mathcal Z_{\rm hyp}(c)$ for some $c$ (regardless of their order). The 
generalised energy conservation property can be fully characterised by the term $F^{(0)}(t,\xi)
=\diag R_0(t,\xi)$
constructed within the diagonalisation procedure. 
\begin{thm}\label{them:3:GECL}
 The system \eqref{eq:3:CP} has the generalised energy conservation property if and only if
 \begin{equation}\label{eq:3:GECL-cond}
     \sup_{(t,\xi), (s,\xi)\in \mathcal Z_{\rm hyp}(c)} \left\| \int_s^t \Im F^{(0)}(\tau,\xi)
     \mathrm d\tau \right\| <\infty
 \end{equation} 
 holds true for some zone constant $c$.
\end{thm}
\begin{proof}[Sketch of proof]
Since the matrices $M(t,\xi)\in\mathcal S_*\{0,0\}$ and $N_k(t,\xi)\in\mathcal S_*\{0,0\}$ 
have uniformly bounded inverses $M^{-1}(t,\xi), N_k^{-1}(t,\xi) \in \mathcal S_*\{0,0\}$, the
generalised energy conservation property of system \eqref{eq:3:CP} is equivalent to that of
system \eqref{eq:3:CP-diagk}. We exploit this fact in combination with the polynomial bounds
\eqref{eq:3:Ek-bound} which are also valid (with the same exponents) for the fundamental matrix 
$\widetilde {\mathcal E}_k(t,s,\xi)$ of  \eqref{eq:3:CP-diagk}. 

The main idea now is, that when choosing $k$ large enough, the polynomial decay of the remainder term $R_k(t,\xi)$  compensates the increasing behaviour and we can use the Duhamel representation
\begin{equation}\label{eq:3:GECL-IntEq}
    \widetilde{\mathcal E}_k(t,s,\xi) = \mathcal E_k(t,s,\xi)\mathcal Z_k(s,\xi)
        - \mathrm i \int_t^\infty \mathcal E_k(t,\theta,\xi) R_k(\theta,\xi) \widetilde{\mathcal E}_k(\theta,s,\xi) \mathrm d\theta 
\end{equation}
for $k\ge K_1+K_2$ with 
\begin{equation}
   \mathcal Z_k(s,\xi) = \mathrm I + \mathrm i \int_s^\infty \mathcal E_k(s,\theta,\xi)  R_k(\theta,\xi) \widetilde{\mathcal E}_k(\theta,s,\xi) \mathrm d\theta
\end{equation}
as integral equation relating $\widetilde{\mathcal E}_k(t,s,\xi)$ to $\mathcal E_k(t,s,\xi)$. Indeed,
\begin{multline}
  \|\mathcal Z_k(s,\xi)-\mathrm I \| \lesssim  |\xi|^{-k} \int_t^\infty \left(\frac{1+\theta}{1+s}\right)^{K_1} 
  \frac1{(1+\theta)^{k+1}} \left(\frac{1+\theta}{1+s}\right)^{K_2} \mathrm d\theta \\
  \approx \big( (1+s) |\xi|\big)^{-k} 
\end{multline}  
and similar for 
\begin{multline}
\left\| \int_t^\infty \mathcal E_k(t,\theta,\xi) R_k(\theta,\xi) \widetilde{\mathcal E}_k(\theta,s,\xi) \mathrm d\theta\right\|
\lesssim |\xi|^{-k} (1+t)^{K_1-k}(1+s)^{-K_2}.
\end{multline} 
Therefore choosing $k$ and $s$ large implies that the difference $\mathcal E_k(t,s,\xi)-
\widetilde{\mathcal E}_k(t,s,\xi)$ can be estimated in terms of $(1+t)^{K_2} (1+s)^{-K_2-k}|\xi|^{-k}$
 (coming from $\mathcal Z_k$) and $ (1+t)^{K_1-k}(1+s)^{-K_2}|\xi|^{-k}$ (coming from the integral) and we can treat $\widetilde{\mathcal E}_k(t,s,\xi)$
as small perturbation of $\mathcal E_k(t,s,\xi)$.

The statement now follows from the observation that \eqref{eq:3:GECL-cond} is equivalent to
a uniform bound on both $\mathcal E_k(t,s,\xi)$ and its inverse.
\end{proof}

It should be clear from the statement of the previous theorem, but we will 
in particular point out that the integral condition \eqref{eq:3:GECL-cond} is independent
of the choice of the diagonaliser $M(t,\xi)$ of $A_1(t,\xi)$ as long as it satisfies
the symbol conditions $M(t,\xi) , M^{-1}(t,\xi)\in\mathcal S_*\{0,0\}$. Condition 
\eqref{eq:3:GECL-cond} does depend on the 
lower order terms $A(t,\xi)-A_1(t,\xi)$ modulo $\mathcal S_*\{-1,2\}$ and on the 
time-evolution of the system of eigenspaces of $A_1(t,\xi)$. 

We will give an example. If we consider the damped wave equation 
\begin{equation}\label{eq:3:damped}
    u_{tt} - \Delta u + 2b(t) u_t = 0,
\end{equation}
it can be rewritten as a system of first order in $(|\xi|\widehat u, \D_t \widehat u)^\top$
with coefficient matrix 
\begin{equation}
  A(t,\xi) = \begin{pmatrix}
        0 & |\xi| \\ |\xi| &2 \mathrm i b(t) 
  \end{pmatrix} \in \mathcal S_*\{1,0\}.
\end{equation}
Applying the diagonalisation scheme of Section~\ref{sec:3:diag} to this symbol yields terms $\mathcal D(t,\xi) = \diag(|\xi|,-|\xi|)\in\mathcal S_*\{1,0\}$ and $F_0(t,\xi) = F^{(0)}(t,\xi)
 = \mathrm i b(t) \mathrm I \in \mathcal S_*\{0,1\}$. Hence, the generalised energy conservation
 property holds for \eqref{eq:3:damped} if and only if the integral
\begin{equation}
    \int_0^t b(\tau)\mathrm d\tau
\end{equation} 
 remains bounded. In case that $b(t)\ge0$, this just means $b\in L^1(\mathbb R_+)$
 and solutions are even asymptotically free in this case. If the above integral is not uniformly
 bounded, we obtain the two-sided estimate
\begin{equation}
   |\mathcal E_k (t,s,\xi) | \approx \exp \left(\int_s^t b(\tau)\mathrm d\tau\right).
\end{equation}
This is consistent with the observations of Section~\ref{sec:2:HF} for the scale invariant model case. The equation \eqref{eq:3:damped}
was studied by the second author in \cite{Wirth:2006} under similar assumptions.

\subsubsection{Perturbation series arguments} For the following we assume that
\eqref{eq:3:GECL-cond} is valid such that in particular
\begin{equation}
    \| \mathcal E_k(t,s,\xi) \| \approx 1
\end{equation}
for all $k\ge 1$ uniformly on $(t,\xi),(s,\xi)\in\mathcal Z_{\rm hyp}(c_k)$. Then we can follow
a simpler argument to construct $\mathcal E(t,s,\xi)$ explicitly.

We make the ansatz 
\begin{equation}\label{eq:3:fund-sol-diag}
    \mathcal E_k(t,s,\xi) \mathcal Q_k(t,s,\xi)
\end{equation}
for the fundamental solution of the diagonalised system $\D_t - \mathcal D - F_{k-1} - R_k$
with a so far unkown matrix function $\mathcal Q_k(t,s,\xi)$. A short calculation shows
that the function $\mathcal Q_k(t,s,\xi)$ solves
\begin{equation}\label{eq:3:Qk-eq}
  \D_t   \mathcal Q_k(t,s,\xi) = \mathcal R_k(t,s,\xi) \mathcal Q_k(t,s,\xi) ,\qquad \mathcal Q_k(s,s,\xi)=\mathrm I,
\end{equation}
with
\begin{equation}\label{eq:3:calRk-def}
   \mathcal R_k(t,s,\xi) = \mathcal E_k(s,t,\xi) R_k(t,\xi) \mathcal E_k(t,s,\xi).
\end{equation}
We can represent the solution to \eqref{eq:3:Qk-eq} in terms of the Peano--Baker series
\begin{multline}\label{eq:3:Qk-series}
   \mathcal Q_k(t,s,\xi) = \mathrm I + \sum_{\ell=1}^\infty \mathrm i^\ell \int_s^t 
   \mathcal R_k(t_1,s,\xi) \int_s^{t_1} \mathcal R_k(t_2,s,\xi) \\\cdots \int_s^{t_{\ell-1}} \mathcal R_k(t_\ell,s,\xi)\mathrm dt_\ell \cdots \mathrm d t_2\mathrm dt_1.
\end{multline}
If $k\ge1$, the symbol estimate $R_k(t,\xi)\in\mathcal S_*\{-k,k+1\}$ allows to show the
boundedness of $\mathcal Q_k(t,s,\xi)$. Indeed, using $\|\mathcal R_k(t,s,\xi)\| \approx \|R_k(t,\xi)\|$ uniform in $s$ we obtain
\begin{equation}
   \| \mathcal Q_k(t,s,\xi) \| \le \exp\left( \int_s^t \| \mathcal R_k(\tau,s,\xi)\|\mathrm d\tau\right)
   \lesssim \exp(C c^{-k})
\end{equation}
uniformly in $\mathcal Z_{\rm hyp}(c)$.

We will derive one more consequence from the series representation. The uniform integrability
of $\mathcal R_k(t,s,\xi)$ over $\mathcal Z_{\rm hyp}(c_k)$ implies that the limit
\begin{equation}
    \lim_{t\to\infty} \mathcal Q_k(t,s,\xi) = \mathcal Q_k(\infty,s,\xi)
\end{equation}
exists locally uniform in $(s,\xi)\in \mathcal Z_{\rm hyp}(c_k)$. This is most easily seen by checking the Cauchy criterion or reducing it to the estimate
\begin{multline}
\| \mathcal Q_k(\infty,s,\xi) - \mathcal Q_k(t,s,\xi) \|    \lesssim
  \int_t^\infty \| R_k(\tau,\xi) \| \|\mathcal Q_k(\tau,s,\xi)\| \mathrm d\tau \\ \lesssim \int_t^\infty \| R_k(\tau,\xi) \| \mathrm d\tau. 
\end{multline}
The matrices $\mathcal Q_k(t,s,\xi)$ as well
as  $\mathcal Q_k(\infty,s,\xi)$ are invertible as the uniform lower bound on their determinant
\begin{multline}
    \det \mathcal Q_k(t,s,\xi) = \exp\left(\mathrm i\int_s^t \mathrm{trace}\, 
    \mathcal R_k(\tau,s,\xi)\mathrm d\tau \right)\\ \ge \exp \left( -d \int_s^t \| R_k(\tau,\xi) \| \mathrm d\tau \right) \ge C>0
\end{multline}
shows.

In the sequel we will also need estimates for $\xi$-derivatives of $\mathcal Q_k(t,s,\xi)$,
which can be obtained by differentiating the above representations provided $k$ is sufficiently large. There are some subtleties involved.
If we consider $\xi$-derivatives of $\mathcal E_k(t,s,\xi)$ then logarithmic terms might appear 
from integrals of $\xi$-derivatives of $F^{(0)}(t,\xi)\in\mathcal S_*\{0,1\}$, i.e., we get
\begin{equation}
     \| \D_\xi^\alpha \mathcal E_k(t,t_\xi,\xi) \| \lesssim (1+t)^{|\alpha|} \big(\log(\mathrm e+t)\big)^{|\alpha|}
\end{equation}
for $t_\xi$ solving $(1+t_\xi)|\xi|=c_k$ (or being equal to $0$ if this solution becomes negative). This implies that derivatives of $\mathcal R_k(t,s,\xi)$
defined in \eqref{eq:3:calRk-def} satisfy weaker estimates than the symbol estimates of
$R_k(t,\xi)$. Based on the definition of the hyperbolic zone and the above estimate we obtain
\begin{equation}
   \|\D_\xi^\alpha \mathcal R_k(t,t_\xi,\xi) \| \lesssim |\xi|^{-1-|\alpha|} 
        (1+t)^{-2} \big(\log(\mathrm e+t)\big)^{|\alpha|} 
\end{equation}
for $|\alpha|\le (k-1)/2$ and uniform in $\mathcal Z_{\rm hyp}(c_k)$. Differentiating the series 
representation \eqref{eq:3:Qk-series} term by term and using this estimate yields 
\begin{equation}\label{eq:3:der-est} 
   \| \D_\xi^\alpha \mathcal Q_k(t,t_\xi,\xi) \| \lesssim |\xi|^{-|\alpha|} \big(\log(\mathrm e+t)\big)^{|\alpha|}
\end{equation}
uniform in $\mathcal Z_{\rm hyp}(c_k)$ and with constants depending on $k$.

\subsection{Examples and resulting representations of solutions}
We will discuss some examples in more detail in order to show how to apply the previously developed theory. We will restrict our attention to homogeneous problems as in this case
the estimates within $\mathcal Z_{\rm pd}$ are easily obtained. In fact, assuming that 
$A(t,\xi)$ is positively homogeneous of order one
in $\xi$ yields $\|A(t,\xi) \| \le C|\xi|$ uniform in $t$ and, therefore, also directly the estimate
\begin{equation}\label{eq:3:Zpd-est-1}
   \|\mathcal E(t,0,\xi) \| \le \exp \left ( \int_0^t   \|A(\tau,\xi)\| \mathrm d\tau \right) \le
  \exp (C t |\xi|) \lesssim 1 ,\qquad (t,\xi)\in\mathcal Z_{\rm pd},
\end{equation}
for the fundamental solution defined by \eqref{eq:3:fundSol}. Furthermore, differentiating the differential equation \eqref{eq:3:fundSol} implies that derivatives satisfy
\begin{equation}\label{eq:3:Zpd-est}
\| \D_\xi^\alpha \mathcal E(t_\xi,0,\xi) \| \le C_\alpha |\xi|^{-|\alpha|} .
\end{equation}
In the general case, the treatment for small frequencies
needs more care and is often reduced to solving Volterra equations in suitable weighted 
spaces, see \cite{Wirth:2006}, \cite{HW:2009} or \cite{dAR:2011}. 

\subsubsection{Symmetric hyperbolic systems} First, we consider symmetric hyperbolic differential systems of the form \eqref{eq:3:diffHypSystem}, 
\begin{equation}\label{eq:3:expl0}
   \D_t U = \sum_{j=1}^n A_j(t) \D_{x_j} U,\qquad U(0,\cdot) = U_0, 
\end{equation}
for self-adjoint matrices $A_j(t) \in \mathcal T\{0\}$. By partial Fourier
transform this is related to $\D_t \widehat U = A(t,\xi) \widehat U$ with
\begin{equation}\label{eq:3:diffHypSystemA}
   A(t,\xi) = \sum_{j=1}^n A_j(t)\xi_j \in \mathcal S_*\{1,0\}.
\end{equation}
We assume that the roots of \eqref{eq:3:diffHypSystemRoots} satisfy
$$|\lambda_i(t,\xi) - \lambda_j(t,\xi) |\ge C|\xi|$$ 
uniformly in $t$ and $i\ne j$ for some constant $C>0$. Then we can follow the scheme of Sections~\ref{sec:3:diag} and \ref{sec:3:SolvDiag} to explicitly determine leading terms of the representation of solutions. 

As \eqref{eq:3:diffHypSystemA} defines a self-adjoint matrix, we choose the diagonaliser $M(t,\xi)$ as unitary matrix depending smoothly on $t$ and $\xi$. This gives $M(t,\xi) \in\mathcal S_*\{0,0\}$.
Furthermore, denoting the columns of the matrix $M(t,\xi)$ as $\nu_j(t,\xi)$, $j=1,\ldots ,d$,
we can express the first matrices constructed in the diagonalisation scheme as
$\mathcal D(t,\xi) = \diag ( \lambda_1(t,\xi) ,\ldots ,\lambda_d(t,\xi))$ and 
\begin{equation}
   \big( F^{(0)}(t,\xi)\big)_{jj} = -\mathrm i  \overline{\partial_t \nu_j(t,\xi)} \cdot \nu_j(t,\xi)
   = -\frac{\mathrm i}2 \partial_t  \|\nu_j(t,\xi)\|^2 = 0.
\end{equation}
The next diagonal terms starting with $F^{(1)}(t,\xi)$ are in general non-zero such that
$F_{k-1}(t,\xi)\in\mathcal S_*\{-1,2\}$ in place of the weaker result from Lemma~\ref{lem:3:diagLem}. In consequence, the estimates  \eqref{eq:3:der-est} are valid without the logarithmic terms.

Combining the representation \eqref{eq:3:fund-sol-diag} with the representation \eqref{eq:3:Ek-def} and the estimates  \eqref{eq:3:der-est}  and using the symbol properties of the diagonaliser $M(t,\xi)$ and $N_k(t,\xi)$ together with estimate \eqref{eq:3:Zpd-est} we obtain structural information about the representation of solutions. Note first that for any $t\ge t_\xi$, i.e., within the
hyperbolic zone $\mathcal Z_{\rm hyp}(c_k)$, the representation 
\begin{multline}
   \mathcal E(t,0,\xi) = M^{-1}(t,\xi) N_k^{-1}(t,\xi) \mathcal E_k(t,t_\xi,\xi) \mathcal Q_k(t,t_\xi,\xi)\\
   N_k(t_\xi,\xi) M(t_\xi,\xi) \mathcal E(t_\xi,0,\xi)
\end{multline}
holds true. Regrouping the expressions a bit, we obtain the following theorem. Note that the precise zone constants depend on the number of derivatives of $B_j(t,\xi)$ we have to estimate.

\begin{thm}\label{thm:3:sol-rep-systems}
Any solution to the system \eqref{eq:3:expl0} is representable in the form
\begin{equation}\label{eq:3:SolRep}
   \widehat U(t,\xi) = \sum_{j=1}^d \mathrm e^{\mathrm i t \vartheta_j(t,\xi)} B_j(t,\xi) \widehat U_0(\xi) + \mathcal E_{\rm pd} (t,\xi) \widehat U_0(\xi),
\end{equation}
where $\mathcal E_{\rm pd}(t,\xi)$ is uniformly bounded and supported within the zone $\mathcal Z_{\rm pd}(2c_k)$, the phase function is defined in terms of the homogeneous characteristic roots
$\lambda_j(t,\xi)$ via
\begin{equation}
   \vartheta_j(t,\xi) = \frac1t \int_0^t \lambda_j(\tau,\xi)\mathrm \mathrm d\tau
\end{equation}
and the amplitudes $B_j(t,\xi)$ are supported within $\mathcal Z_{\rm hyp}(c_k)$ and satisfy the 
symbolic estimates
\begin{equation}
    \| \D_\xi^\alpha B_j(t,\xi) \| \le C_\alpha |\xi|^{-|\alpha|}
\end{equation}
for all $|\alpha|\le (k-1)/2$.
\end{thm}
 
We omit an explicit proof as it is just a combination of the above mentioned estimates and representations. We only point out that one uses a variant of $\chi_{\mathrm hyp}(t,\xi)$ to cut off the zone $\mathcal Z_{\rm pd}(c_k)$ and to decompose the representation smoothly into parts.
 
\subsubsection{Second order equations} We consider first homogeneous second order equations
of the form
\begin{equation}\label{eq:3:expl1}
    \partial_t^2 u - \sum_{i,j=1}^n a_{ij}(t) \partial_{x_i} \partial_{x_j} u = 0,\qquad u(0,\cdot) = u_0,\quad u_t(0,\cdot)=u_1,
\end{equation}
for $a_{ij}(t)\in\mathcal T\{0\}$. The assumption of uniform strict hyperbolicity is equivalent to assuming that
\begin{equation}
   \sum_{i,j=1}^n a_{ij}(t) \xi_i\xi_j \ge C |\xi|^2
\end{equation}
for some constant $C>0$. Let $a(t,\xi)$ denote the (positive) square root of the left-hand side of this inequality. Then $a(t,\xi)\in\mathcal S_*\{1,0\}$ and if we 
denote $\widehat U = (a(t,\xi) \widehat u, \D_t \widehat u)^\top$, then equation \eqref{eq:3:expl1}
is reduced to the first order system
\begin{equation}\label{eq:3:3.78}
   \D_t \widehat U = A(t,\xi) \widehat U = \begin{pmatrix} \frac{\D_t a(t,\xi) }{a(t,\xi)} & a(t,\xi) \\ a(t,\xi) & 0 \end{pmatrix} \widehat U,\qquad \widehat U(0,\cdot) = \widehat U_0.
\end{equation}
The above equation and this system were treated by Reissig \cite{Reissig:2004} and in a similar setting Reissig--Yagdjian \cite{RY:2000}. 

Clearly $A(t,\xi)\in\mathcal S_*\{1,0\}$. Furthermore, a diagonaliser of the homogeneous principal part is given by 
\begin{equation} M = 
\begin{pmatrix} 1& 1\\  -1 & 1 \end{pmatrix}, \qquad M^{-1} = \frac12 \begin{pmatrix} 1 & 1  \\ -1 & 1 \end{pmatrix},
\end{equation}
so that $\mathcal D(t,\xi) = \diag\big(a(t,\xi), -a (t,\xi)\big)$ and 
$F^{(0)}(t,\xi) = \frac{\D_t a(t,\xi) }{2a(t,\xi)}  \mathrm I$. The matrix $\mathcal E_0(t,s,\xi)$ can be calculated explicitly as
\begin{multline}
   \mathcal E_0(t,s,\xi) = \exp\left( \mathrm i \int_s^t \big(\mathcal D(\tau,\xi) + F^{(0)}(\tau,\xi)\big) \mathrm d\tau \right)\\ =
   \frac{\sqrt{a(t,\xi)}}{\sqrt{a(s,\xi)}}
     \diag\left(\exp\left(\pm \mathrm i\int_s^t a(\tau,\xi) \mathrm d\tau\right)\right).
\end{multline}
As can be seen clearly from this representation, the exponential of the primitive of $F^{(0)}(t,\xi)$ is
a symbol from $\mathcal S_*\{0,0\}$ and the logarithmic terms appearing in 
\eqref{eq:3:der-est} do not occur. Furthermore, estimate \eqref{eq:3:Zpd-est-1}
applies for the fundamental solution associated to $(\widehat u, \D_t \widehat u)^\top$
and by equivalence of norms also to the one of \eqref{eq:3:3.78}.   Hence, we obtain a representation of solutions reminiscent to that of Theorem~\ref{thm:3:sol-rep-systems}.

\begin{thm}\label{thm:3:3.6}
Any solution to \eqref{eq:3:expl1} can be written as 
\begin{equation}\label{eq:3:sol-rep-expl2}
   \widehat u(t,\xi) = \sum_{j=0,1}  \left(  \sum_{\pm} \mathrm e^{\pm \mathrm i t \vartheta(t,\xi) } b_{j,\pm,\rm hyp}(t,\xi) +   b_{j,\rm pd}(t,\xi)\right ) \widehat u_j (\xi)
\end{equation}
with $b_{j,\rm pd} (t,\xi)$ uniformly bounded and supported within $\mathcal Z_{\rm pd}(2c_k)$, the phase function given in terms of $a(t,\xi)$ as
\begin{equation} 
  \vartheta(t,\xi) = \frac1t \int_0^t a(\tau,\xi) \mathrm d\tau
\end{equation}
and amplitudes $b_{j,\pm,\rm hyp}(t,\xi)$ supported within $\mathcal Z_{\rm hyp}(c_k)$ and satisfy
\begin{equation}
   |\D_\xi^\alpha b_{j,\pm,\rm hyp}(t,\xi) | \le C_\alpha |\xi|^{-|\alpha|-j}
\end{equation}
for $|\alpha|\le (k-1)/2$.
\end{thm}

\subsection{Dispersive estimates}\label{sec:3:dispersive}
In order to derive dispersive estimates from the results obtained so far,
we have to provide a suitable parameter dependent version of the appearing oscillatory
integrals. Such parameter dependent versions were developed in the setting of
equations of constant coefficients with lower order terms in \cite{Ruzhansky:2010}
based on the multi-dimensional van der Corput lemma in \cite{Ruzhansky:2009}.
For the presentation below we follow \cite{RW:2011} and \cite{RW:2011b}. 

\subsubsection{Contact indices for families of surfaces} In order to formulate estimates in a uniform way over a family of surfaces, we need some more notation generalising the treatment of Section~\ref{sec:1:StatPhas}. Let again $\Sigma$ be any closed smooth hypersurface in $\mathbb R^n$ and let $p\in\Sigma$ be any point. We translate and rotate the hypersurface in such a way that it can be parameterised as
\begin{equation}
   \{ (y, h(y)) :  y\in\Omega \}
\end{equation}
in a neighbourhood of the point $p$ for a suitable open set $\Omega\subset\mathbb R^{n-1}$. 
We start by considering the convex situation. 2-planes $H$ containing the normal $N_p\Sigma$ are determined by directions $\omega\in\mathbb S^{n-2}$ in $\Omega\subset\mathbb R^{n-1}$ and thus the 
contact orders $\gamma(\Sigma;p,H)$ correspond to the vanishing order of the function
$\rho \mapsto h(\rho\omega)$ at $\rho=0$. We need to measure this vanishing
in a more quantitative way and define for some given parameter $\gamma\in\mathbb N_{\ge2}$
\begin{equation}\label{eq:3:kappa-def}
   \varkappa(\Sigma,\gamma; p) = \inf_{|\omega|=1} \sum_{j=2}^{\gamma} 
   \left| \partial_\rho^j h(\rho\omega)|_{\rho=0} \right|.       
\end{equation}
From the definition of the contact index $\gamma(\Sigma)$ in \eqref{eq:1:SugInd}, it immediately follows 
that $\varkappa(\Sigma,{\gamma(\Sigma)};p)>0$ for all $p\in\Sigma$. Because the
function $p\mapsto\varkappa(\Sigma,\gamma;p)$ is continuous and $\Sigma$ closed by assumption, 
the minimum
\begin{equation}\label{eq:3:kappa-def2}
    \varkappa(\Sigma, \gamma) = \min_{p\in\Sigma} \varkappa(\Sigma,\gamma ; p)
\end{equation}
exists and is also strictly positive for $\gamma=\gamma(\Sigma)$.

If we find a uniform lower bound on $\varkappa(\Sigma_\lambda,\gamma)$ for a family of 
convex surfaces $\Sigma_\lambda$ depending on parameters $\lambda\in\Upsilon$ for some number $\gamma\ge2$ then we say that this family satisfies a uniformity condition. We further define the uniform contact index
\begin{equation}\label{eq:3:kappa-unif}
\gamma_{\rm unif} (\{\Sigma_\lambda : \lambda\in\Upsilon \}) = 
\min\{ \gamma\in\mathbb N_{\ge2} :  \inf_{\lambda\in\Upsilon} \varkappa(\Sigma_\lambda,\gamma) > 0 \}.
\end{equation}
For real parameters $t\in\mathbb R_+$ (which can be thought of as time for now) it makes sense to consider a weaker notion of asymptotic contact index and we define  
\begin{equation}\label{eq:3:kappa-as}
\gamma_{\rm as} (\{\Sigma_t : t\ge t_0 \}) = 
\min\{ \gamma\in\mathbb N_{\ge2} :  \liminf_{t\to\infty} \varkappa(\Sigma_t,\gamma) > 0 \}.
\end{equation}
In the particular case that $\Sigma_\lambda$ is a family of spheres of different radii $r_\lambda$, the constant $\varkappa(\Sigma_\lambda,2)\approx r_\lambda^{-1}$ is a measure of the curvature
of the sphere and we have $\gamma_{\rm unif}(\{\Sigma_\lambda\})=2$ if and only if $r_\lambda$ stays bounded. If on the other hand $r_\lambda$ is unbounded, the uniformity condition is not satisfied and no uniform contact order exists.

For non-convex surfaces we need to define a similar uniform non-convex contact index. We only explain the differences to the above formulae. First, we replace \eqref{eq:3:kappa-def}
by
\begin{equation}
  \varkappa_0(\Sigma,\gamma; p) =\sup_{|\omega|=1} \sum_{j=2}^{\gamma} 
   \left| \partial_\rho^j h(\rho\omega)|_{\rho=0} \right|.
\end{equation}
and then set $\varkappa_0(\Sigma,\gamma)=\min_{p\in\Sigma} \varkappa_0(\Sigma,\gamma;p)$.
Again $\varkappa_0(\Sigma,\gamma_0(\Sigma))>0$ and we use the analogues to
\eqref{eq:3:kappa-unif} and \eqref{eq:3:kappa-as} to define uniform and asymptotic contact orders.

If a family $\Sigma_t$ satisfies the uniformity condition, then the constants in the estimates
of Lemmata ~\ref{lem:1:sug1} and \ref{lem:1:sug2} are uniform over the family of surfaces. 

\subsubsection{Estimates for $t$-dependent Fourier integrals} Now we are in a position to discuss
Fourier integrals appearing in the representations \eqref{eq:3:SolRep} and \eqref{eq:3:sol-rep-expl2}, or, more generally, representations obtained by the construction of Section~\ref{sec:3:SolvDiag}. We omit the possibly occurring logarithmic terms here; they can be included easily  with a slight change in the
decay rates.
 
We consider a $t$-dependent family of operators
\begin{equation}
T_t :   u_0 \mapsto \int_{\R^n} \mathrm e^{\mathrm i (x\cdot\xi + t\vartheta(t,\xi) ) } a(t,\xi) \widehat u_0(\xi) \mathrm d\xi
\end{equation}
for a real-valued homogeneous phase function $\vartheta(t,\xi)\in C^\infty(\mathbb R_+\times (\mathbb R^n\setminus\{0\}))$ satisfying $\vartheta(t,\rho\xi) = \rho\vartheta(t,\xi)$ for $\rho>0$ and 
\begin{equation}\label{eq:3:phase-bound}
  C^{-1} |\xi| \le \vartheta(t,\xi) \le C|\xi|,\qquad |\D_\xi^\alpha \vartheta(t,\xi) | \le C_\alpha |\xi|^{1-|\alpha|},
\end{equation}
for some constants $C>0$, $C_\alpha$ and all $t\ge t_0$, and an amplitude $a(t,\xi)$ supported within $\mathcal Z_{\rm hyp}(c)$ and satisfying the symbol estimates
\begin{equation}\label{eq:3:amp-est}
   |\D_\xi^\alpha a(t,\xi)| \le C_\alpha |\xi|^{-|\alpha|}
\end{equation}
for a certain (finite) number of derivatives. Associated to the phase function we consider the family of slowness or Fresnel surfaces 
 \begin{equation}
    \Sigma_t = \{ \xi\in \mathbb R^n : \vartheta(t,\xi) =1\}.
 \end{equation}
Then the following theorem can be obtained by the method of \cite{RW:2011}.
\begin{thm}\label{thm:3:dispEsp}
Assume the family $\Sigma_t$ is convex for $t\ge t_0$, satisfies the uniformity condition and that $\gamma=\gamma_{\rm as}(\{\Sigma_t\})$ is its asymptotic contact index. Then the estimate
\begin{equation}
   \| T_t \|_{B^r_{1,2} \to L^\infty} \le C t^{-\frac{n-1}{\gamma} }
\end{equation}
holds true for $r=n-\frac{n-1}{\gamma}$.
\end{thm}

We remark that only finitely many derivatives of the phase (meaning also finite smoothness of $\Sigma_t$) and of the symbol $a(t,\xi)$ are needed to prove this statement.

\begin{proof} Using the definition of the Besov spaces from \eqref{eq:1:Besov} together with H\"older inequality we see that it suffices to prove
\begin{equation}\label{eq:3:3.92}
    \| T_t \circ \phi(2^{-j} |\D|) u_0 \|_{L^\infty} \lesssim 2^{j (n-\frac{n-1}\gamma)} t^{-\frac{n-1}\gamma} \| u_0\|_{L^1},
\end{equation}
i.e., corresponding bounds for the operator with amplitude $a(t,\xi) \phi(2^{-j}|\xi|)$
taking into account the desired Besov regularity. 
The uniform bound \eqref{eq:3:phase-bound} on the phase allows us to find a function $\psi\in C_0^\infty(\mathbb R_+)$ satisfying $\phi(|\xi|) \psi( \vartheta(t,\xi) ) = \phi(|\xi|)$. 
We denote by
\begin{equation} \label{eq:3:I-int}
   I_j(t,x) = \int \mathrm e^{\mathrm i(x\cdot \xi+ t\vartheta(t,\xi))} a(t,\xi) \psi(2^{-j} \vartheta(t,\xi)) \mathrm d\xi
\end{equation}
the (smooth!) convolution kernel of an operator related to \eqref{eq:3:3.92} and it is 
sufficient to prove the uniform bound
\begin{equation}\label{eq:3:3.93}
   \sup_{x\in\mathbb R^n}  |I_j(t,x)| \lesssim 2^{j (n-\frac{n-1}\gamma)} t^{-\frac{n-1}\gamma},
   \qquad j\ge 1,
\end{equation}
in combination with a related low-frequency estimate.
 The bound in $I_j(t,x)$ can be achieved by the 
stationary phase method in combination with the multi-dimensional van der Corput lemma
due to \cite{Ruzhansky:2009} . 

Stationary points of the phase are solutions to $x+t\nabla_\xi \vartheta(t,\xi) =0$. We use a cut-off function $\chi\in C_0^\infty(\mathbb R^n)$ with $\chi(x)=1$ for small $|x|$ and decompose the 
integral \eqref{eq:3:I-int} into two terms
\begin{align}
I_j^{(1)}(t,x) &= \int \mathrm e^{\mathrm i(x\cdot \xi+ t\vartheta(t,\xi))} \chi(x/t + \nabla_\xi\vartheta(t,\xi)) a(t,\xi) \psi(2^{-j} \vartheta(t,\xi))\mathrm d\xi, \\
I_j^{(2)}(t,x) &= \int \mathrm e^{\mathrm i(x\cdot \xi+ t\vartheta(t,\xi))} (1-\chi(x/t + \nabla_\xi\vartheta(t,\xi)))a(t,\xi) \psi(2^{-j} \vartheta(t,\xi)) \mathrm d\xi.
\end{align} 
Based on $|x+t\nabla_\xi\vartheta(t,\xi) |\gtrsim t$, the second integral can be treated by 
integration by parts giving 
\begin{equation}
| I_j^{(2)}(t,\xi) | \le C_N t^{-N} 2^{j(n-N)}
\end{equation}
uniform in $j$ and $t$ and for any number $N$.\footnote{Of course, $N$ depends on the number of derivatives we can estimate by \eqref{eq:3:amp-est}. The  number that we need is $\lceil (n-1)/\gamma\rceil$.} For the first integral we use the structure of the level sets $\Sigma_t$
and restrict consideration to large values of $t$. We localise the integral into narrow cones, by translation and rotation we can assume that $\xi$ is within a sufficiently small conical neighbourhood of $(0,\ldots,0,1)^\top$ and $\Sigma_t$ can be parameterised as
$\{(y, h_t(y)) : y\in\ U\}$ with $h$ vanishing to at least second order in $0$. The function
$\nabla h_t : U \to \nabla h_t (U) \subset \mathbb R^{n-1}$ is a homeomorphism and $h_t$ is 
concave as $\Sigma_t$ was assumed to be convex. Furthermore, we have
\begin{equation}\label{eq:3:ht-bound}
    | \D_y^\alpha h_t(y)| \le C_\alpha,\qquad t\ge t_0,
\end{equation}
uniform on $U$ if $U$ is small enough. Indeed, as $\vartheta(t,y,h_t(y))=1$ we obtain by differentiation 
$\nabla_y \vartheta + \partial_{\xi_n} \vartheta \nabla h_t = 0$ and by Euler's identity
$\partial_{\xi_n}\vartheta(t,e_n)=\vartheta(t, e_n)$. Therefore, the bound $|\nabla_y\vartheta(t,y,h)|\lesssim 1$ implies the bound on $\nabla h_t(y)$ uniform in $t$. Higher order derivatives follow in analogy.

Associated to $\Sigma_t$ we have the Gauss map
\begin{equation}
   G : \Sigma_t \ni \xi \mapsto \frac{\nabla_\xi\vartheta(t,\xi)}{|\nabla_\xi\vartheta(t,\xi)|}\in\mathbb S^{n-1}
\end{equation}
and for given $x=(x',x_n)$ near $-t \nabla_\xi \vartheta(t,e_n)$ we define
$z_t\in U$ by $G(z_t, h_t(z_t))= -x / |x|$. 
Making the change of variables $\xi=(r y,r h_t(y))$ with $\vartheta(t,\xi)=r\approx 2^j$, we get that the localised part of $I_j^{(1)}(t,x)$ equals
\begin{multline}
 \int_0^\infty \int_U \mathrm e^{\mathrm i r( x'\cdot y + x_n h_t(y) + t)}
   a(t,r y, r h_t(y)) \psi(2^{-j}r) \tilde \chi(t,x,y) \kappa(t,r,y) \mathrm dy\mathrm dr\\
  = 2^{jn} \int_0^\infty \int_U \mathrm e^{\mathrm i 2^j r( x'\cdot y + x_n h_t(y) + t)}
   \tilde a_j(t,r,y)\tilde\chi(t,x,y) \mathrm dy \mathrm dr
\end{multline}
with $\tilde\chi(t,x,y) = \chi(x/t+\nabla_\xi \vartheta(t,y,h_t(y))$,  $\kappa(t,r,y) = | \partial \xi / \partial (r,y)|$ the Jacobi determinant of the $t$-dependent change of variables, and
\begin{equation}
 \tilde a_j(t,r,y) = \psi(r)  a(t, 2^j r y, 2^j r h_t(y)) \kappa(2^j r,y )
\end{equation}
satisfying uniform bounds of the form
\begin{equation}
  | \D_y^\alpha \tilde a(t,r,y) | \le C_\alpha. 
\end{equation}
Here we used the symbol estimates of $a$ in combination with \eqref{eq:3:ht-bound} and also
that the latter imply uniform bounds on $y$-derivatives of the Jacobian. Note, that $r$-derivatives
of the Jacobian have symbolic behaviour by homogeneity. The stationary phase estimate \cite[Theorem 2.1]{Ruzhansky:2009} is applied to the integral over $U$, which gives
\begin{equation}
\left|  \int_U \mathrm e^{\mathrm i 2^j r( x'\cdot y + x_n h_t(y) + t)}
   \tilde a_j(t,r,y)\tilde\chi(t,x,y) \mathrm dy \right| \le C | 2^j x_n |^{-\frac{n-1}\gamma}
\end{equation}
with a constant $C$ uniform in the remaining variables. Due to the localisation we have
that  $|x|\approx |x_n|\approx t$ and therefore we obtain the desired bound \eqref{eq:3:3.92}.

The remaining low-frequency estimate follows in an analogous manner, omitting the dyadic decomposition and using the relations $t|\xi|\gtrsim 1$ combined with the symbolic estimates
and the idea that uniformly bounded Fourier multipliers supported within $\mathcal Z_{\rm pd}$
always lead to $t^{-n}$ decay rates (which are much faster).  
\end{proof}

If the convexity assumption is dropped, decay rates get much worse but on the other hand we also need less regularity in the estimate. We only formulate the theorem.
\begin{thm}\label{thm:3:dispEsp-2}
Assume the family $\Sigma_t$ satisfies the uniformity condition and that $\gamma_0=\gamma_{0,\rm as}(\{\Sigma_t\})$ is its non-convex asymptotic contact index. Then the estimate
\begin{equation}
   \| T_t \|_{B^r_{1,2} \to L^\infty} \le C t^{-\frac{1}{\gamma_0} }
\end{equation}
holds true for $r=n-\frac{1}{\gamma}$.
\end{thm}

\subsubsection{Extensions to fully variable setting} We conclude Section~\ref{sec:3} with some more
general estimates.  They apply to global Fourier integral operators  of a very particular structure at 
infinity and are essentially based on the observation that the variables $t$ and $x$ appear parametric in dispersive estimates. They also fit to the earlier observation that $t\xi$ is a natural co-variable when describing solutions to free wave equations and equations with weak dissipation.

We consider the operator
\begin{equation}
    T_t : u_0 \mapsto \int_{\R^n} \mathrm e^{\mathrm i (x\cdot \xi + t \vartheta(t,x,\xi))} a(t,x,\xi) \widehat u_0(\xi)\mathrm d\xi
\end{equation}
for a homogeneous real phase function $\vartheta(t,x,\xi) \in C^\infty (\mathbb R_+ \times \mathbb R^n \times (\mathbb R^n\setminus\{0\}))$ with $\vartheta(t,x,\rho\xi) = \rho \vartheta(t,x,\xi)$  for $\rho>0$ and
\begin{equation}
   C^{-1} |\xi| \le \vartheta (t,x,\xi) \le C|\xi|,\qquad |\D_\xi^\alpha \vartheta(t,x,\xi)| \le |\xi|^{1-|\alpha|},
\end{equation}
 and the amplitude $a(t,x,\xi)$ supported in $\{ (t,x,\xi) : (1+t)|\xi| \ge c\}$ for some constant $c$ 
 and subject to 
 \begin{equation}
   |\D_\xi^\alpha a(t,x,\xi) | \le C_\alpha |\xi|^{-|\alpha|}.
 \end{equation}
A major difference to the previous considerations is that we now have a family of surfaces parameterised by $t$ and $x$
\begin{equation}
  \Sigma_{t,x} = \{ \xi\in\mathbb R^n : \vartheta(t,x,\xi) =1\}
\end{equation}
and we have to use a contact index defined asymptotic with respect to $t$ and uniform in $x$.
Let therefore $\gamma=\gamma_{\rm as,unif}(\{\Sigma_{t,x}\})$ with
\begin{equation}\label{eq:3:as-unif-CI}
  \gamma_{\rm as,unif} (\{\Sigma_{t,x}\})  = \min\{\gamma\in\mathbb N_{\ge 2} :
  \liminf_{t\to\infty} \inf_{x\in\mathbb R^n} \varkappa(\Sigma_{t,x},\gamma)>0\}
\end{equation}
and we assume that all the surfaces are convex. Then the proof of Theorem~\ref{thm:3:dispEsp}
carries over and gives
\begin{equation}
  \| T_t u_0\|_{L^\infty} \lesssim t^{-\frac{n-1}\gamma} \|u_0\|_{B^r_{1,2}}
\end{equation}
for $r\ge n-\frac{n-1}\gamma$. A similar replacement works for the non-convex situation with
the weaker decay rate $t^{-1/\gamma_0}$ for the non-convex asymptotic uniform contact index
defined in analogy to \eqref{eq:3:as-unif-CI}.

\subsection{An alternative low-regularity approach: asymptotic integration}\label{AI}
We will conclude Section~\ref{sec:3} with a short review and reformulation of recent results due to Matsuyama--Ruzhansky \cite{MR:2010}, which allow for very low regularity of the coefficients. The approach originates in asymptotic integration arguments due to Wintner \cite{Wintner:1947}. The price to pay is that results are not optimal with respect to regularity and they do also not reach the critical threshold for lower order terms.  

We will formulate them for hyperbolic systems as it is in the context of the
present exposition. The case of scalar equations has been treated in
\cite{MR:2010} while the case of systems can be found in
\cite{MR:2011}. To be precise, we consider the Cauchy problem
\begin{equation}\label{eq:CP:3:5}
   \D_t U = A(t,\D) U, \qquad U(0,\cdot)=U_0,
\end{equation}
for initial data $U_0\in L^2(\R^n;\mathbb C^d)$ and with symbol split as $A(t,\xi) = A_1(t,\xi) + A_0(t,\xi)$
into a homogeneous strictly hyperbolic part
\begin{equation}
   A_1(t,\rho\xi) = \rho A_1(t,\xi) ,\qquad \forall \rho>0,
\end{equation}
having real eigenvalues $\mathrm{spec}\, A_1(t,\xi) = \{ \lambda_1(t,\xi),\ldots ,\lambda_d(t,\xi) \} \subset\R$ 
uniformly separated
\begin{equation}
   \inf_{t,\xi\not=0} \frac{|\lambda_i(t,\xi)-\lambda_j(t,\xi)|}{|\xi|} >0
\end{equation}
for all $i\ne j$. Difference to the previous considerations is that we do {\em not} assume smoothness with respect to $t$, instead we assume that the symbol is of bounded variation and satisfies
\begin{equation}\label{eq:3.4:cond}
  |\xi|^{-1}  \partial_t A_1(\cdot,\xi) \in L^1(\R_+; \mathbb C^{d\times d})  , \qquad A_0(\cdot,\xi) \in L^1(\R_+; \mathbb C^{d\times d}) 
\end{equation}
uniform in $\xi\not=0$. 

These assumptions have some almost immediate consequences.
\begin{lem}
\begin{enumerate}
\item $|\xi|^{-1} \partial_t \lambda_j(t,\xi) \in L^1(\R_+)$ uniform in $\xi\not=0$.
\item There exists a matrix $N(t,\xi)$ diagonalising $A_1(t,\xi)$
\begin{equation}\label{eq:3.4:diag}
  A_1(t,\xi)N(t,\xi) = N(t,\xi) \diag\big(\lambda_1(t,\xi),\dots,\lambda_d(t,\xi)\big),\qquad \xi\ne0,
\end{equation}
with $\|N(t,\xi)\|, \|N^{-1}(t,\xi)\|$ uniformly bounded and $\partial_t N(t,\xi) \in L^1(\R_+)$ uniform in $\xi\ne0$.
\end{enumerate}
\end{lem}
\begin{proof}[Sketch of proof]
The second statement is analogous to \cite[Prop. 6.4]{Mizohata:1973} and also to Lemma~\ref{lem:3:SymbEstDiag1}.3, while the first statement follows
from differentiating \eqref{eq:3.4:diag} with respect to $t$.
\end{proof}

We can use $N(t,\xi)$ to transform the system into a diagonal dominated form. Considering $V(t,\xi) = N^{-1}(t,\xi) \widehat U(t,\xi)$ we obtain the
equivalent system
\begin{equation}
   \D_t V = \big( \mathcal D(t,\xi) + R(t,\xi) \big) V,
\end{equation}
with $\mathcal D(t,\xi) = \diag\big(\lambda_1(t,\xi),\dots,\lambda_d(t,\xi)\big)$ and 
\begin{equation}
  R(t,\xi) = N^{-1}(t,\xi) A_0(t,\xi) N(t,\xi) + \big(\D_t N^{-1}(t,\xi)\big) N(t,\xi).
\end{equation}
By assumption on $A_0(t,\xi)$ and construction of $N(t,\xi)$ it follows that $R(\cdot,\xi)\in L^1(\R_+)$ uniform in $\xi$ and we can solve the resulting system
by an argument similar to Section~\ref{sec:3:SolvDiag}. The only difference is that we obtain just uniform results with respect to $\xi$ and have no information about the behaviour of derivatives with respect to $\xi$.

\begin{thm}
Solutions to \eqref{eq:CP:3:5} can be written as
\begin{equation}   
   \widehat U(t,\xi) = \sum_{j=1}^d \mathrm e^{\mathrm i t\vartheta_j(t, \xi) } B_j(t,\xi) \widehat U_0(\xi)
\end{equation}
with phase functions given in terms of the characteristic roots $\lambda_j(t,\xi)$,
\begin{equation}
   \vartheta_j(t,\xi) = \frac1t \int_0^t \lambda_j(\theta,\xi) \mathrm d\theta,
\end{equation}
and amplitudes $B_j(t,\xi) \in L^\infty(\R_+\times\R^n)$ satisfying
\begin{equation}
   B_j(t,\xi) = {\boldsymbol\alpha}_j(\xi) + {\boldsymbol\varepsilon}_j(t,\xi),\qquad \lim_{t\to\infty} {\boldsymbol\varepsilon}_j(t,\xi)= 0
\end{equation}
locally uniform in $\xi\ne0$.
\end{thm}
%

If one has additional information on the behaviour of $\xi$-derivatives of the symbols $A_k(t,\xi)$, $k=1,2$, one can prove better estimates
and in particular also obtain symbolic type behaviour of the amplitudes $B_j(t,\xi)$. We give only one corresponding
statement. If $A_0(t,\xi)=0$ and $A_1(t,\xi)$ satisfies
\begin{equation}
    |\xi|^{|\alpha|-1}  \D_\xi^\alpha \partial_t A(\cdot,\xi) \in L^1(\R_+) 
\end{equation}
uniformly in $\xi\ne0$, then one can show that
\begin{equation}
  \| \D_\xi^\alpha B_j(t,\xi) \| \le C_\alpha, \qquad \xi\ne0,
\end{equation}
and  dispersive estimates based on Theorems~\ref{thm:3:dispEsp} and \ref{thm:3:dispEsp-2}  the (explicitly known) finite number of derivatives
needed there can be obtained by requiring a higher regularity of initial data. We refer to \cite{MR:2010, MR:2011}
for further details.

\section{Effective lower order perturbations}\label{sec:4}
If lower order terms are too large to be controlled, it becomes important to investigate the behaviour 
of solutions for bounded frequencies. We will restrict ourselves to situations where an asymptotic construction for $\xi\to0$ becomes important and provide some essential estimates for this.

\subsection{The diffusion phenomenon}\label{sec:4.1}
 The classical diffusion phenomenon gives an asymptotic
equivalence of damped wave equations and the heat equation. It was first observed for porous media type equations and its formulation for the damped wave equation is due to
Nishihara \cite{Nishihara:1997}, \cite{Ikehata:2003c} and independently Han--Milani
\cite{Milani:2001} in various formulations. The estimates of Nishihara were extended to
arbitrary dimensions by Narazaki \cite{Narazaki:2004} and provide $L^p$--$L^q$ type estimates
for differences of solutions to damped wave and heat equation. 

We will follow a different line of thought here and provide estimates in the flavour of 
Radu--Todorova--Yordanov \cite{RTY:2011}. They are only energy type estimates and therefore
easier to obtain. On the other hand, they are flexible enough to be formulated for abstract Cauchy problems in Hilbert spaces based on the continuous spectral calculus for self-adjoint operators.

Let in the following $u(t,x)$ be the solution to the Cauchy problem for the damped wave equation
\begin{equation}\label{eq:4:CP-damped}
   u_{tt}-\Delta u + u_t = 0,\qquad u(0,\cdot) = u_0, \quad u_t(0,\cdot)=u_1,
\end{equation}
and similarly $v(t,x)$ be the solution to the Cauchy problem for the heat equation
\begin{equation}\label{eq:4:CP-heat}
   v_t = \Delta v, \qquad v(0,\cdot) = v_0 = u_0+u_1
\end{equation} 
with related data. Then their solutions are related due to the following theorem. Note that in general 
\begin{equation}
  \| u(t,\cdot) \|_{L^2} \le C \big( \|u_0\|_{L^2} + \|u_1\|_{H^{-1}}\big),\qquad \|v(t,\cdot)\|_{L^2} \le C \|v_0\|_{L^2}
\end{equation}
are the best possible (operator-norm) estimates for solutions of both problems. This can be seen similar to the discussion in Section~\ref{sec:1:constCoeff} using explicit representations of Fourier multipliers. However, forming the difference of the solutions improves estimates by one order.
We denote by $\mathrm e^{t\Delta}$ the heat semigroup.

\begin{thm}\label{thm:4:diff-Phen}
Let $u(t,x)$ and $v(t,x)$ be solutions to \eqref{eq:4:CP-damped} and \eqref{eq:4:CP-heat}, respectively. Then
\begin{enumerate}
\item the difference of the solutions satisfies
\begin{multline}
\| u(t,\cdot) - v(t,\cdot) \|_{L^2} \lesssim t^{-1} \left( \| \mathrm e^{t\Delta/2} u_0\|_{L^2} + \| \mathrm e^{t\Delta/2} u_1\|_{L^2}  \right) \\ + \mathrm e^{-t/16} \left( \|u_0\|_{L^2} + \|u_1\|_{H^{-1}}\right)
\end{multline}
for all $t\ge 1$;
\item moreover, for all $k\in\mathbb N$ and 
$\alpha\in\mathbb N_0^n$ the higher order estimate
\begin{multline}
\| \D_t^k \D_x^\alpha\big( u(t,\cdot) - v(t,\cdot)\big) \|_{L^2} \lesssim t^{-1-k-|\alpha|/2} \left( \| \mathrm e^{t\Delta/2} u_0\|_{L^2} + \| \mathrm e^{t\Delta/2} u_1\|_{L^2}  \right) \\ + \mathrm e^{-t/16} \left( \|u_0\|_{H^{k+|\alpha|}} + \|u_1\|_{H^{k+|\alpha|-1}}\right)
\end{multline}
holds true for $t\ge 1$.
\end{enumerate}   
\end{thm}
\begin{proof} As the theorem extends the results of \cite{RTY:2011} we include a sketch of the proof. Fourier transform relates the Cauchy problems to initial value problems for ordinary 
differetial equations which can be solved explicitly. The solution of the parabolic problem
$\widehat v(t,\xi)$ is given by
\begin{equation}
   \widehat v(t,\xi) = \mathrm e^{-t|\xi|^2}\big( \widehat u_0(\xi) + \widehat u_1(\xi)\big),
\end{equation}
while the solution to the damped wave equation can be written as 
\begin{equation}
   \widehat u(t,\xi) = \mathrm e^{-t/2}         \big( \cos(t\sqrt{|\xi|^2 -1/4} ) \widehat u_0 + 
   \frac{\sin(t \sqrt{|\xi|^2- 1/4})}{\sqrt{|\xi|^2-1/4}} \widehat u_1(\xi)\big) 
\end{equation}
for $|\xi| >\frac12$ and similarly
\begin{equation}
   \widehat u(t,\xi) = \mathrm e^{-t/2}         \big( \cosh(t\sqrt{1/4-|\xi|^2 } ) \widehat u_0 + 
   \frac{\sinh(t \sqrt{1/4-|\xi|^2})}{\sqrt{1/4-|\xi|^2}} \widehat u_1(\xi)\big) 
\end{equation}
for $|\xi| < \frac12$. As we are concerned with polynomial estimates, we can neglect all
exponentially decaying terms (and collect them in the $\mathrm e^{-t/16}$ estimate later on).
This happens for $\widehat v(t,\xi)$ if $|\xi|\ge 1/4$, for $\widehat u(t,\xi)$ for $|\xi|>1/2$ 
and for the exponentially decaying terms defining $\cosh$ and $\sinh$. What we are left with 
are the terms
\begin{align}
   v(t,x) : \qquad & \mathrm e^{-t |\xi|^2} \big( \widehat u_0 (\xi) + \widehat u_1(\xi) ) \\
   u(t,x) : \qquad &  \mathrm e^{-t/2 + t \sqrt{1/4 - |\xi|^2}} \big( C_0(\xi) u_0(\xi) + C_1(\xi) u_1(\xi) 
   \big)\label{eq:4:u-lf-exp} 
\end{align}
locally near $\xi=0$ and with radially symmetric smooth functions $C_0(\xi)$ and $C_1(\xi)$ satisfying $C_0(0)=C_1(0)=1$. Our first observation concerns the exponent in \eqref{eq:4:u-lf-exp}. It satisfies
\begin{equation}
    \sqrt{1/4 - |\xi|^2}  = 1/2 - |\xi|^2 - \sum_{k=2}^\infty b_k |\xi|^{2k},\qquad b_k > 0,
\end{equation}
so that the main terms of the difference $\widehat u(t,\xi)-\widehat v(t,\xi)$ are a sum of terms of the form
\begin{align}
   \mathrm e^{-t|\xi|^2} \left( \mathrm e^{- t b_2 |\xi|^4 + t \mathcal O(|\xi|^6) } -1 \right) (\widehat u_0(\xi) + \widehat u_1(\xi)), \\
   \mathrm e^{-t |\xi|^2 + t \mathcal O(|\xi|^4) } \big(( C_0(\xi)-1) \widehat u_0(\xi) + ( C_1(\xi)-1) \widehat u_1(\xi) \big).
\end{align}
The first statement is proven, if we can bound these two terms by the ones on the right hand side; i.e., we are looking for uniform bounds for the multipliers
\begin{equation}
   t \mathrm e^{-t|\xi|^2/2 }  \left( \mathrm e^{- b_2 t |\xi|^4 + t\mathcal O(|\xi|^6)  } -1 \right) 
\end{equation}
and 
\begin{equation}
   t   \mathrm e^{-t |\xi|^2/2 + t\mathcal O(|\xi|^4) }( C_j(\xi)-1) ,\qquad j=0,1.
\end{equation} 
Using the elementary estimates $|\mathrm e^{-s}-1| \le s$ and $s\mathrm e^{-s} \lesssim1$ for $s\ge0$ we conclude that the first multiplier is bounded by
 $t^2 |\xi|^4 \exp(-t|\xi|^2 / 2 )\lesssim1$, while
the second multiplier can be estimated by $|\xi|^2 t \mathrm e^{-t|\xi|^2}\lesssim 1$ 
and the desired statement follows.

For the second statement we observe two things. First, when estimating $x$-derivatives this just gives additional factors of $\xi$ for small frequencies and the above estimates improve by corresponding $t^{-1/2}$ factors for each derivative. Similarly, when considering $t$-derivatives the representations get one additional factor of $|\xi|^2$ for small $\xi$ and each $t$-derivative
improves the estimate by a factor of $t^{-1}$.
\end{proof}

\subsection{Diagonalisation for small frequencies}\label{sec:4.2} 
Versions of the diffusion phenomenon can also be obtained in a variable coefficient setting.
We will restrict our consideration here mainly to $t$-dependent hyperbolic systems, for the sake
of brevity even differential hyperbolic systems of the form 
\begin{equation}\label{eq:4:CP}
   \D_t U = \sum_{k=1}^n A_k(t)\D_{x_k} U + \mathrm i B(t) U,\qquad U(t,\cdot) = U_0.
\end{equation}
Main difference to the considerations in Section~\ref{sec:3} is that we assume now that
$B(t)$ is {\em not} of lower order in the $\mathcal T$-hierarchy, i.e., we assume that
\begin{equation}
   A_k(t), B(t) \in \mathcal T\{0\} .
\end{equation}
We will make three main assumptions here;
\begin{description}
\item[(B1)] the matrices $A_k(t)$  are self-adjoint and $B(t)\ge0$;
\item[(B2)] the matrix $B(t)$ has $d$  eigenvalues 
$0=\delta_1(t) < \delta_2(t) <\dots < \delta_d(t)$
satisfying 
\begin{equation}
   \liminf_{t\to\infty}  |\delta_i(t) - \delta_j(t) | > 0,\qquad i\ne j;
\end{equation}
\item[(B3)] the matrices $A(t,\xi) = \sum_{k=1}^n A_k(t)\xi_k$ and $B(t)$ satisfy for all $v\in\mathbb C^d$
\begin{equation}
  \frac1c \| v\|^2 \le \sum_{j=1}^{d-1} \epsilon_j \| B(t) (A(t,\xi))^j v\|^2 \le c\|v\|^2,\qquad t\ge t_0,
\end{equation}
for any choice of numbers $\epsilon_0,\ldots,\epsilon_{d-1}>0$ and suitable constants $c$ and $t_0$ depending on them.
\end{description}
Assumption (B1) guarantees that the system is symmetric hyperbolic and (partially) dissipative.  Therefore,
the energy estimate 
\begin{equation} 
        \|U(t,\cdot)\|_{L^2}\le \|U_0\|_{L^2}
\end{equation}
is valid.  By Assumption (B2) we know that one mode is not dissipated, while Assumption (B3) will be used to show that
 the high frequency parts of solutions are still exponentially decaying. Inspired by Beauchard--Zuazua \cite{BZ:2011}, we will refer to (B3) as uniform Kalman rank 
 condition. If $A_k$ and $B$ are independent of $t$ it just means that
 \begin{equation}
    \rank\big ( B \big  |  A(\xi) B \big | \cdots \big |  A(\xi)^{d-1}B \big) = d,
 \end{equation}
 which is the classical Kalman rank condition arising in the control theory of ordinary differential systems. Under certain natural assumptions, this is equivalent to 
 the algebraic condition of Kawashima--Shizuta \cite{KS:1985}, but the latter are more complicated to rewrite uniformly depending on parameters.
 
 Our aim is 
 to understand the large-time behaviour of small frequencies here.

\begin{figure}
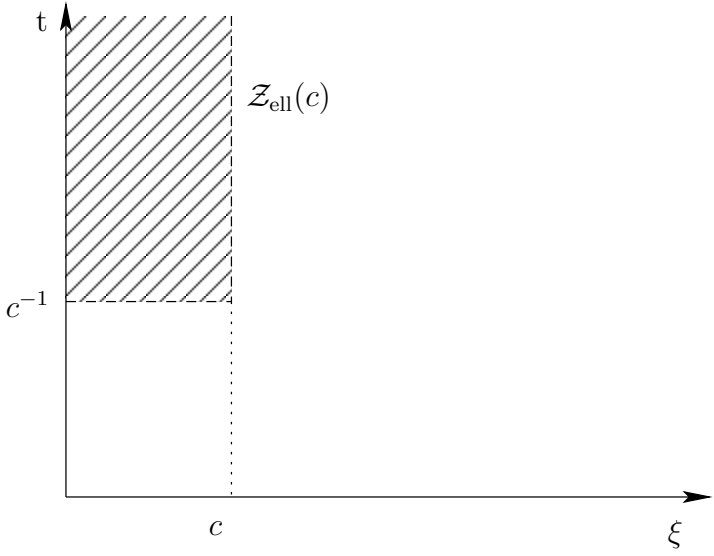
\caption{Zones for the small-frequency diagonalisation}
\end{figure}

For small $\xi$ we use a diagonalisation procedure similar to the one from Section~\ref{sec:3:diag}. It works within the elliptic zone $\mathcal Z_{\rm ell} (c) =\{(t,\xi) : |\xi| \le c, t\ge c^{-1}\}$
collecting small frequencies for large times. All considerations will be done for $t\ge t_0$ with
$t_0$ sufficiently large. This ensures uniform diagonalisability of $B(t)$.

\subsubsection{Initial step} We denote by $M(t)$ a diagonaliser of $B(t)$ satisfying the uniform bounds $M(t), M^{-1}(t)\in\mathcal T\{0\}$. Denoting $V^{(0)}(t,\xi) = M^{-1}(t) \widehat U(t,\xi)$, we obtain the equivalent system
\begin{equation}
   \D_t V^{(0)} = \left( \mathcal D(t) + \sum_{k=1}^n M^{-1}(t) A_k(t) M(t) \xi_k +\big( \D_t M^{-1}(t)\big) M(t) \right) V^{(0)}
\end{equation}
with $\mathcal D(t) =\mathrm i \diag(\delta_1(t), \ldots, \delta_d(t))$. We denote the two non-diagonal
terms by
\begin{equation}
    R_1(t,\xi) = \sum_{k=1}^n M^{-1}(t) A_k(t) M(t) \xi_k + \big( \D_t M^{-1}(t)\big) M(t), 
\end{equation}
remarking in particular that the first term is linear in $\xi$ with coefficients in 
$\mathcal T\{0\}$, while
the second one is independent of $\xi$, but in the better class $\mathcal T\{1\}$. Both are small 
compared to the difference $\delta_i(t)-\delta_j(t)$ provided $|\xi|$ is small and $t$ is large.

In the sequel we will use the notation
\begin{equation}
     \mathcal P\{m\} = \left\{ p(t,\xi) =  \sum_{|\alpha|\le m} p_{\alpha}(t) \xi^\alpha : p_\alpha (t) \in \mathcal T\{m-|\alpha|\} \right\}
\end{equation}
for polynomials with coefficients in the $\mathcal T$-classes. By construction,
$R_1(t,\xi)\in\mathcal P\{1\}$.

\subsubsection{The diagonalisation hierarchy} We start with a system
\begin{equation}
   \D_t V^{(0)} = \big(\mathcal D(t)  +  R_1(t,\xi) \big) V^{(0)} 
\end{equation}
with $P_1(t,\xi) \in \mathcal P\{1\}$. Before setting up the complete hierarchy, we will discuss its first step. Similar to Section~\ref{sec:3:diag} we construct a matrix $N^{(1)}(t,\xi)\in{\mathcal P}\{1\}$ such that
\begin{multline}
    \big(\D_t -\mathcal D(t) - R_1(t,\xi)\big) (\mathrm I + N^{(1)}(t,\xi))
    \\- (\mathrm I + N^{(1)}(t,\xi))  \big(\D_t -\mathcal D(t,\xi) - F_1(t,\xi)\big) \in{\mathcal P}\{2\}
\end{multline}
holds true for some diagonal matrix $F_1(t,\xi)\in{\mathcal P}\{1\}$. 
Collecting all terms not belonging to the right class yields again conditions for the
matrices $N^{(1)}(t,\xi)$ and $F_1(t,\xi)$. Indeed, 
\begin{equation}
  [\mathcal D(t) , N^{(1)}(t,\xi) ] = - R_1(t,\xi) + F_1(t,\xi) 
\end{equation}
must be satisfied and, therefore, we have 
\begin{equation}
   F_1(t,\xi) = \diag R_1(t,\xi), \qquad \left( N^{(1)}(t,\xi) \right)_{i,j} = \frac{\big(R_1(t,\xi)\big)_{i,j}}{\delta_i(t)-\delta_j(t)}, 
\end{equation}
while we may again choose diagonal entries to be $(N^{(1)}(t,\xi))_{i,i}=0$. As desired, this implies
$N^{(1)}(t,\xi), F_1(t,\xi)\in{\mathcal P}\{1\}$. 

Recursively, we will construct matrices $N^{(k)}(t,\xi)\in \mathcal P\{k\}$ and $F^{(k)}(t,\xi)
\in \mathcal P\{k\}$ diagonal, such that for
\begin{equation}
   N_K(t,\xi) = \mathrm I + \sum_{k=1}^K N^{(k)}(t,\xi) ,\qquad
   F_{K}(t,\xi) = \sum_{k=1}^{K} F^{(k)}(t,\xi),
\end{equation}
the estimate
\begin{multline}
    B_K(t,\xi) = \big(\D_t -\mathcal D(t) - R_1(t,\xi)\big) N_K(t,\xi)
    \\-N_K(t,\xi)  \big(\D_t -\mathcal D(t) - F_K(t,\xi)\big) \in{\mathcal P}\{K+1\}
\end{multline}
is valid. 

We just did this for $K=1$, it remains to do the recursion $k\mapsto k+1$. Assume
$B_k(t,\xi)\in\mathcal P\{k+1\}$. The requirement to be met is that
\begin{equation}
    B_{k+1}(t,\xi) - B_k(t,\xi) =  -[\mathcal D(t), N^{(k+1)}(t,\xi)] + F^{(k+1)}(t,\xi)  \mod \mathcal P\{k+2\} 
\end{equation}
for diagonal $F^{(k+1)}(t,\xi)$, which yields $F^{(k+1)}(t,\xi) = -\diag B_k(t,\xi)$ together with
\begin{equation}
    \left( N^{(k+1)}(t,\xi) \right)_{i,j} = - \frac{\big(B_k(t,\xi)\big)_{i,j}}{\delta_i(t)-\delta_j(t)} .
\end{equation}
Again, the diagonal terms can be set to zero, $(N^{(k+1)}(t,\xi))_{i,i}=0$. 
It is evident that the construction implies $F^{(k+1)}(t,\xi),N^{(k+1)}(t,\xi)\in\mathcal P\{k+1\}$
together with $B_{k+1}(t,\xi) \in \mathcal P\{k+2\}$.

The matrices $N_k(t,\xi)\in\mathcal P\{0\}$ are invertible with inverse $N^{-1}_k(t,\xi)\in\mathcal P\{0\}$ if we restrict our consideration to a sufficiently small elliptic zone $\mathcal Z_{\rm ell}(c_k)$.
The result of the above consideration can be summarised in the following lemma.

\begin{lem} Assume (B1) and (B2) and let $k\in\mathbb N$, $k\ge 1$. 
Then there exists a constant $c_k$ and matrices $N_k(t,\xi)\in\mathcal P\{0\}$, diagonal $F_{k}(t,\xi)\in\mathcal P\{1\}$ and $R_{k+1}(t,\xi)\in\mathcal P\{k+1\}$ such that
\begin{equation}
   \big(\D_t- \mathcal D(t) - R_1(t,\xi)\big) N_k(t,\xi) = N_k(t,\xi)\big(\D_t -\mathcal D(t)-F_k(t,\xi)-R_{k+1}(t,\xi)\big) 
\end{equation} 
holds true within $\mathcal Z_{\rm ell}(c_k)$. The matrix $N_k(t,\xi)$ is uniformly invertible
within this zone.  
\end{lem}

It is worth having a closer look at the upper-left corner entry of the matrix $F_k(t,\xi)$ and consequences for them based on Assumption (B3).
The above lemma implies that modulo $\mathcal P\{3\}$ the entry is of the form
\begin{equation}\label{eq:4:f1k}
  f_1^{(k)}(t,\xi) = \mathrm i \sum_{i,j=1}^d \alpha_{i,j}(t) \xi_i\xi_j + \sum\nolimits_{i=1}^d \beta_i(t) \xi_i +  \gamma(t)
   \mod \mathcal P\{3\}
\end{equation}
with $\alpha_{i,j}(t)\in \mathcal T\{0\}$, $\beta_i(t)\in\mathcal T\{0\}$ and $\gamma(t)\in\mathcal T\{1\}$. On the other hand, modulo $\mathcal O(t^{-1})$
the eigenvalues of $F_k(t,\xi)$ and of $A(t,\xi)$ coincide. Since (B3) is a spectral condition implying that $0$ is a local (quadratic) minimum of an eigenvalue branch contained in the complex upper half-plane, 
some terms in \eqref{eq:4:f1k} have to vanish. In particular we see that $\beta_i(t)$ has to decay, $\beta_i(t)\in\mathcal T\{1\}$, and also that the real part of the quadratic matrix $(\alpha_{i,j}(t))_{i,j}$ is {positive definite} modulo $\mathcal T\{1\}$. The latter is a direct consequence of the non-degeneracy of that minimum.

\begin{cor}\label{cor:sysDP:par-terms} Assume (B1), (B2) and B(3) and let $k\ge 1$. Then modulo $\mathcal P\{3\}$ the upper-left  corner entry of $F_k(t,\xi)$ satisfies
\begin{equation}
  f_1^{(k)}(t,\xi) = \mathrm i \xi^\top {\boldsymbol\alpha(t)} \xi + \boldsymbol\beta(t)^\top \xi + \gamma(t) 
   \mod \mathcal P\{3\}
\end{equation}
with $\boldsymbol\alpha(t)\in\mathcal T\{0\}\otimes\mathbb C^{n\times n}$ having positive definite real part uniform in $t\ge t_0$ for $t_0$ sufficiently large, $\boldsymbol\beta(t)\in\mathcal T\{1\}\otimes \mathbb C^n$ and $\gamma(t)\in\mathcal T\{1\}$. 
\end{cor} 

Later on we will see that the essential information for deducing diffusion phenomena and related asymptotic properties is contained in the terms
described in this corollary. Following methods of \cite{Wirth:2009} it is possible to relax assumption (B2) to the weaker requirement that $0$ is simple eigenvalue of 
$B(t)$ for all $t$ and uniformly separated from the remaining part of the spectrum of $B(t)$. Then a block-diagonalisation scheme can be established which
separates the corresponding mode and proves analogues of the above statements. 

\subsection{Asymptotic integration and small frequency expansions}
We consider the transformed problem in $V^{(k)}(t,\xi) = N_k(t,\xi) V^{(0)}(t,\xi)$, 
\begin{equation}\label{eq:sysDP:eq:CPk}
  \mathrm D_t V^{(k)}(t,\xi) = \big( \mathcal D(t) + F_k(t,\xi) + R_{k+1} (t,\xi) \big) V^{(k)}(t,\xi), 
\end{equation}
and reformulate this as integral equation for its fundamental solution. Denoting it by
$\mathcal E_k(t,s,\xi)$, we know that it solves the above equation to matrix initial data $\mathcal E_k(s,s,\xi) = \mathrm I\in\mathbb C^{d\times d}$. 

Let $\Theta_k(t,s,\xi)$ be the fundamental solution to the diagonal system $\mathrm D_t - \mathcal D(t)-F_k(t,\xi)$. Then 
\begin{equation}
   \Theta_k(t,s,\xi) = \begin{pmatrix} \Xi_k(t,s,\xi) & 0 \\ 0 & \tilde\Theta_k(t,s,\xi) \end{pmatrix},
\qquad    \| \tilde\Theta_k(t,s,\xi) \| \lesssim \mathrm e^{-\tilde c ({t-s})} 
\end{equation}
holds true for $t\ge s$ uniformly within $\mathcal Z_{\rm ell}(\epsilon)$ for $\epsilon\le c_k$ sufficiently small.
Here, 
\begin{equation}
 \Xi_k(t,s,\xi) = \exp\left(\mathrm i \int_s^t f_1^{(k)}(\theta,\xi) \mathrm d\theta\right)
\end{equation}
gives (for $k=2$) the fundamental solution to a parabolic problem
and $\tilde \Theta_k(t,s,\xi)$ is exponentially decaying as the fundamental solution of a dissipative system.
Furthermore, the matrix $\mathcal E_k(t,s,\xi)$ satisfies the Volterra integral equation
\begin{equation}
   \mathcal E_k(t,s,\xi)  = \Theta_k(t,s,\xi) + \int_s^t \Theta_k(t,\theta,\xi) R_{k+1}(\theta,\xi) 
   \mathcal E_k(\theta,s,\xi) \mathrm d\theta.
\end{equation}
We solve this equation using the Neumann series
\begin{multline}\label{sysDP:eq:NeumannS}
  \mathcal E_k(t,s,\xi) = \Theta_k(t,s,\xi) + \sum_{\ell=1}^\infty \mathrm i^\ell
  \int_s^t \Theta_k(t,t_1,\xi) R_{k+1}(t_1,\xi)\int_s^{t_1} \cdots\\\cdots \int_s^{t_{\ell-1}} 
   \Theta_k(t_{\ell-1},t_\ell,\xi) R_{k+1}(t_\ell,\xi)\mathrm d t_\ell \cdots \mathrm d t_1.
\end{multline}
This series converges and can be estimated by 
\begin{equation}
    \| \mathcal E_k(t,s,\xi) \| \le \exp\left(\int_s^t \| R_{k+1}(\theta,\xi) \|\mathrm d\theta\right).
\end{equation}
Based on the estimates for the remainder term 
$R_{k+1}(t,\xi) \in\mathcal P\{k+1\}$, we even obtain uniform convergence 
within the smaller zone $\mathcal Z_{\rm ell}(c_k)\cap \{ t|\xi|^{(k+1)/2} \le \delta \}$
for any constant $\delta$  and $\| \mathcal E_k(t,s,\xi) - \Theta_k(t,s,\xi) \| \to 0$ as
$c_k\to0$ for fixed $\delta>0$ as soon as we choose $k\ge 1$. 

We can obtain a slightly better estimate based on the uniform invertibility of $\Xi_k(t,s,\xi)$
for bounded $t|\xi|^2$ in consequence of Corollary~\ref{cor:sysDP:par-terms}.

\begin{lem}\label{lem:sysDP:sol_est-1}
   Let $k\ge 2$ and $\delta>0$. Then the fundamental solution $\mathcal E_k(t,s,\xi)$ 
   satisfies the uniform bound
   \begin{equation}
      \| \mathcal E_k(t,s,\xi) \| \le C_k  | \Xi_k(t,s,\xi) |
      ,\qquad  t\ge s\ge t_0, \quad t|\xi|^2\le \delta,
   \end{equation}
   for some constant $C_k>0$ depending on $t_0$, $\delta$ and $k$.   
\end{lem}
\begin{proof}
 Multiplying \eqref{sysDP:eq:NeumannS} by $\Xi_k^{-1}(t,s,\xi)$ in combination with the uniform
 bounds on $\Theta_k(t,s,\xi)$ yields directly
 \begin{align*}
  \| \mathcal E_k(t,s,\xi) \|& \le | \Xi_k(t,s,\xi) | \left\| \Xi_k^{-1}(t,s,\xi)\Theta_k(t,s,\xi) 
 +\cdots \right\|\\&
  \le   | \Xi_k(t,s,\xi) | \exp\left(C' \int_s^t \|R_{k+1}(\theta,\xi)\|\mathrm d\theta \right),
\end{align*}
whenever $t|\xi|^2\le\delta$.  Furthermore, for $k\ge 2$ the remaining integral is uniformly 
bounded on this set.
\end{proof}

\subsection{Lyapunov functionals and parabolic type estimates}\label{sec:4.3}
In this section we will partly follow the considerations of Beauchard--Zuazua, \cite{BZ:2011}, and explain how condition (B3) of Section~\ref{sec:4.2} allows to derive parabolic type decay estimates for solutions to the Cauchy problem \eqref{eq:4:CP}. The construction in \cite{BZ:2011} was inspired by the Lyapunov functionals used by Villani \cite{Villani:2011}.

\begin{lem}\label{eq:sysDP:KalmEst}
Assume (B1), (B2), (B3). Then all solutions to \eqref{eq:4:CP} satisfy the point-wise estimate
\begin{equation}
   \| \widehat U(t,\xi) \|^2 \le C \mathrm e^{-\gamma t [\xi]^2 } \|\widehat U_0(\xi)\|^2,\qquad [\xi] = |\xi|/\langle \xi\rangle \simeq \min\{|\xi|,1\},
\end{equation}
in Fourier space with constants $C$ and $\gamma$ depending only on the coefficient matrices $A_k(t)$ and $B(t)$.
\end{lem}
\begin{proof}[Sketch of proof]
The proof follows essentially \cite[Sec. 2.2]{BZ:2011}, we will only explain the major steps and necessary modifications to incorporate the time-dependence of 
matrices. As the problem is $L^2$-well-posed and dissipative, it suffices to prove the statement only for $t\ge t_0$ for sufficiently large $t_0$.

For a still to be specified selection $\epsilon=(\epsilon_0,\ldots, \epsilon_{d-1})$ of positive reals, $\epsilon_j>0$, we consider the Lyapunov functional 
\begin{multline}
   \mathbb L_\epsilon[\widehat U;t,\xi] = \| \widehat U(t,\xi) \|^2 +\\ \min\{|\xi|,|\xi|^{-1}\} \, \sum_{j=1}^{d-1} \epsilon_j \Im \langle B(t) A(t,\frac{\xi}{|\xi|}) ^{j-1} \widehat U(t,\xi) , B(t) A(t,\frac{\xi}{|\xi|})^{j} \widehat U(t,\xi) \rangle.
\end{multline}
The uniform Kalman rank condition (B3) implies that for suitable choices of the parameters $\epsilon$ and for $t\ge t_0$ the two-sided estimate
\begin{equation}
  \frac14 \|\widehat U(t,\xi)\|^2  \le  \mathbb L_\epsilon[\widehat U;t,\xi] \le 4 \|\widehat U(t,\xi\|^2
\end{equation}
holds true. Therefore, all we have to do is to prove the desired estimate for $ \mathbb L_\epsilon[\widehat U;t,\xi] $ which follows from
\begin{equation}
   \partial_t  \mathbb L_\epsilon[\widehat U;t,\xi] + \gamma [\xi]^2  \mathbb L_\epsilon[\widehat U;t,\xi] \le 0
\end{equation}
for suitable $\gamma$ and suitably chosen family $\epsilon$.

Formally differentiating $\mathbb L_\epsilon[\widehat U;t,\xi]$ with respect to $t$ yields the terms considered by \cite[Sec. 2.2]{BZ:2011} together with further
terms containing derivatives of the coefficient matrices. The latter ones are bounded by 
\begin{equation}
 \mathcal O(t^{-1} ) \min\{|\xi|,|\xi|^{-1}\}  \|\widehat U(t,\xi)\|^2,
\end{equation}
which is dominated by  $\gamma [\xi]^2  \mathbb L_\epsilon[\widehat U;t,\xi]$ whenever $t|\xi|\gtrsim1$ such that choosing $\gamma$ smaller yields the desired bound.
It remains to consider $t|\xi|\lesssim1$. Here $A(t,\xi)$ can be treated as small perturbation of $B(t)$ and the diagonalisation scheme and Lemma~\ref{lem:sysDP:sol_est-1} yields the corresponding bound.
\end{proof}

We will draw a consequence from this statement. It is obtained in combination with H\"older inequality and 
the boundedness properties of Fourier transform.

\begin{cor}
Assume (B1), (B2), (B3). Then all solutions to  \eqref{eq:4:CP} satisfy 
\begin{equation}
   \| U(t,\cdot)\|_{L^q} \le C (1+t)^{-\frac n2(\frac1p-\frac1q)} \| U_0\|_{L^p_r}
\end{equation}
for all $1\le p\le 2\le q\le \infty$ and with $r\ge n(1/p-1/q)$.
\end{cor}
\begin{proof}
For $|\xi|\gtrsim 1$ the previous lemma in combination with Sobolev embedding theorem yields exponential decay under the imposed regularity. Therefore, it is enough to
consider bounded $\xi$. Then the estimate is elementary,
\begin{align*}
   \| U(t,\cdot) \|_q \le \|\widehat U(t,\cdot)\|_{q'} \le \| \mathrm e^{-\gamma t|\cdot|^2} \|_{r} \|\widehat U_0\|_{p'} \le C (1+t)^{-\frac n2(\frac1p-\frac1q)} \|U_0\|_p 
\end{align*}
provided $\supp \widehat U_0(\xi) \subset \{ \xi : |\xi|\le 1\}$ and $\frac1 r=\frac1{q'}-\frac1{p'}=
\frac1p-\frac1q$.
\end{proof}

\subsection{A diffusion phenomenon for partially dissipative hyperbolic systems}
Now we will combine the estimates of the previous section with slightly improved results 
obtained from the low-frequency diagonalisation. First we construct a parabolic reference 
problem, whose fundamental solution is given by $\Xi_2(t,s,\xi)$ and afterwards we will explain why and in what sense solutions are asymptotically equivalent.

Following Corollary~\ref{cor:sysDP:par-terms} it is reasonable to consider the `parabolic' problem\footnote{Note, that $\Re\boldsymbol\alpha(t)> 0$ in the sense of self-adjoint matrices. We will understand parabolicity in this sense.}
\begin{equation}\label{sysDP:eq:par} 
   \partial_t  w = \nabla \cdot \boldsymbol \alpha(t) \nabla w + \boldsymbol\beta (t) \cdot\nabla w + 
 \mathrm i  \gamma(t) w,\qquad w(t_0)=w_0,
\end{equation}
for a scalar-valued unknown function $w_0$. To relate the both problems, we observe that
the first row of $\mathcal E_k(t,s,0)$ tends to a limit as $t\to\infty$.  This is just a consequence of the integrability of $R_{k+1}(t,0)$ for $k\ge2$. We use this to define
\begin{equation}
   W_k(s) = \lim_{t\to\infty} e_1^\top \mathcal E_k(t,s,0).
\end{equation}
It is easy to see that $W_k(s)=W_2(s)$ for all $k$.

\begin{lem}\label{sysDP:lem:solEst-2} 
The fundamental solution $\mathcal E_k(t,s,\xi)$, $k$ sufficiently large, satisfies the estimate
\begin{equation}
   \| \mathcal E_k(t,s,\xi) - \Xi_k(t,s,\xi) e_1 W_2(s) \| \le C_k (1+t)^{-\frac12},\qquad t\ge s\ge t_0,
\end{equation}
uniformly on $|\xi|\le 1$.
\end{lem}
\begin{proof}
We make use of a constant $\delta>0$, to be fixed later on, to decompose the extended phase space into several zones. 

\paragraph{1} If $t|\xi|^{2} \ge \delta \log t$ with $\delta$ chosen large enough, both terms 
can be estimated separately by $\exp(-\tilde\gamma t|\xi|^2)$ for some constant $\gamma$.
This follows for the first one by  Lemma~ \ref{eq:sysDP:KalmEst} and for the second
by the parabolicity of \eqref{sysDP:eq:par}  in consequence of 
Corollary~\ref{cor:sysDP:par-terms}.  But then
\begin{equation}
\mathrm e^{-\tilde\gamma t|\xi|^2} \le \mathrm e^{-\tilde \gamma\delta \log t} = t^{-\tilde \gamma\delta} \lesssim t^{-\frac12},\qquad \tilde\gamma\delta\ge\frac12.
\end{equation} 

\paragraph{2}
If  $t|\xi|^2\le \delta$ for some $\delta$, we use the results from the asymptotic integration of the diagonalised system. First, we claim that $\Xi_k^{-1}(t,t_0,\xi)e_1^\top\mathcal E_k(t,t_0,\xi)$
converges locally uniform in $\xi$ as $t\to t_\xi$ for $t_\xi |\xi|^2 = \delta$. To see this, we use
the Neumann series representation \eqref{sysDP:eq:NeumannS} multiplied by the uniformly bounded $\Xi_k^{-1}(t,s,\xi)$ combined with Cauchy criterion. We denote the resulting limit
as $W_k(s,\xi)$ and observe that it coincides with $W_k(s)$ for $\xi=0$,
\begin{equation}
   \| W_k(s,\xi) - W_k(s) \| \lesssim |\xi|,
\end{equation}
and satisfies
\begin{equation}
  \| e_1^\top \mathcal E_k(t,s,\xi) - \Xi_k(t,s,\xi) W_k(s,\xi) \| \lesssim t^{-1}.
\end{equation}
The first of these estimates follows as uniform limit for estimates of the difference
$\Xi_k^{-1}(t,s,0)e_1^\top\mathcal E_k(t,s,0)-\Xi_k^{-1}(t,s,\xi)e_1^\top\mathcal E_k(t,s,\xi)$. 
Indeed, using the Neumann series we see that the first terms are equal, the second terms
are reduced to the estimate $\|R_{k+1}(t,0)-R_{k+1}(t,\xi)\| \lesssim |\xi| (t^{-k} + |\xi|^k)$
following directly from the definition of the $\mathcal P\{k+1\}$-classes and, therefore, 
\begin{multline}
   \|\Xi_k^{-1}(t,s,0)e_1^\top\mathcal E_k(t,s,0)-\Xi_k^{-1}(t,s,\xi)e_1^\top\mathcal E_k(t,s,\xi)\|\\
   \lesssim |\xi| \int_s^t (\theta^{-k} + |\xi|^k) \mathrm d\theta  + 
   ...
\end{multline}
and the right-hand side is uniformly bounded by $|\xi|$. Taking limits proves the estimate.
The second estimate is similar. Again using the Neumann series we see that this difference
can be estimated by
\begin{multline}
   \| \Xi_k^{-1}(t,s,\xi)  e_1^\top \mathcal E_k(t,s,\xi) - W_k(s,\xi) \| 
 \\  \le \int_t^{t_\xi} \| R_{k+1}(\tau,\xi)\| \exp\left(\int_{t_0}^{t_\xi} \|R_{k+1}(\theta,\xi) \|\mathrm d\theta
   \right)\mathrm d\tau \lesssim t^{-\frac12} 
\end{multline}
due to $\|R_{k+1}(t,\xi)\|=\mathcal O(t^{-\frac32})$ for $k\ge 2$ and $t|\xi|^2\le \delta$.

Combining both of the above estimates and using that the other rows in $\mathcal E_k$ are
exponentially decaying we get
\begin{multline}
    \|  \mathcal E_k(t,s,\xi) - \Xi_k(t,s,\xi) e_1 W_k(s) \|
      \le   \|  \mathcal E_k(t,s,\xi) - \Xi_k(t,s,\xi) e_1 W_k(s,\xi) \| \\+ \|  \Xi_k(t,s,\xi) e_1 W_k(s,\xi)- \Xi_k(t,s,\xi) e_1 W_k(s)\| \\ \lesssim t^{-\frac12} + \mathrm e^{-\tilde \gamma t |\xi|^2}|\xi| 
     \lesssim t^{-\frac12}.
\end{multline}

\paragraph{3} 
It remains to consider the logarithmic gap between both parts, i.e., $\delta \le t|\xi|^2 \le \delta\log t$. Here we use that for $k$ sufficiently large the remainder term $R_{k+1}(t,\xi)$ decays as
$t^{-k-1+\epsilon}$, while the polynomial growth rate of $\Xi^{-1}_k(t,\xi)$ is independent of 
$k$ for large $k$. Choosing $k$ large enough, the Neumann series argument gives
\begin{multline}
  \| \Xi_k^{-1}(t,s,\xi) e_1^\top \mathcal E_k(t,s,\xi) - \tilde W_k(s,\xi) \| 
  \lesssim \int_t^{\tilde t_\xi} \| \Xi_k^{-1}(\theta,s,\xi) R_{k+1}(\theta,\xi) \| \mathrm d\theta  
  \lesssim t^{-\frac12}
\end{multline}
with $\tilde t_\xi$ defined by $\tilde t_\xi |\xi|^2 = \delta \log \tilde t_\xi$ and 
$\tilde W_k(s,\xi) = \lim_{t\to\tilde t_\xi} \Xi_k^{-1} (t,s,\xi) e_1^\top \mathcal E_k(t,s,\xi)$. The
existence of the latter limit follows for large $k$ and again $\tilde W_k(s,\xi) - W_k(s)$
coincides up to order $\mathcal O(|\xi|)$.
\end{proof}

To obtain a statement in terms of the original equation, we introduce
$K(t,\xi) = M(t)N_2(t,\xi) e_1$. By definition we have 
$K(t,\xi) \in \mathcal P\{2\}$. We define further $w_0 = W U_0$ in such a way that we cancel 
the main term of the solution within $\mathcal Z_{\rm ell}(c_k)\cap \{ t|\xi| \le \delta \}$, i.e., 
we define
\begin{equation}\label{eq:4:init-data-DP}
   \widehat w_0 = W_2(t_0) N_2^{-1}(t_0,\xi) M^{-1}(t_0) \mathcal E(t_0,0,\xi) \chi(\xi) \widehat U_0
\end{equation}
with $\chi(\xi)\in C^\infty_0(\mathbb R^n)$, $\chi(\xi)=1$ near $\xi=0$ and $\mathrm{supp}\,\chi\subset B_{c_2}(0)$. Then the estimate of Lemma~\ref{sysDP:lem:solEst-2} implies the following statement. The logarithmic term is caused by comparing $\Xi_k(t,s,\xi)$ with $\Xi_2(t,s,\xi)$.

\begin{cor} Let $U(t,x)$ be solution to \eqref{eq:4:CP}. The the solution $w(t,x)$ to \eqref{sysDP:eq:par} with data given by \eqref{eq:4:init-data-DP} satisfies
\begin{equation}
\| U(t,\cdot) -  K(t,\D) 
w(t,\cdot) \|_2   \le 
    C' (1+t)^{-\frac12} \log(\mathrm e+ t).
\end{equation}
\end{cor}

As only the upper-left corner of the diagonalised problem is of interest, Assumption (B2) can be relaxed and it is sufficient to guarantee block-diagonalis\-ability. We will not go into detail here, but remark that it is sufficient to have that the eigenvalue $0$ is uniformly separated from the remaining spectrum of the family $B(t)$. Then the diagonalisation scheme can be modified based on the method of \cite{Wirth:2009} and an analogous result obtained.

%

\section{Examples and counter-examples}
Both, in Section~\ref{sec:3} and Section~ \ref{sec:4} we made symbol like assumptions on coefficients, e.g., we considered hyperbolic systems
\begin{equation}
    \D_t U = \sum_{k=1}^n A_k(t) \D_{x_k} U  
\end{equation}
for coefficient matrices $A_k(t)\in\mathcal T\{0\}$, meaning that derivatives of the coefficients
are controlled by 
\begin{equation}
  \|  \D_t^\ell A_k(t) \|  \le C_\ell \left(\frac1{1+t}\right)^{\ell}.
\end{equation}
We will use this section to show that assumptions controlling the amount of oscillations in the
time-behaviour are in fact necessary in order to control the large-time behaviour of the energy.
Yagdjian pointed out in \cite{Yagdjian:2001} the deteriorating effect time-periodic propagation 
speeds might have on energy estimates and more generally on the global existence of 
small data solutions for nonlinear wave models. The key idea behind these is based on 
Floquet's theory for periodic ordinary differential equations and in particular Borg's theorem.

We will recall these first in connection with a very simple model and then give an idea how these 
can be used to prove sharpness of estimates and sharpness of assumptions on coefficients for wave models.

\subsection{Parametric resonance phenomena}
We will restrict ourselves to the simple model
\begin{equation}\label{eq:5:CP1}
   u_{tt} - a^2(t)\Delta u = 0,\qquad u(0,\cdot)=u_0,\quad u_t(0,\cdot) = u_1,
\end{equation}
of a wave equation with variable propagation speed. Basic assumption will be that
$a^2(t)$ is positive, smooth, periodic
\begin{equation}
  a(t+1) = a(t),
\end{equation}
and non-constant. Using a partial Fourier transform with respect to the $x$-variable, this is 
equivalent to Hill's equation
\begin{equation}\label{eq:Hill}
   \widehat u_{tt} + |\xi|^2 a^2(t) \widehat u = 0
\end{equation}
with spectral parameter $|\xi|^2$ or, equivalently, the first order system
\begin{equation}\label{eq:5:CP1-sys}
  \D_t \widehat U  = \begin{pmatrix} 0 & |\xi| \\ a^2(t) |\xi|& 0 \end{pmatrix} \widehat U
\end{equation}
for $\widehat U=(|\xi|\widehat u, \D_t u)^\top$. Its fundamental matrix is again denoted as
$\mathcal E(t,s,\xi)$. Due to periodicity, it is of interest to consider the monodromy matrix
$\mathcal M(\xi)=\mathcal E(1,0,\xi)$. Since 
the trace of the matrix in \eqref{eq:5:CP1-sys} is zero and hence
$\det \mathcal M(\xi) = 1$, it has either two non-zero eigenvalues of the form $\exp(\pm \kappa(\xi))$ or is a Jordan matrix to the eigenvalue $1$. The 
number $\kappa(\xi)$ is called the Floquet exponent of \eqref{eq:Hill}. Its importance stems 
from the following lemma. It can be found in a similar form in \cite{WW:1902}.

\begin{lem}[Floquet-Lemma]\label{lem:5:Floquet}
Assume that $\mathcal M(\xi)$ is diagonalisable with Floquet exponent $\kappa(\xi)$. Then
\eqref{eq:Hill} has a fundamental system of solutions of the form
\begin{equation}\label{eq:5:Floquet-sol}
    \mathrm e^{\pm t \kappa(\xi)} f_\pm(t,\xi) 
\end{equation}
for non-vanishing $1$-periodic functions $f_\pm(t,\xi)$ depending analytically on $\xi$.
\end{lem}
\begin{proof} We will give a proof based on the system \eqref{eq:5:CP1-sys}. The same 
method applies to arbitrary periodic systems of differential equations.
By assumption we find an invertible matrix $N(\xi)$ 
such that
\begin{equation}\label{eq:5:5.8}
  \mathcal M(\xi) N (\xi) = N (\xi) \diag(\mathrm e^{\kappa(\xi)},-\mathrm e^{\kappa(\xi)})
  =  N(\xi)\exp\left(  K(\xi)\right)
\end{equation}
holds true with $K(\xi) = \diag(\kappa(\xi),-\kappa(\xi))$. We use the fundamental matrix to define
the function
\begin{equation}
  F(t,\xi)=  \mathcal E(t,0,\xi) N(\xi) \exp\left(-t K(\xi)\right) N^{-1}(\xi).
\end{equation}
Then a simple calculation shows that the periodicity implies 
$\mathcal E(t+1,1,\xi) =\mathcal E(t,0,\xi)$ and, therefore,  we obtain 
\begin{equation}
  F(t+1,\xi) = \mathcal E(t,0,\xi)\mathcal M(\xi) N(\xi) \mathrm e^{-K(\xi)} \mathrm e^{-t K(\xi)} N^{-1}(\xi) =F(t,\xi)
\end{equation}
based on \eqref{eq:5:5.8}.
Furthermore, by construction we see that any solution of the system \eqref{eq:5:CP1-sys} is of the form
\begin{equation}
  \widehat U(t,\xi) = \mathcal E(t,0,\xi) \widehat U_0 =  F(t,\xi) \exp\left(-t K(\xi)\right)  \widehat U_0.
\end{equation}
The statement of Lemma~\ref{lem:5:Floquet} follows by looking at individual entries of this matrix
in combination with the definition of $\widehat U$.
\end{proof}

Since $a(t)$ was assumed to be real, the above representation implies a symmetry for
the Floquet exponents. Either, they are both imaginary or they are both real. If they are imaginary,
all solutions to \eqref{eq:Hill} remain bounded, while for real Floquet exponents an exponential
dichotomy appears. Looking at the original equation \eqref{eq:5:CP1} this means that for
appropriately chosen initial data, solutions will have an exponentially increasing energy. 

For us the following statement is of interest. In the present form it is due to Colombini--Spagnolo
\cite{ColSpan:1984}, for more detailed results on Hill's equation see Magnus--Winkler 
\cite{MaWi:1966}.

\begin{lem}[Borg's theorem]\label{lem:5:borg}
Assume the coefficient $a(t)$ is locally integrable, $1$-periodic and non-constant. 
Then there exists an open interval $\mathcal J\subset \mathbb R_+$, such that for
all $|\xi|\in \mathcal J$ the Floquet exponent satisfies $\kappa(\xi)>0$.
\end{lem}

In combination, the above lemmata imply that we can find initial data such that the solution
is exponentially increasing. Note, that no assumption on the size of the coefficient was made, we only used that $a(t)$ is not constant.

\begin{cor}
There exist inital data $u_0,u_1\in\mathscr S(\mathbb R^n)$ such that the solution
$u(t,x)$ to \eqref{eq:5:CP1} satisfies
\begin{equation}
    \liminf_{t\to\infty} \frac{\log \mathbb E(u;t)}t > 0.
\end{equation}
\end{cor}

\subsection{Construction of coefficients and initial data}
We stay with the model \eqref{eq:5:CP1}, 
\begin{equation}\label{eq:5:CP1-a}
   u_{tt} - a^2(t)\Delta u = 0,\qquad u(0,\cdot)=u_0,\quad u_t(0,\cdot) = u_1,
\end{equation}
but drop the periodicity assumption.  The above sketched instability mechanism will be used to construct coefficient functions and data
for prescribed energetic behaviour.

To be precise, let $\tau_k\nearrow\infty$, $\delta_k$, $\eta_k$ and $n_k$ be 
sequences such that
\begin{align}
    \tau_{k+1} > \tau_k + \delta_k,\qquad \eta_k\le 1, \qquad n_k\in\mathbb N_{>0}.
\end{align}
The sequences $\tau_k$ and $\delta_k$ define disjoint intervals $\mathcal I_k=[\tau_k,\tau_k+\delta_k]$.
We will let our coefficient oscillate at least $n_k$ times on the interval $\mathcal I_k$. 
Let therefore $\phi\in C_0^\infty(\mathbb R)$ be supported in the interval $(0,1)$
and non-vanishing with $|\phi(t)|<1$ and denote by $b(t)$ its periodisation,
$b(t) = \phi(t-\lfloor t\rfloor)$.  Then we consider
\begin{equation}
  a(t) = \begin{cases} 
  1+  \eta_k b\big(\frac{n_k}{\delta_k} (t-\tau_k)\big) , \qquad &t\in \mathcal I_k ,\quad k=1,2,\ldots\\
  1,\qquad & \text{otherwise}.
  \end{cases}
\end{equation}
The function $a(t)$ is bounded,
\begin{equation}
   0< a(t) < 2,
\end{equation}
and satisfies the estimate
\begin{equation}\label{eq:5:a-est}
  | \D_t^\ell a(t) | \le C_\ell \eta_k \big(\frac{n_k}{\delta_k}\big)^\ell  ,\qquad t\in\mathcal I_k,
\end{equation}
the constants $C_\ell$ are independent of $k$.

The freedom in the choice of sequences can be used to construct examples of coefficient functions in order to prove sharpness of energy
estimates. For instance, we obtain $a(t)\in\mathcal T\{0\}$
for the choice $\tau_k \sim k$, $\delta_k \sim k/2$, $\eta_k=1$, and $n_k=1$ while 
increasing $n_k$ to $n_k=\lceil k^{1-\alpha}\rceil$, $\alpha\in(0,1)$, yields the weaker estimate 
\begin{equation}
   |\D_t^\ell a(t)|\le C_\ell \left(\frac1{1+t}\right)^{\alpha\ell}.
\end{equation}

\begin{figure}
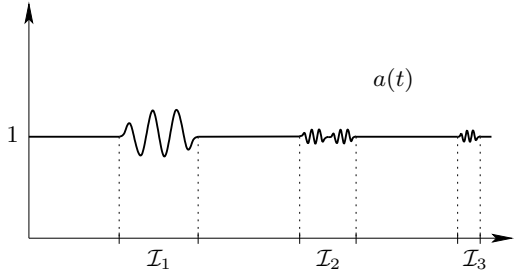
\caption{Sketch of the constructed coefficient $a(t)$.}
\end{figure}

If we want to show lower bounds on the energy or the sharpness of energy estimates, we have to construct initial data. One particularly simple idea is to use a sequence of initial data
$u_{0,k}, u_{1,k}\in\mathscr S(\mathbb R^n)$ which leads to an exponential increase
of the energy within the interval $\mathcal I_k$ and still satisfies good estimates in the earlier
intervals $\mathcal I_{k'}$, $k'<k$. The latter can be achieved by adjusting the sequences
in the definition of the coefficient function. 
This idea was employed in \cite{Hirosawa:2007}, \cite{HW:2009} to discuss the sharpness of the results.

The following statement is entirely in terms of the defining sequences and provides such a lower bound on the energy behaviour of solutions to \eqref{eq:5:CP1-a}. We use $\eta_k=1$ for simplicity in order to reduce the statement to Borg's theorem in the form of Lemma~\ref{lem:5:borg}.

\begin{lem}
Assume $a(t)$ is defined in terms of the above sequences $\tau_k$, $\delta_k$
and $n_k$. Then there exists a sequences of initial data
$u_{0,k}$, $u_{1,k}$ with normalised initial energy $\mathbb E(u_k;0)= 1$ such that the solution to \eqref{eq:5:CP1-a} satisfies
\begin{equation}
   \log \mathbb E(u_k, \tau_{k}+\delta_k)   \ge 2\kappa n_k - 2 c \sum_{\ell=1}^{k-1} n_\ell 
\end{equation}
with $\kappa>0$ small and $c= \sup_s \frac{|b'(s)|}{b(s)}>0$.
\end{lem}
\begin{proof} We first outline the main strategy.
We choose the sequence $u_k$ in such a way that $\widehat u_k(\tau_k,\xi)$ restricted to
the interval $\mathcal I_k$ is an exponentially increasing (eigen-) solution to the periodic
problem on that interval. To relate different $k$, we note that all appearing periodic problems are rescaled versions of each other and we use the same instability interval $\mathcal J$ for all of them.
We follow these solutions backward to the initial line and adjust them in such a way that they are
normalised in initial energy. 

To be precise, let $\mathcal J$ be an interval of instability of the periodic problem
\begin{equation}
    \widehat v_{tt} + |\xi|^2 (1+b(t))^2 \widehat v  = 0.
\end{equation}
Let further $\widehat v(t,\xi)$ be an exponentially increasing eigensolution of the form \eqref{eq:5:Floquet-sol}  supported  
inside the interval $\mathcal J$ such that on $\supp\widehat v$ the Floquet exponent satisfies
$\kappa(\xi) \ge \kappa >0$ for some positive constant $\kappa$. We can rescale $\widehat v$
to 
\begin{equation} \label{eq:5:5.20}
 \widehat u_k(t,\xi) = \mu_k   \widehat v( \frac{n_k}{\delta_k}  (t-\tau_k) , \frac{\delta_k}{n_k} \xi ), \qquad \eta_k\in\mathbb R_+,
\end{equation}
such that $u_k(t,x)$ solves \eqref{eq:5:CP1-a} on $\mathcal I_k$. We will use $\mu_k$ to normalise solutions. Equation \eqref{eq:5:5.20} implies
\begin{equation}\label{eq:5:5.22}
   \mathbb E(u_k, \tau_k+\delta_k) = \mathrm e^{2n_k\kappa} \mathbb E(u_k,\tau_k).
\end{equation}
Outside of $\bigcup_\ell \mathcal I_\ell$ the energy is conserved, while it might decrease or increase on the intervals $\mathcal I_\ell$. Therefore we obtain the lower bound
by looking at the worst type behaviour as estimated by Gronwall inequality from 
\begin{multline}
\left|    \partial_t \int \big( |u_t|^2 + a^2(t) |\nabla u|^2 \big)\mathrm d x \right| = 2\frac{|a'(t)|}{a(t)} \int a^2(t) |\nabla u|^2 \mathrm d x  \\ \le 2 \frac{|a'(t)|}{a(t)}  \int \big( |u_t|^2 + a^2(t) |\nabla u|^2 \big)\mathrm d x,
\end{multline}
i.e., $\mathbb E(u_k, \tau_\ell+\delta_\ell) \ge \mathrm e^{-2 \delta_\ell c_\ell}  \mathbb E(u_k, \tau_\ell)$
with $c_\ell = \sup_{t\in\mathcal I_\ell}\frac{|a'(t)|}{a(t)} = \frac{n_\ell}{\delta_\ell} \sup_s \frac{ |b'(s)|}{b(s)}$. In combination this yields
\begin{equation}\label{eq:5:5.24a}
    \log \mathbb E(u_k,\tau_k) - \log \mathbb E(u_k,0) \ge - 2  c \sum_{\ell=1}^{k-1} n_\ell
\end{equation}
and the statement follows by combining \eqref{eq:5:5.24a} with \eqref{eq:5:5.22}.
\end{proof}

When using this construction one has to pay attention to some particular facts.
First, the bad behaviour of solutions is localised to $|\xi|\approx n_k / \delta_k\to 0$ on $\mathcal I_k = [\tau_k,\tau_k+\delta_k]\to \infty$. In view of \eqref{eq:5:a-est} this is related to the decomposition of the phase space into zones and the bad increase in energy happens close to the boundary of the hyperbolic zone. Second, the choice of sequences does matter. Even if it can not be seen in the statement of the previous lemma, one wants to construct coefficient functions
violating conditions as closely as possible in order to draw interesting conclusions.

We draw one consequence from the above lemma. When considering \eqref{eq:5:CP1-a} with
$a(t)\in\mathcal T\{0\}$, then solutions satisfy the a global generalised energy conservation property in the sense that
\begin{equation}\label{eq:5:5.24}
    \mathbb E(u;t) \approx \mathbb E(u;0)
\end{equation} 
uniform in $t$ and with constants depending only on the coefficient function $a(t)$. This follows from Theorem~\ref{thm:3:3.6} and was originally proven in Reissig--Smith \cite{RS:2005}. If we use the sequences $\tau_k=\sigma^k$, $\delta_k=\sigma^{k-1}$ and $n_k=\lceil \sigma^{q k}\rceil$ for some $q>0$ and a given parameter $\sigma>0$, we obtain a coefficient $a(t)$ satisfying the estimate
\begin{equation}
   |\D_t^\ell a(t) | _{t\in\mathcal I_k} \le C_\ell \left( \frac{\lceil \sigma^{qk}\rceil}{\sigma^k} \right)^\ell\approx C_\ell \left( \frac{1}{1+t}\right)^{(1-q)\ell}.
\end{equation}
The solution to this problem can not satisfy \eqref{eq:5:5.24}, as
\begin{equation}
  \kappa \sigma^{qk} - c  \sum_{\ell=1}^{k-1}  \sigma^{q\ell}  = \kappa \sigma^{qk} - c\frac{\sigma^{qk}-1}{\sigma^q-1} \to \infty ,\qquad k\to\infty,
\end{equation}
provided $\sigma$ is chosen large enough in comparison to $\kappa$ and $c$.

A slight modification of the argument allows to show that the assumption $a(t)\in\mathcal T_\nu\{0\}$ for some $\nu>0$ is also not sufficient to deduce the estimate \eqref{eq:5:5.24}.

\section{Related topics}

Most of the result presented here were based on diagonalisation procedures in order to deduce
asymptotic information on the representations of solutions. This is natural and has a long history in the study of hyperbolic equations and coupled systems.
For diagonalisation schemes in broader sense
and their application we also refer to \cite{Jachmann:2010}. Some more applications
are discussed there too. 

Our main concern was the derivation of energy and dispersive type estimates describing the asymptotic behaviour of solutions to hyperbolic equations. As it is impossible to cover all
directions appearing there, we refer also to the expository article 
\cite{Wirth:2010} for a discussion on current 
results about energy type estimates for wave models with bounded coefficients. 
Wave models with unbounded coefficients (meaning polynomially increasing propagation speed  
or even exponentially increasing propagation speed) have been extensively studied by 
Reissig \cite{Reissig:1997b} and Reissig--Yagdjian, \cite{RY:2000},  \cite{Yagdjian:2000a}, \cite{RY:2000b}, \cite{RY:2000c}, see also Galstian \cite{Galstian:2003}.

There is an interesting duality to be observed here. Wave models with increasing coefficients and
their large-time behaviour are intimately connected to the well-posedness  issues of weakly 
hyperbolic equations. In a similar way, models with bounded coefficients are related to
well-posedness statements of equations with $\log$-Lipschitz coefficients and their microlocal
analysis. See, e.g.,  the results of Kubo--Reissig \cite{KR:2004} or Kinoshita--Reissig \cite{KR:2005} where the key ingredient of the consideration is a diagonalisation scheme for pseudo-differential versions of the symbol classes introduced in Section~\ref{sec:3:diag} applied locally in time.

The considerations of the generalised energy conservation property in Section~\ref{sec:3:SolvDiag} were inspired by the work of Hirosawa, \cite{Hirosawa:2007} and later
joint work with him, \cite{HW:2008}, \cite{HW:2009}.  Main question arising there is to what 
extent symbolic conditions on coefficients can be weakened without loosing 
uniform  bounds on the energy. As the presented results are sharp, weaker estimates for derivatives have to be compensated by additional stabilisation conditions.
There has been recent work in this direction by D'Abbicco--Reissig \cite{dAR:2011} 
for $2\times 2$ hyperbolic systems generalising results of \cite{HW:2009} and also ongoing joint research of the second author with Hirosawa.


\end{document}

%% file: Figure2.pdf_t
\begin{picture}(0,0)%
\includegraphics{Figure2.pdf}%
\end{picture}%
\setlength{\unitlength}{4144sp}%
\begingroup\makeatletter\ifx\SetFigFontNFSS\undefined%
\gdef\SetFigFontNFSS#1#2#3#4#5{%
  \reset@font\fontsize{#1}{#2pt}%
  \fontfamily{#3}\fontseries{#4}\fontshape{#5}%
  \selectfont}%
\fi\endgroup%
\begin{picture}(4077,3312)(436,-2731)
\put(451,-1771){\makebox(0,0)[lb]{\smash{{\SetFigFontNFSS{12}{14.4}{\familydefault}{\mddefault}{\updefault}{\color[rgb]{0,0,0}1}%
}}}}
\put(4231,-2671){\makebox(0,0)[lb]{\smash{{\SetFigFontNFSS{12}{14.4}{\familydefault}{\mddefault}{\updefault}{\color[rgb]{0,0,0}$\xi$}%
}}}}
\put(451,389){\makebox(0,0)[lb]{\smash{{\SetFigFontNFSS{12}{14.4}{\familydefault}{\mddefault}{\updefault}{\color[rgb]{0,0,0}t}%
}}}}
\put(721,-1546){\makebox(0,0)[lb]{\smash{{\SetFigFontNFSS{12}{14.4}{\familydefault}{\mddefault}{\updefault}{\color[rgb]{0,0,0}$\mathcal Z_{\rm pd}=\{ t|\xi|\lesssim  1\} $}%
}}}}
\put(2026,-646){\makebox(0,0)[lb]{\smash{{\SetFigFontNFSS{12}{14.4}{\familydefault}{\mddefault}{\updefault}{\color[rgb]{0,0,0}$\mathcal Z_{\rm hyp}=\{ t|\xi|\gtrsim 1\}$}%
}}}}
\end{picture}%

%% file: Figure3.pdf_t
\begin{picture}(0,0)%
\includegraphics{Figure3.pdf}%
\end{picture}%
\setlength{\unitlength}{4144sp}%
\begingroup\makeatletter\ifx\SetFigFontNFSS\undefined%
\gdef\SetFigFontNFSS#1#2#3#4#5{%
  \reset@font\fontsize{#1}{#2pt}%
  \fontfamily{#3}\fontseries{#4}\fontshape{#5}%
  \selectfont}%
\fi\endgroup%
\begin{picture}(4257,3312)(256,-2731)
\put(4231,-2671){\makebox(0,0)[lb]{\smash{{\SetFigFontNFSS{12}{14.4}{\familydefault}{\mddefault}{\updefault}{\color[rgb]{0,0,0}$\xi$}%
}}}}
\put(451,389){\makebox(0,0)[lb]{\smash{{\SetFigFontNFSS{12}{14.4}{\familydefault}{\mddefault}{\updefault}{\color[rgb]{0,0,0}t}%
}}}}
\put(1486,-2626){\makebox(0,0)[lb]{\smash{{\SetFigFontNFSS{12}{14.4}{\familydefault}{\mddefault}{\updefault}{\color[rgb]{0,0,0}$c$}%
}}}}
\put(271,-1321){\makebox(0,0)[lb]{\smash{{\SetFigFontNFSS{12}{14.4}{\familydefault}{\mddefault}{\updefault}{\color[rgb]{0,0,0}$c^{-1}$}%
}}}}
\put(1711,-61){\makebox(0,0)[lb]{\smash{{\SetFigFontNFSS{12}{14.4}{\familydefault}{\mddefault}{\updefault}{\color[rgb]{0,0,0}$\mathcal Z_{\rm ell}(c)$}%
}}}}
\end{picture}%

%% file: Figure4.pdf_t
\begin{picture}(0,0)%
\includegraphics{Figure4.pdf}%
\end{picture}%
\setlength{\unitlength}{3108sp}%
\begingroup\makeatletter\ifx\SetFigFontNFSS\undefined%
\gdef\SetFigFontNFSS#1#2#3#4#5{%
  \reset@font\fontsize{#1}{#2pt}%
  \fontfamily{#3}\fontseries{#4}\fontshape{#5}%
  \selectfont}%
\fi\endgroup%
\begin{picture}(4077,2191)(436,-1610)
\put(451,-556){\makebox(0,0)[lb]{\smash{{\SetFigFontNFSS{9}{10.8}{\familydefault}{\mddefault}{\updefault}{\color[rgb]{0,0,0}1}%
}}}}
\put(1576,-1546){\makebox(0,0)[lb]{\smash{{\SetFigFontNFSS{9}{10.8}{\familydefault}{\mddefault}{\updefault}{\color[rgb]{0,0,0}$\mathcal I_1$}%
}}}}
\put(2926,-1546){\makebox(0,0)[lb]{\smash{{\SetFigFontNFSS{9}{10.8}{\familydefault}{\mddefault}{\updefault}{\color[rgb]{0,0,0}$\mathcal I_2$}%
}}}}
\put(4096,-1546){\makebox(0,0)[lb]{\smash{{\SetFigFontNFSS{9}{10.8}{\familydefault}{\mddefault}{\updefault}{\color[rgb]{0,0,0}$\mathcal I_3$}%
}}}}
\put(3376,-106){\makebox(0,0)[lb]{\smash{{\SetFigFontNFSS{9}{10.8}{\familydefault}{\mddefault}{\updefault}{\color[rgb]{0,0,0}$a(t)$}%
}}}}
\end{picture}%